\documentclass[11pt,letterpaper]{amsart}
\usepackage[utf8x]{inputenc}
\usepackage[english]{babel}
\usepackage{amsthm} 
\usepackage{tikz-cd}
\usepackage{amsmath} 
\usepackage{amsfonts}
\usepackage{amssymb}
\usepackage{mathrsfs}
\usepackage[all]{xy}
\usepackage{mathtools}
\usepackage{url}

\makeatletter
\@namedef{subjclassname@2020}{\textup{2020} Mathematics Subject Classification}
\makeatother

\theoremstyle{plain}
\newtheorem{theorem}{Theorem}[section]
\newtheorem{theo}[theorem]{Theorem}

\newtheorem{lemma}[theorem]{Lemma}
\newtheorem{lemm}[theorem]{Lemma}

\newtheorem{proposition}[theorem]{Proposition}

\newtheorem{corollary}[theorem]{Corollary}
\newtheorem{coro}[theorem]{Corollary}
\newtheorem*{theorem*}{Theorem}

\newtheorem{prop}[theorem]{Proposition}

\theoremstyle{definition}
\newtheorem{definition}[theorem]{Definition}
\newtheorem{defi}[theorem]{Definition}

\newtheorem{exam}[theorem]{Example}
\newtheorem{notation}[theorem]{Notation}
\newtheorem{nota}[theorem]{Notation}

\theoremstyle{remark}
\newtheorem{rema}[theorem]{Remark}

\newtheorem{remark}[theorem]{Remark}

\let\leq\leqslant
\let\geq\geqslant
\let\phi\varphi
\newcommand{\ie}{{\it i.e.}\ }
\newcommand{\eg}{{\it e.g.}\ }

\DeclarePairedDelimiter{\floor}{\lfloor}{\rfloor}
\newcommand{\sing}{\text{sing}}
\newcommand{\sm}{\text{sm}}
\newcommand{\fat}{\text{fat}}

\newcommand{\arc}{\mathscr{L}_{\infty}}

\newcommand{\leqmds}{\leq_{\text{mds}}}
\newcommand{\geqmds}{\geq_{\text{mds}}}
\newcommand{\leqmdss}[1]{\leq_{\text{mds},#1}}
\newcommand{\leqmdsX}{\leq_{\text{mds},X}}

\newcommand{\leqmdsY}{\leq_{\text{mds},Y}}

\newcommand{\leqs}{\leq_{\sigma}}

\newcommand{\leqX}{\leq_{_{X}}}
\newcommand{\E}{\mathscr{E}}
\newcommand{\DF}{\mathscr{E}}
\newcommand{\leqY}{\leq_{_{Y}}}
\newcommand{\geqX}{\geq_{_{X}}}

\DeclareMathOperator{\val}{val}
\newcommand{\mds}{\cC} 
\DeclareMathOperator{\DV}{DV} 
 
\DeclareMathOperator{\Val}{Val} 
\DeclareMathOperator{\DivEss}{DivEss} 
\DeclareMathOperator{\Ess}{Ess} 
\DeclareMathOperator{\Nash}{Nash} 

\DeclareMathOperator{\St}{St}
\DeclareMathOperator{\cent}{cent} 
 
\DeclareMathOperator{\Exc}{Exc}
\DeclareMathOperator{\exc}{exc}
\DeclareMathOperator{\Adh}{Adh}

\DeclareMathOperator{\Relint}{Relint}
\DeclareMathOperator{\ord}{ord}

\DeclareMathOperator{\Hom}{Hom}
\DeclareMathOperator{\Tail}{Tail}
\DeclareMathOperator{\Loc}{Loc}

\DeclareMathOperator{\Pol}{Pol}

\DeclareMathOperator{\Div}{Div}
\DeclareMathOperator{\D}{\mathcal{D}}
\DeclareMathOperator{\Ray}{Ray}
\DeclareMathOperator{\Ver}{Ver}
\DeclareMathOperator{\Ter}{Ter}
\DeclareMathOperator{\Supp}{Supp}
\DeclareMathOperator{\Span}{Span}

\newcommand{\imply}{\Rightarrow}

\newcommand{\injects}{\hookrightarrow}

\newcommand{\isom}{\overset{\sim}{\to}}

\newcommand{\cA}{{\mathcal A}}\newcommand{\cC}{{\mathcal C}}\newcommand{\cD}{{\mathcal D}}

\newcommand{\cL}{{\mathcal L}}
\newcommand{\cM}{{\mathcal M}}\newcommand{\cN}{{\mathcal N}}\newcommand{\cO}{{\mathcal O}}\newcommand{\cP}{{\mathcal P}}
\newcommand{\cT}{{\mathcal T}}

\newcommand{\cZ}{{\mathcal Z}}

\newcommand{\scO}{\mathscr{O}}

\newcommand{\mfp}{\mathfrak{p}}
\newcommand{\mfq}{\mathfrak{q}}

\DeclareMathOperator{\Min}{Min}
\DeclareMathOperator{\MinVal}{MinVal}

\DeclareMathOperator{\Inf}{Inf}
\DeclareMathOperator{\Spec}{Spec}
\DeclareMathOperator{\bSpec}{\textbf{Spec}}

\newcommand{\sumu}[1]{\underset{#1}{\sum}}

\newcommand{\sumsubu}[1]{\sumu{\substack{#1}}}

\newcommand{\Infsubu}[1]{\underset{\substack{#1}}{\Inf}}

\newcommand{\ide}[1]{\langle #1 \rangle}

\newcommand{\acc}[2]{\left\langle #1\, ,\,#2 \right\rangle} 
\newcommand{\wt}{\widetilde}
\newcommand{\str}[1]{\scO_{#1}} 
 
\newcommand{\inv}{\times} 
\newcommand{\vide}{\varnothing}

\newcommand{\Hypz}{\mathcal{N}} 
\newcommand{\Hyp}{\Hypz_{\bQ}} 
\newcommand{\Page}[1]{\Hypz_{#1,\bQ}} 
 
\newcommand{\Spi}{\mathcal{S}} 
\newcommand{\PD}{\mathcal{D}} 
\newcommand{\Hfan}{H\Sigma}
\newcommand{\valu}{{\rm val}}

\def\bZ{{\mathbb Z}}

\def\bN{{\mathbb N}}
\def\QQ{{\mathbb Q}}
\def\bQ{{\mathbb Q}}

\def\bC{{\mathbb C}}

\def\deg{{\mathrm{deg}}}

\def\div{{\mathrm{div}}}

\def\supp{{\mathrm{supp}}}

\def\bG{{\mathbb{G}}}

\def\bA{{\mathbb{A}}}

\def\bP{{\mathbb{P}}}

\let\D\poldiv
\newcommand{\leqD}{\leq_{_{\PD}}}
\newcommand{\leqtD}{\leq_{_{\wt{\PD}}}}
\newcommand{\leqE}{\leq_{_{\E}}}

\newcommand{\hypercone}{HC}
\newcommand{\hcone}{\hypercone}
\newcommand{\cayley}{C}

\def\TT{{\mathbb{T}}}

\def\div{{\rm div}}

\def\dim{{\rm dim}}

\def\rk{{\rm rk}}

\newcommand{\dbr}[1]{[[ #1 ]]}

\newcommand{\opcit}{{\it op.cit.}}

\newcommand{\genus}{\rho_{g}} 

\begin{document}
\title[The Nash problem for torus actions]{The Nash problem for torus actions of complexity one}

\author{David Bourqui}
\address{IRMAR,
Universit\'e de Rennes 1,
Campus de Beaulieu,
b\^atiment 22,
35042 Rennes Cedex,
France}
\email{david.bourqui@univ-rennes1.fr}
\urladdr{https://perso.univ-rennes1.fr/david.bourqui/}

\author{Kevin Langlois}
\address{Departemento de Matem\'atica, Universidade Federal do Cear\'a (UFC), Campus do Pici, Bloco 914, CEP 60455-760. Fortaleza-Ce, Brazil}
\email{langlois.kevin18@gmail.com}
\urladdr{https://sites.google.com/site/kevinlangloismath/}

\author{Hussein Mourtada}
\address{Universit\'e Paris-Cité, CNRS, Institut de Math\'ematiques de Jussieu-Paris Rive Gauche, F-75013, France}
\email{hussein.mourtada@imj-prg.fr}
\urladdr{https://webusers.imj-prg.fr/~hussein.mourtada/}


\begin{abstract}
We solve the equivariant generalized Nash problem for any
non-rational normal variety with torus action of complexity
one. Namely, we give an explicit combinatorial description of the
Nash order on the set of equivariant divisorial valuations on any
such variety. Using this description, we positively solve the
classical Nash problem in this setting, showing that every essential
valuation is a Nash valuation.  We also describe terminal valuations
and use our results to answer negatively a question of de Fernex and Docampo by
constructing examples of Nash valuations which are neither minimal nor
terminal, thus illustrating a striking new feature of the class of
singularities under consideration.
\end{abstract}

\maketitle

\section{Introduction}
By  Hironaka's theorem (1964), an algebraic variety $X$ defined over a
  field of characteristic zero has infinitely many resolutions of
  singularities. In his celebrated paper \cite{MR1381967} (written in
  1968), Nash tried to capture some common information to all these
  resolutions using the arc space associated with $X$.
He defined an injection from the set of  irreducible components of the space of arcs centered at the
  singular locus of $X$ to the set
  of essential divisors (or divisorial valuations) of $X$. Here, a divisorial valuation $\nu$ centered
  on the singular locus of an algebraic variety $X$ is said to be
  essential if it appears on every resolution of singularities (more
  precisely, if the center of $\nu$ on every resolution of
  singularities of $X$ is an irreducible component of the exceptional
  locus). Nash asked whether this injection was surjective.
Since then, a tremendous amount of work has been dedicated to this
question, hereafter also designated by the classical Nash problem.
Numerous works dealing with the case of surfaces
(see \eg \cite{Reg:ratsurf,MR1719822,Ple:apropos,PlePop:classnonrat,MR2451219,PePer:Nashquotient}) culminated in
\cite{MR2979864}, where Bobadilla and Pereira showed that the answer
was always positive in dimension two. On the other hand, in
\cite{IshKol:Nash}, Ishii and Kollar gave a 4-dimensional example for
which the answer was negative; 3-dimensional counter-examples to the
classical Nash problem were exhibited thereafter in \cite{deFer:three:cex} and \cite{MR3094648}.

These counter-examples prompted the challenge to determine stronger
and geometrically meaningful conditions on essential valuations that
guarantee, in any dimension, that they are Nash valuations. The
ultimate goal would be to establish a nice general and geometric
characterization of Nash valuations. We hereafter designate this
question by the extended Nash problem. A remarkable step in this direction was accomplished in
\cite{MR3437873}, where de Fernex and Docampo showed that any terminal
valuation is Nash, recovering in particular the result of Bobadilla
and Pereira. In \cite{MR2905305}, Lejeune and Reguera showed
that any valuation determined by a non-uniruled exceptional divisor
is Nash.

A useful way towards a better apprehension of a general characterization of
Nash valuations is to determine significantly large families of higher
dimensional singularities for which the answer to Nash's
question is positive. The first example is again due to Ishii and Kollar in
\cite{IshKol:Nash}: they showed that the classical Nash problem holds for toric
singularities; subsequent refinements of this result are to be found in \cite{Ish:arcs:valuations}.
Nowadays, in dimension $\geq 3$, the classical Nash problem is known to hold for the following families:
some quasi-rational hypersurface singularities (\cite{MR2894841}),
some families of 3-dimensional hypersurfaces (\cite{MR3393448}),
possibly reducible quasi-ordinary hypersurface singularities (\cite{MR2359548}),
a certain class of normal isolated singularities of arbitrary dimension,
(\cite{PlePop:families}), Schubert varieties (\cite{DocNig:grassman}).

Being given two divisorial valuations $\nu$, $\nu'$ on an algebraic
variety, the condition: ``any arc with order $\nu'$ is a limit of arcs
with order $\nu$'' defines a poset structure on the set of divisorial valuations,
the Nash order, and the Nash valuations may be understood as the
minimal elements for the Nash order of the set of valuations centered at the singular locus.
Thus a better understanding of the nature of the Nash valuations could be achieved
via the study of the Nash order. The generalized Nash problem asks for a meaningful description of the
Nash order. Few instances of a satisfactory solution to this problem
seem to be known. It has been solved for equivariant valuations on toric
varieties (\cite{Ish:toric,Ish:maximal}) and determinantal varieties
(\cite{MR3020097}); in \cite{DocNig:grassman}, some partial results
are obtained for Grassmanians. In its full generality, it seems to be a very
difficult problem, even for valuations on the affine plane (see
\cite{MR3709136}).

\subsection{Results on the generalized and the classical Nash problems}
In the present paper, we study the classical Nash problem and its generalized version for 
a normal variety endowed with a complexity one torus action.
Assuming that our variety is non-rational, in other words that the rational quotient
by the torus action is a curve of positive genus, and using the
semi-combinatoric description of its geometry due to Altmann-Hausen
and Timashev, we solve the generalized Nash problem for equivariant valuations.
\begin{theo}\label{theo:main}
Let $k$ be an algebraically closed field of characteristic zero.
Consider a divisorial fan $\E$ over a smooth projective curve of positive
genus, and the associated normal $k$-variety $X$, which comes equipped
with a complexity one torus action.
Then the Nash order on the set of equivariant valuations of $X$ can be explicitly described in terms 
of an ``hypercombinatorial'' order on the set of integral points of the book of
valuations associated with $\E$, which on each page of the book
coincide with the combinatorial order used in the toric case by Ishii
and Kollar in \cite{Ish:maximal}.
\end{theo}
See Subsection \ref{subsec:hypercombin:order} for a precise
description of the hypercombinatorial order alluded to in the statement.

Note that despite the terminology, though a solution to the generalized
Nash problem in particular allows to explicitly describe the set of
Nash valuations, it does not provide a solution to the classical Nash
problem for free, for the latter requires in addition a sufficiently fine
understanding of the set of essential valuations.

Under the same hypotheses as Theorem \ref{theo:main}, we study
carefully resolutions of singularities in order to 
describe the set of essential valuations; this description coupled with Theorem \ref{theo:main} allows
us to give a positive answer to the classical Nash problem. 

\begin{theo}\label{theo:main:2}
We keep the setting of Theorem \ref{theo:main}. Then any essential valuation of $X$ is a Nash valuation. Equivalently, the
Nash map is bijective.
\end{theo}
In fact our arguments allow to establish
Theorems \ref{theo:main} and \ref{theo:main:2}
also for any normal variety endowed with a complexity one torus
action and which is toroidal. Note however that in this case, the bijectivity
of the Nash map can be deduced directly from Ishii and Kollar's result
(see Remark \ref{rema:nash:toroidal}).

In an ongoing work, we plan to get back to the case where the rational
quotient is a curve of genus zero, which seems much more
challenging (see Subsection \ref{subsec:rationalcase} below, as well
as Section \ref{sec:casePone}).

\subsection{Terminal valuations and an answer to a question of
 de Fernex and Docampo}
Under the assumptions of Theorem \ref{theo:main}, we also give a
description of the set of terminal valuations (which is a subset of
the set of Nash valuations by \cite{MR3437873}); see Theorem \ref{theo:terval:cplxone}.
As pointed out in \cite[\S 6.3, {\em in fine}]{MR3437873}
in the case of the two main families for which the Nash map is known to be bijective, the Nash
valuations are either all of them minimal (the case of toric
varieties) or all of them terminal (the case of surface
singularities), and in general no examples of Nash valuations which are neither
minimal nor terminal are known.
This led the authors of \opcit\ to suggest that the set
of Nash valuations might be in full generality the (non-disjoint) union of the set of minimal
valuations and the set of terminal valuations.
As pointed out in the discussion after Corollary 6.16
in \cite{deFer:survey}, if this happens to be the case, this would 
give a complete solution to the extended Nash problem.

Using Theorem \ref{theo:main} and our description of the terminal valuations, we obtain
the following result, which shows in some sense that the extended Nash
problem remains widely open, and illustrates a striking and completely
new feature of the class of singularities that we are considering.
\begin{theo}\label{thep:miniter}
Keep the setting of Theorem \ref{theo:main}. 
Then there are examples of such varieties $X$ possessing Nash
valuations which are neither minimal nor terminal.
\end{theo}
See \S \ref{subsec:nash:nonterm:nonmin} for a description of such examples.

\subsection{An application}
Trinomial hypersurfaces are classical examples of normal affine varieties with torus action of
complexity one. Using the work of
Kruglov (cf. \cite{MR4057986}), which makes the Altmann-Hausen
description explicit for trinomial hypersurfaces, our main result on
the classical Nash problem implies the following.
\begin{theorem}\label{TheoremTrinomialNash}
Consider the hypersurface
\[X  =  \mathbb{V}(t_{1}^{\underline{n}_{1}} + t_{2}^{\underline{n}_{2}} + t_{3}^{\underline{n}_{3}}) \subset \mathbb{A}_{\mathbb{C}}^{n}
\text{ with }t_{i}^{\underline{n}_{i}} =  \prod_{j = 1}^{r_{i}} t_{i, j}^{n_{i,j}} \text{ for } i = 1,2,3,\]
where $r_{i}, n_{i,j}\in\mathbb{Z}_{> 0}$ and $n =  r_{1} + r_{2} + r_{3}$. Let us write
$ u_{i}:= {\rm gcd}(n_{i,1}, \ldots, n_{i, r_{i}})$, $d =  {\rm gcd}(u_{1}, u_{2}, u_{3}),$
$ d_{1}  =  {\rm gcd}(u_{2}/d, u_{3}/d)$,  $d_{2}  =  {\rm gcd}(u_{1}/d, u_{3}/d)$, $ d_{3}  =  {\rm gcd}(u_{1}/d, u_{2}/d),$
and $u  =  dd_{1}d_{2} d_{3}$.
If \[u - d_{1} - d_{2} - d_{3}\geq 0,\]
then the classical Nash problem is valid for the variety $X$, \ie the Nash map
is bijective.
\end{theorem}
Note that Theorem \ref{TheoremTrinomialNash} is sharp in the sense that there are examples of trinomial hypersurfaces 
with condition $u - d_{1} - d_{2} - d_{3}< 0$ for which the Nash map
is not bijective. For instance, the hypersurface 
of Johnson-Kollar $t_{0}t_{1} + t_{2}^{2} +t_{3}^{5} =  0$ (see
\cite{MR3094648} and Section \ref{sec:casePone}) is such an example. 
\subsection{Strategy of proof}
We now describe our strategy for proving Theorems \ref{theo:main} and \ref{theo:main:2}. 
For any normal variety $X$ equipped with a complexity one torus
action, there exists a natural equivariant proper morphism $\wt{X}\to X$,
called the toroidification, where $\wt{X}$ has toroidal
singularities (see subsection \ref{subsec:toroidification}).\\
In case $X$ is non-rational,
and drawing our inspiration from
the proof of the main result of \cite{MR2905305}
we show that this morphism has the following crucial property:
any wedge on $X$ not contained in the
singular locus lifts to $\wt{X}$ (see Subsection \ref{subsec:lifting:wedges}).\\
The fact that $\wt{X}$ has toroidal singularities allows to describe
explicitly the Nash order on the set of its equivariant valuations,
starting from the known case of toric varieties, and using elementary
properties of étale morphisms and a consequence of Reguera's curve
selection lemma (see Subsection \ref{subsec:hypercombin:order}). Then
one deduces from the lifting property along the toroidification the
sought for description of the Nash order on $X$ (see subsection
\ref{subsec:nashorder:posgen}); again Reguera's curve selection lemma
is a crucial tool.

In order to establish the bijectivity of the Nash map, we now have to
locate the essential valuations of $X$. The toroidification morphism is a
partial equivariant desingularization of $X$, in the sense that any
equivariant desingularization of $X$ factors through it (see
Subsection \ref{subsec:eqres}); it is a consequence of Luna's slice
Theorem.  And since $\tilde{X}$ is toroidal, the location of its
essential valuations is provided by Ishii and Kollar's argument in the
toric case. Using the description of the exceptional locus of the
toroidification, this already gives information on the essential
valuations of $X$, but not sharp enough to conclude.
Due in particular to the fact that in general not every essential valuation of $\tilde{X}$ is
an essential valuation of $X$,
some extra work is needed in order to obtain a sufficiently
accurate description of a finite set of
equivariant valuations containing all the essential valuations of our
non-toroidal variety $X$; the arguments remain purely combinatorial,
but one has to take into account the degree of the polyhedral divisor
defining $X$ (see Subsection \ref{subsec:prop:T:ess}).
By our description of the Nash order, the finite set we obtain is contained in
the set of Nash valuations, and we are done, since any Nash valuation
is essential.

\subsection{The rational case}\label{subsec:rationalcase}
In case $X$ is a rational normal variety equipped with a complexity
one torus action, the above arguments completely fail. Even in
dimension $2$, where the classical Nash problem is known to have a positive
solution, it is no longer true that the toroidification is a partial equivariant desingularization of $X$,
nor that wedges lift along the toroidification. The 3-dimensional
counter-example to the classical Nash problem of Kollar and Johnson, which, as
already pointed out, can be
equipped with a complexity one torus action, may also be
seen as an illustration of the issues encountered when the rational
quotient of $X$ is the projective line. We strongly believe that the
situation deserves deeper investigation. Though
the toroidification seems no longer useful 
to understand the essential valuations of $X$,
we might still be able to exploit the semi-combinatorial description of the
equivariant resolutions to obtain a meaningful interpretation of
essential valuations in terms of hypercones. Understanding the Nash
order in terms of the semi-combinatorial data is also an interesting 
challenge, already in dimension 2. As a first step in this direction,
using the toric embedding defined by Ilten and Manon in \cite{MR3978437}, we give a combinatoric
description of the pointwise order (which is finer than the Nash
order) on the set of equivariant valuations (see Proposition \ref{prop:pointwise}).
Let us stress that in view of Johnson-Kollar's threefold, one may hope
that a finer understanding of the issues in the rational case should allow to produce systematic families of counter-examples
to the classical Nash problem among varieties equipped with a torus action,
which in turn would be useful to get a better comprehension of why 
essential valuations fail in general to be Nash.

\subsection{Organization of the paper}
We briefly describe the content of each section of the paper.

Section \ref{sec:valu} contains the necessary background on essential
and Nash valuations, maximal divisorial sets and Reguera's curve
selection lemma. Special attention is paid to the toric case, which is
an important ingredient of the proof of our results.

In Section \ref{sec:algtorac}, we recall some useful facts about the
geometry of normal varieties equipped with a complexity-one torus action.
In the last subsection, some technical lemmas about
extensions of valuation along étale morphisms are proved.

In Section \ref{sec:eqres}, we obtain information about the
equivariant resolutions and the location
of the essential valuations of a non-rational variety equipped with 
a complexity-one torus action.

In Section \ref{sec:nash}, we define a poset structure of
combinatorial nature on the set of equivariant valuations 
of a variety equipped with a complexity-one torus action.
In the non-rational case, we show that wedges lift to the
toroidification, allowing us to deduce the main results of the paper.

In Section \ref{sec:term}, we give a combinatorial descriptions of the
terminal valuations of a non-rational variety equipped with 
a complexity-one torus action.

Section \ref{sec:exam} describes some examples constructed from
polyhedral divisors and illustrating our results. In particular
examples of Nash valuations which are neither terminal nor minimal are given.

Section \ref{sec:casePone} discusses the rational case, pointing out
the extra difficulties in comparison with the non-rational case, and
giving a combinatorial description of the pointwise order on
the set of equivariant valuations.

\subsection{Acknowledgments}
We thank Roi Docampo for useful discussions.

We are grateful to Shihoko Ishii for her answers about the behaviour
of the Nash problem with respect to étale morphisms (see Remark 
\ref{rema:nash:toroidal}).

This work was initiated during visits of K. Langlois at IRMAR
(Universit\'e de Rennes 1) and H. Mourtada at the Mathematisches Institut (Heinrich-Heine-Universit\"at D\"usseldorf),
and some progresses were made while D. Bourqui and K. Langlois were
visiting the Institut de Math\'ematiques de Jussieu-Paris Rive Gauche
(Universit\'e Paris-Cit\'e). We are grateful to these institutions for their
hospitality and financial support. We have also benefited from the
financial support of the ANR project LISA (ANR-17-CE40-0023).

D. Bourqui and H. Mourtada were partially supported by the PICS project {\em More 
Invariants in Arc Schemes}.

\subsection{General notation}
Let $(E,\prec)$ be a poset. For any subset $F$ of $E$, we denote by
$\Min(F,\prec)$ the set of elements of $F$ which are minimal for the
induced poset structure. If $\prec'$ is another order on $E$, we
denote by $\prec'\imply \prec$ the fact that $\prec'$ is finer than
$\prec$, \ie for any $x_1,x_2\in E$, $x_1\prec'x_2\imply x_1\prec x_2$.

Throughout the whole paper, we fix an algebraically closed
  field of characteristic zero, denoted by $k$.
By an algebraic variety $X$ over $k$, we mean an integral scheme
separated and of finite type over $k$. We denote by $X^{\sing}$ the
singular locus of $X$ and by $X^{\sm}$ its smooth locus. In case $X$
is affine, we denote by $k[X]$ its algebra of global regular functions.

\section{Essential and Nash valuations and the Nash order}\label{sec:valu}
\subsection{Essential valuations}
\label{sec:essential-valuations}
By a {\em ($k$)-valuation} of an extension $K$ of $k$, we always mean
a discrete valuation on $K$ trivial on $k$ and with group of value
contained in $\bZ$, that is, a map
$\nu\colon K\to \bZ\cup \{+\infty\}$ such that:
\begin{enumerate}
\item $\nu(k)=\{0\}$;
\item $\forall f,g\in K,\quad \nu(fg)=\nu(f)+\nu(g)$;
\item $\forall f,g\in K,\quad \nu(f+g)\geq \inf(\nu(f),\nu(g))$;
\item $\nu^{-1}(\{+\infty\})=\{0\}$.
\end{enumerate}
The associated valuation ring is $\cO_{\nu}:=\{f\in K,\,\nu(f)\geq 0\}$.
The valuation ideal $\cM_{\nu}:=\{f\in K,\,\nu(f)> 0\}$
is then a prime ideal of $\cO_{\nu}$. A {\em $\bQ$-valuation} is a map $\nu\colon K\to \bQ\cup \{+\infty\}$
such that a positive multiple of $\nu$ is a valuation in the previous sense.

By an algebraic variety $X$ over $k$, we mean an integral scheme
separated and of finite type over $k$. We denote by $X^{\sing}$ the
singular locus of $X$ and by $X^{\sm}$ its smooth locus. In case $X$
is affine, we denote by $k[X]$ its algebra of global regular functions.

Let $X$ be an algebraic variety $X$ over $k$
and $K:=k(X)$ be the function field of $X$. A valuation $\nu$ on $K$ is
said to be {\em centered on $X$} if there exists a non-empty open affine subset $U$
of $X$ such that $k[U]\subset \cO_{\nu}$. Then the prime ideal
$\cM_{\nu}\cap k[U]$ defines a schematic point $\cent_X(\nu)$ of $X$ which does not
depend on the choice of $U$ and is called the {\em center of $\nu$ on  $X$}. 
In the literature, the center is often rather defined as the closed subset of $X$
whose $\cent_X(\nu)$ is the generic point. We adopt the present
definition since it will prove convenient to use the schematic
language of generic points.
We denote by $\Val(X)$ the set of valuations on $k(X)$ which are centered
on $X$. In case $X$ is normal and $E$ is a prime divisor on $X$, the
valuation $\ord_E$ defined by the local ring of $X$ at $E$ (which is a
discrete valuation ring) and its multiples are prototypical examples
of elements of $\Val(X)$.

A {\em divisorial valuation} of $X$ is an element $\nu\in \Val(X)$
such that there exists a normal $k$-variety $Z$ which is $k$-birational to $X$
and such that $\nu$ is centered on $Z$, and $\Adh(\cent_Z(\nu))$
is a prime divisor of $Z$.
Equivalently, there exist a prime divisor $E$
on $Z$ and a positive integer $\ell$ such that $\nu=\ell\cdot \ord_E$.
We say that the divisorial valuation $\nu$ is of multiplicity one if
one can take $\ell=1$. Equivalently, $\nu(k(X)^{\star})=\bZ$.

Let $\DV(X)\subset \Val(X)$ be the set of divisorial valuations on
$X$.
For the sake of convenience, the trivial valuation (which sends $K\setminus\{0\}$ to $0$)
is considered as a divisorial valuation. We set
\[
\DV(X)^{\sing}:= \{\nu\in \DV(X),\quad \cent_X(\nu)\in X^{\sing}\}.
\]
A {\em resolution of singularities of $X$} is a a proper birational morphism
$f\colon Z\to X$ with $Z$ a smooth variety and where $f$ induces an isomorphism $f^{-1}(X^{\sm})\isom X^{\sm}$.
Its {\em exceptional locus} $\Exc(f)\subset Z$ is the closed subset of $Z$ where $f$ is not an
isomorphism, in other words $z\in \Exc(f)$ if and only if the induced morphism of
local rings $\str{X,f(z)}\to \str{Z,z}$ is not an isomorphism.
We say that a divisorial valuation $\nu\in \DV(X)$ is {\em
  $f$-exceptional}
if $\cent_Z(\nu)$ is the generic point of an irreducible component of $\Exc(f)$.

A {\em divisorial resolution of the singularities of $X$} is a resolution of singularities
 $f\colon Y\to X$ such that the exceptional locus $\Exc(f)$ is of pure codimension $1$.

A divisorial valuation $\nu\in \DV(X)$ of multiplicity one is said to be a ({\em
  divisorially}) {\em essential valuations} of $X$ if
for every (divisorial) resolution $f\colon Z\to X$ of the singularities of $X$,
$\nu$ is $f$-exceptional.
Denote by $\Ess(X)$ (resp. $\DivEss(X)$) the set of essential
(resp. divisorially essential) valuations on $X$.
Note that by the very definitions one has $\Ess(X)\subset \DivEss(X) \subset \DV(X)^{\sing}$.

We now assume that the variety $X$ is equipped with an action of an algebraic group $G$.
We denote by $\Val(X)_G$ the set of elements of $\Val(X)$ which are
$G$-equivariant (\ie invariant under the natural action of $G(k)$ on $k(X)$) ; we also use the self-explanatory notation $\DV(X)_G$ and so on.

A divisorial valuation $\nu\in  \DV(X)$ is said to be a ({\em divisorially}) {\em $G$-essential valuations} of $X$ 
if for every (divisorial) $G$-equivariant resolution $f\colon Z\to X$ of the singularities of $X$,
$\nu$ is $f$-exceptional.
Note that by
\cite[Proposition 3.9.1]{MR2289519},
there exist $G$-equivariant divisorial resolutions of
the singularities of $X$. 
Let $G-\Ess(X)$ (resp. $G-\DivEss(X)$) be the set of 
$G$-essential (resp. divisorially $G$-essential)
valuations of $X$.

The following diagram of inclusions follows directly from the definitions.
\[
\xymatrix{
\DivEss(X)
\ar@{}[r]|-*[@]{\subset}
&
G-\DivEss(X)
\ar@{}[r]|-*[@]{\subset}
&
\DV(X)^{\sing}_G
\\
\Ess(X)
\ar@{}[u]|-*[@]{\subset}
\ar@{}[r]|-*[@]{\subset}
&
G-\Ess(X)
\ar@{}[u]|-*[@]{\subset}
&
}
\]

\subsection{Arc spaces, fat arcs and the pointwise order}
\label{subsec:arc-spaces-fat}
We denote by $\arc(X)$ the space of arcs associated with a
$k$-algebraic variety $X$.
For more information and references on the properties of arc spaces,
see \eg \cite{MR2559866}. Recall in particular that $\arc(X)$ is a
$k$-scheme such that for any $k$-extension $K$ one has a functorial
bijection between the set $\arc(X)(K)$ of $K$-points of $\arc(X)$ and
the set $X(K\dbr{t})$ of $K\dbr{t}$-points of $X$. An element of
$\arc(X)(K)$ will be called a $K$-arc on $X$. If $\alpha$ is a $K$-arc
on $X$, the natural $k$-algebra morphism $K\dbr{t}\injects K((t))$
(resp. $K\dbr{t}\to K,\,t\mapsto 0$) induces an element of $X(K((t))$ 
(resp. $X(K)$) respectively, whose image in $X$ is called the generic point
(resp. the origin or the special point) of the $K$-arc $\alpha$.
Note that any $\alpha\in \arc(X)$, with residue field $\kappa(\alpha)$,
naturally induces a $\kappa(\alpha)$-arc on $X$ (often still denoted
by $\alpha$). Recall also that there exists a natural $k$-morphism
$\pi_X\colon \arc(X)\to X$ mapping any $\alpha\in \arc(X)$ to its
origin $\alpha(0)$. For any subset $A\subset \arc(X)$ we denote by
$\arc(X)^A$ the set $\pi_X^{-1}(A)$.
Recall also that any morphism $f\colon Z\to X$ of algebraic varieties
naturally induces a morphism of $k$-schemes
$\arc(f)\colon \arc(Z)\to \arc(X)$  such that for any $k$-extension $K$, the map
$\arc(Z)(K)\to \arc(X)(K)$ induced by $\arc(f)$ coincides with the map
$Z(K\dbr{t})\to X(K\dbr{t})$ induced by $f$.

Let $\alpha\in \arc(X)$
and $U$ be an open affine subset of $X$ containing the origin of $\alpha$.
Then $\alpha$ induces a $k$-algebra morphism $\alpha^{\ast}_U\colon k[U]\to \kappa(\alpha)[[t]]$.
Following \cite{Ish:arcs:valuations}, we say that $\alpha$ is a {\em fat arc} if $\alpha^{\ast}_U$ is
injective. The condition does not depend on the choice of $U$, and is
equivalent to the fact that the generic point of $\alpha$ is the
generic point of $X$ (Recall that $X$ is assumed to be irreducible.). If $\alpha$ is fat, it defines a valuation
$\ord_{\alpha}\in \Val(X)$ as follows:  for any $f\in k[U]\setminus\{0\}$, one sets $\ord_\alpha(f):=
\ord_t(\alpha^{\ast}f)\in \bN$. Note that $\cent_{X}(\ord_{\alpha})=\alpha(0)$.
We denote by $\arc(X)^{\fat}$ the set of fat arcs on $X$.

Let $\nu\in \Val(X)$. We denote by $\arc(X)^{\ord=\nu}$  the set of
fat arcs on $X$ such that $\ord_{\alpha}=\nu$. Note that
$\arc(X)^{\ord=\nu}$ is non-empty (see \eg \cite[Proposition 3.12]{mor:arc:valuations}).

\begin{rema}\label{rema:prop:bir:fat:ord}
Let $f\colon Y\to X$ be a proper birational morphism. By the valuative
criterion of properness, $f$ induces a bijective morphism
\[
\arc(f)\colon \arc(Y)^{\fat}\to \arc(X)^{\fat}.
\]
Moreover $\Val(X)$ naturally identifies with $\Val(Y)$ and the diagram
\[
\xymatrix{
\arc(Y)^{\fat} \ar[d]_{\arc(f)} \ar[rr]^{\ord} && \Val(Y)=\Val(X)\\
\arc(X)^{\fat}\ar[urr]_{\ord}&&
}
\]
is commutative. In particular $\arc(f)$ sends $\arc(Y)^{\ord=\nu}$
bijectively to $\arc(X)^{\ord=\nu}$. 
\end{rema}
There is a natural poset structure on $\Val(X)$, the ``pointwise''
poset structure, introduced by Ishii in \cite{Ish:maximal}.
\begin{definition}
  One defines a poset structure $\leqX$ on $\Val(X)$ as follows: 
let $\nu_1,\nu_2\in \Val(X)$~; then  $\nu_1\leqX\nu_2$ if
one of the following equivalent conditions hold.
\begin{enumerate}
\item there exists
a non-empty open affine subset $U$ of $X$ such that $\cent_{\nu_1}(X)\in
U$ and for any $f\in k[U]$ one has $\nu_1(f)\leq \nu_2(f)$;
\item for any $f\in \str{X,\cent_{\nu_1}(X)}$, one has $\nu_1(f)\leq \nu_2(f)$;
\item for any non-empty open affine subset $U$ of $X$ such that $\cent_{\nu_1}(X)\in
U$ and for any $f\in k[U]$ one has $\nu_1(f)\leq \nu_2(f)$.
\end{enumerate}
For $\nu\in \Val (X)$, one sets
\[
\arc(X)^{\ord\leqX\nu}:=\{\alpha\in \arc(X),\quad \ord_{\alpha}\leqX\nu\}
\]
\[
\text{and}\quad \arc(X)^{\ord\geqX\nu}:=\{\alpha\in \arc(X),\quad \ord_{\alpha}\geqX\nu\}.
\]
\end{definition}
\begin{rema}
In case $X$ is affine, for any $\nu_1,\nu_2\in \Val(X)$, one has $\nu_1\leqX\nu_2$ if and only if 
for any $f\in k[X]$ one has $\nu_1(f)\leq \nu_2(f)$.
\end{rema}
This poset structure behaves well with respect to the topological order
on arcs, thanks to the following fundamental property, which is a
straightforward consequence of \cite[Proposition 2.7]{Ish:arcs:valuations}.
\begin{prop}\label{prop:upp:dxnu}
Let $\alpha_1,\alpha_2\in \arc(X)^{\fat}$. Assume that $\alpha_2\in \Adh(\alpha_1)$.
Then $\ord_{\alpha_1}\leqX \ord_{\alpha_2}$.
\end{prop}

\subsection{Maximal divisorial sets and the Nash order}
\label{subsec:nash:order}
Let $\nu\in \DV(X)$. Following \cite[Definition 2.8]{Ish:maximal},
one defines $\mds_X(\nu)$ as the Zariski closure in $\arc(X)$ of the
set $\arc(X)^{\ord=\nu}$. One calls $\mds_X(\nu)$
the {\em maximal divisorial set (in short mds) associated with $\nu$}.
By \opcit, $\mds_X(\nu)$ is irreducible and
its generic point $\eta_{X,\nu}$ is fat and satisfies $\ord_{\eta_{X,\nu}}=\nu$.

\begin{rema}\label{rema:restriction:mds}
  Let $\nu\in \DV(X)$ and $U$ be any open subset of $X$ containing
  $\cent_X(\nu)$.  Then $\mds_X(\nu)\cap \arc(U)=\mds_U(\nu)$, and $\eta_{X,\nu}=\eta_{U,\nu}$.
\end{rema}
\begin{rema}\label{rema:prop:bir:fat:ord:bis}
Retain the notation of Remark \ref{rema:prop:bir:fat:ord}.
Then, since $\arc(f)$ is continuous, 
one obtains the inclusion $\arc(f)(\mds_Y(\nu))\subset \mds_X(\nu)$
and the equality $\Adh(\arc(f)(\mds_Y(\nu)))=\mds_X(\nu)$. Here, $\Adh$ stands for the Zariski closure in $\arc(X)$.
\end{rema}
\begin{definition}
We define a poset structure $\leqmdsX$ (or $\leqmds$ when there is no
risk of confusion) on $\DV(X)$ as follows: 
let $\nu_1,\nu_2\in \DV(X)$~; then  $\nu_1\leqmdsX\nu_2$ if and only if $\mds_X(\nu_2)\subset \mds_X(\nu_1)$.
We call $\leqmdsX$ the Nash order (or mds order) on $\DV(X)$.
\end{definition}
\begin{rema}\label{rema:nash:order:local}
One has $\nu_1\leqmdsX\nu_2$ if and only if $\eta_{X,\nu_2}$ is a
  specialization of $\eta_{X,\nu_1}$. In particular, if
  $\nu_1\leqmdsX\nu_2$, then $\cent_X(\nu_2)$ is a specialization of $\cent_X(\nu_1)$.

Moreover, one has $\nu_1\leqmdsX\nu_2$ if and only if there exists an open
subset $U$ of $X$ containing $\cent_X(\nu_2)$ such that $\nu_1\leqmdss{U}\nu_2$.

Indeed by Remark \ref{rema:restriction:mds}, if $\nu_1\leqmdsX\nu_2$,
then for any open subset $U$ of $X$ containing $\cent_X(\nu_2)$ one has $\nu_1\leqmdss{U}\nu_2$.

On the other hand, assume that there exists an open
subset $U$ of $X$ containing $\cent_X(\nu_2)$ and $\cent_X(\nu_2)$ such that $\nu_1\leqmdss{U}\nu_2$.
Then $\eta_{U,\nu_2}$ is a specialization of $\eta_{U,\nu_2}$ in
$\arc(U)$. But $\eta_{U,\nu_i}=\eta_{X,\nu_i}$ and $\Adh_{\arc(U)}(\eta_{X,\nu_1})\subset \Adh_{\arc(X)}(\eta_{X,\nu_1})$.
\end{rema}

\begin{prop}\label{prop:proper:bir:ord:mds}
Let $f\colon Y\to X$ be a proper birational morphism.
Let $\nu_1,\nu_2\in \DV(Y)=\DV(X)$, and assume that 
$\mds_Y(\nu_2)\subset \mds_Y(\nu_1)$. Then $\mds_X(\nu_2)\subset \mds_X(\nu_1)$.
In other words, on $\DV(Y)=\DV(X)$, one has $\leqmdsY\,\imply\, \leqmdsX$.
\end{prop}
\begin{proof}
Indeed, if $\mds_Y(\nu_2)\subset \mds_Y(\nu_1)$, since $\arc(f)$ is continuous, then 
\[\Adh(\arc(f)(\mds_Y(\nu_2)))\subset \Adh(\arc(f)(\mds_Y(\nu_1)))\]
and one concludes by Remark \ref{rema:prop:bir:fat:ord:bis}.
\end{proof}
\begin{rema}
If $f\colon Y\to X$ be a proper birational morphism and $\nu\in \DV(Y)=\DV(X)$, 
then $\arc(f)(\eta_{Y,\nu})=\eta_{X,\nu}$. Indeed
\[
\Adh(\arc(f)(\eta_{Y,\nu}))=\Adh(\arc(f)(\Adh(\eta_{Y,\nu})))
=\Adh(\arc(f)(\mds_Y(\nu)))=\mds_X(\nu)
\]
thus $\arc(f)(\eta_{Y,\nu})$ is the generic point of $\mds_X(\nu)$.
\end{rema}

\subsection{The Nash valuations and the Nash problem}
\label{sec:nash-valuations-nash}
\begin{definition}\label{defi:nashval}
The set of {\em Nash valuations} of $X$, denoted by $\Nash(X)$, is the
set of minimal elements of $\DV(X)^{\sing}$ with respect to the Nash order $\leqmdsX$:
\[
\Nash(X):=\Min(\DV(X)^{\sing},\leqmdsX).
\]
\end{definition}
\begin{rema}\label{rema:nash:equiv}
In particular, in the equivariant case, one has 
$\Nash(X)\subset \DV(X)_G^{\sing}$.
In fact, since by Remark \ref{rema:min:order:mds} every $\nu\in \DV(X)^{\sing}$ is $\geqmds \nu'$ for $\nu'\in
\Nash(X)$, one even has $\Nash(X)=\Min(\DV(X)_G^{\sing},\leqmds)$.
\end{rema}
\begin{rema}\label{rema:min:order:mds}
Following \cite{IshKol:Nash}, let us give the original definition of
the Nash valuations by Nash in \cite{MR1381967},
translated into a modern schematic language. With every irreducible component $C$ of $\arc(X)^{X^{\sing}}$ not
contained in $\arc(X^{\sing})$, one associates an essential valuation
of $X$ as follows: take $Y\to X$ a resolution of the singularities of
$X$. Then by the valuative criterion of properness, the generic point $\alpha$
of $C$ lifts to an element $\wt{\alpha}$ of $\arc(Y)$. One then
shows that the closure $\wt{\alpha}(0)$ defines an essential
divisor, and that this construction defines an injective map (the Nash
map) from the set of irreducible components of $\arc(X)^{X^{\sing}}$  to the set $\Ess(X)$.
The set $\Nash(X)$ of Nash valuations is defined as the image of the Nash map.

The latter definition is equivalent to definition \ref{defi:nashval}.
Indeed, by \cite{deFEinIsh:divisorial:valuations}, we know that $\arc(X)^{X^{\sing}}$ has a finite
number of irreducible fat components (that is, their generic point is
a fat arc), and that these fat components are maximal divisorial sets.
Moreover by  \cite{IshKol:Nash}, $\arc(X)^{X^{\sing}}$ has no
irreducible component whose generic point is not fat (recall that the
latter property is known to fail in nonzero characteristic).
On the other hand, for any  $\nu'\in \DV(X)^{\sing}$,
$\mds_X(\nu')$ is an irreducible closed subset of  $\arc(X)^{X^{\sing}}$.
Summing up, the $\mds_X(\nu)$ for $\nu\in \Nash(X)$ (in the sense of
definition \ref{defi:nashval}) are exactly the irreducible components of $\arc(X)^{X^{\sing}}$.
\end{rema}
From definition \ref{defi:nashval} and the above remark, one obtains:
\begin{prop}[Nash]\label{prop:Nash:Ess}
One has $\Nash(X)\subset \Ess(X)$.
\end{prop}
The {\em Nash problem}, in its original formulation, asks whether this inclusion is an equality.
Since \cite{IshKol:Nash}, one knows that the equality does not hold in
general, and a more general form of the Nash problem (already present in \cite{MR1381967}) would ask for a sensible
geometric characterization of the elements of $\Nash(X)$ among the elements of $\Ess(X)$.

Following \cite{Ish:maximal,DocNig:grassman,MR3709136},
what we call the {\em generalized Nash problem} asks for a meaningful interpretation of the
Nash order on $\DV(X)$. The problem is well-understood for equivariant
valuations on toric varieties (\cite{Ish:maximal}, see also \ref{prop:mds:combin:toric} below)
and determinantal varieties (\cite{MR3020097}).
In \cite{DocNig:grassman}, some partials results are obtained
concerning the generalized Nash problem for contact strata in
Grassmanians (See also \cite{Mou,MPl,KMPT} for a variant of this
problem, namely the embedded Nash problem, for a class of surface singularities.).
Theorem \ref{theo:nash:order:Yggeq1} below solves the generalized Nash problem for equivariant
valuations on non-rational normal varieties equipped with a complexity
one torus action.
\begin{rema}\label{rema:mds:local}
For any non-empty open subset $U$ of $X$, one has
$\Nash(U)=\Nash(X)\cap \DV(U)$ and $\Ess(U)=\Ess(X)\cap \DV(U)$. Thus
the Nash problem is a local problem.

The generalized Nash problem is also of local nature, in the following
sense. Let $\nu_1,\nu_2\in \DV(X)$; then the following are equivalent:
\begin{itemize}
\item $\nu_1\leqmdsX\nu_2$;
\item there exists a non-empty open affine
subset $U$ of $X$ such that $\cent_{X}(\nu_1)\in U$, $\cent_X(\nu_2)\in U$ and 
$\nu_1\leqmdss{U}\nu_2$;
\item for any non-empty open subset $U$ of $X$ such that
  $\cent_{X}(\nu_1)\in U$ and $\cent_X(\nu_2)\in U$, 
one has $\nu_1\leqmdss{U}\nu_2$.
\end{itemize}
Also note that if the condition $\nu_1\leqmdsX\nu_2$ is fulfilled,
then $\cent_{X}(\nu_2)$ is a specialization of $\cent_X(\nu_1)$; in
particular, for any covering $X=\cup_{i\in I}U_i$ of $X$ by affine open subsets $U_i$,
there always exists $i\in I$ such that $\cent_{X}(\nu_1),\cent_X(\nu_2)\in U_i$.
\end{rema}
The following proposition, due to Ishii (\cite[Lemma 3.11]{Ish:maximal}), shows that
the pointwise order is finer than the Nash order.
\begin{prop}\label{prop:leqmds:implies:leqX}
Let $\nu,\nu'\in \DV(X)$ such that $\nu\leqmds \nu'$.
Then $\nu\leqX \nu'$.
\end{prop}
\begin{coro}\label{coro:pointwise:order:min}
The set $\Min(\DV(X)^{\sing},\leqX)$ is contained in $\Nash(X)$.
Moreover for every $\nu\in \DV(X)^{\sing}$, there exists $\nu'\in \Min(\DV(X)^{\sing},\leqX)$
such that $\nu'\leqX \nu$.
\end{coro}
\begin{proof}
The inclusion is a direct consequence of Proposition \ref{prop:leqmds:implies:leqX}.
We show the second assertion. By Remark \ref{rema:min:order:mds}, there
exists $\nu_1\in \Nash(X)$ such that $\nu_1\leqmds \nu$ thus $\nu_1\leqX\nu$.
If $\nu_1\in \Min(\DV(X)^{\sing},\leqX)$ we are done. Otherwise, let
$\nu_2\in \DV(X)^{\sing}$ such that $\nu_2\leqX \nu_1$ and $\nu_2\neq \nu_1$.
Again by Remark \ref{rema:min:order:mds}, there
exists $\nu_3\in \Nash(X)$ such that  $\nu_3\leqX\nu_2$. Since
$\nu_3\leqX \nu_2 \leqX \nu_1$ and $\nu_1\neq \nu_2$, we must have
$\nu_3\neq \nu_1$. Since $\Nash(X)$ is finite, repeating the process a
finite number of times will eventually produce $\nu_i\in
\Min(\DV(X)^{\sing},\leqX)$ such that $\nu_i\leqX \nu$.
\end{proof}
\begin{defi}
The elements of the set $\MinVal(X):=\Min(\DV(X)^{\sing},\leqX)$ are called the {\em minimal valuations} on $X$.
\end{defi}
\begin{rema}\label{rema:min:equiv}
In the equivariant case, one has 
\[
\MinVal(X)=\Min(\DV(X)^{\sing}_G,\leqX)
\]
Indeed, let $\nu\in \Min(\DV(X)^{\sing}_G,\leqX)$ and $\nu'\in
\DV(X)^{\sing}$ such that $\nu'\leqX \nu$. There exists $\nu''\in
\DV(X)^{\sing}_G$ such that $\nu''\leqmds \nu'$. In particular
$\nu''\leqX \nu'\leqX \nu$ and $\nu=\nu'$.
\end{rema}
\subsection{The toric case}
\label{subsec:toric}
Thanks to the work of Ishii (\cite{Ish:toric,Ish:maximal}), the Nash order and its relation with
the pointwise order are especially well understood for the toric
valuations of a toric variety. We recall here the relevant
definitions and results, some of which will be needed for our study of the Nash
order on varieties equipped with a complexity one torus
action.
We in particular explain how to obtain the solution of the
Nash problem for toric varieties; though our presentation does not
feature substantial differences with the original argument of Ishii and Kollar in \cite{IshKol:Nash},
it is slightly more direct, in particular thanks to the use of the results of Ishii's
paper \cite{Ish:toric}.

We use standard toric notation, definitions and facts
(see also Section \ref{sec:algtorac} below; a standard
reference on toric geometry is \cite{CLS11}).

We assume that $X=X_{\sigma}=\Spec(k[\sigma^{\vee}\cap M])$ is a toric affine
$k$-variety,
where $M=\Hom(N,\bZ)$ is the dual of a lattice $N$
and $\sigma^{\vee}$ is the dual of a strictly convex polyhedral cone
$\sigma$ of $N\otimes_{\bZ} \bQ$.
Thus $X$ is equipped with an action of the torus $\TT:=\Spec(k[M])$. Let $n\in
\sigma\cap N$. Let $f\in k[X]$, and write $f=\sum_{m\in M\cap \sigma^{\vee}}f_m\cdot \chi^m$
with $f_m\in k$ and where $\chi^{m}$ is character associated with $m\in M$. Set
\[
\nu_n(f):=\Infsubu{m\in M\cap \sigma^{\vee}\\f_m\neq 0}\acc{m}{n}.
\]
Then $\nu_n\in \DV(X)_{\TT}$ and the map $n\mapsto \nu_n$
is a bijection between $\DV(X)_{\TT}$ and $\sigma\cap N$.
Denote by $\sigma_{\sing}$ the union of the relative interiors of the
non-smooth faces of $\sigma$.
Modulo the above identification, one has
\[\DV(X)^{\sing}_{\TT}=\sigma_{\sing}\cap N.\]

On $\DV(X)_{\TT}$, besides $\leqX$ and $\leqmds$, one defines, following
\cite[Definition 4.6]{Ish:toric}, a third natural poset structure of combinatorial nature.
\begin{definition}\label{defi:leqs}
For any $\nu,\nu'\in \DV(X)_{\TT}$, set
\[
\nu\leqs \nu'\text{ iff }\nu'\in \nu+\sigma.
\]
\end{definition}
The following proposition is a straightforward consequence of the definition.
\begin{prop}\label{prop:pointwise:combin:toric}
On $\DV(X)_{\TT}$ the  poset structures defined by $\leqs$  and
$\leqX$ coincide.
\end{prop}
Using Remark \ref{rema:min:equiv}, one then obtains:
\begin{coro}\label{coro:toric:min:min:comb}
The set $\MinVal(X)$ identifies with $\Min(\sigma_{\sing}\cap N,\leqs)$.
\end{coro}
The resolution of the Nash problem in the toric case is then a
direct consequence of the latter corollary and Proposition \ref{prop:nu:not:exc}
below, the proof of which is purely combinatorial and contained in the proof of
\cite[Lemma 3.15]{IshKol:Nash}. Note that this proof relies on a
construction due initially to Bouvier and Gonzalez-Sprinberg (\cite{BouGon:sysgen}).
In order to state the proposition, and for the sake of convenience, we first introduce some notation and
terminology which will also be useful later on. We will often use it
in case the fan to be refined is the fan of the faces of a cone
$\sigma$, identified by abuse of notation with $\sigma$ itself.
\begin{defi}\label{defi:strong:economic}
Let $\Sigma$ be a fan and $\Sigma'$ be a fan refining $\Sigma$.
\begin{enumerate}
\item The fan $\Sigma'$ is said to be a star refinement of $\Sigma$
if $\Sigma'$ may be obtained from $\Sigma$ by a finite succession of
star subdivisions (see \cite[\S 11.1]{CLS11}).
\item The fan $\Sigma'$ is said to be a big refinement of $\Sigma$
  if the following holds: for any \(\tau\in \Sigma\) such that
  \(\tau\notin \Sigma\), there exists a ray of \(\tau\)
  which is not a ray of $\Sigma$.
\item Assume that $\Sigma'$ is smooth, \ie every cone of $\Sigma'$ is smooth;
the fan $\Sigma'$ is said to be a smooth economical refinement of $\Sigma$ if every smooth cone of $\Sigma$
  is a cone of $\Sigma'$.
\end{enumerate}
\end{defi}
\begin{rema}
$\Sigma'$ is a smooth economical refinement of $\Sigma$ 
if and only if the induced equivariant proper birational morphism $X(\Sigma')\to
X(\Sigma)$ is a resolution of singularities of $X(\Sigma)$. Indeed,
the condition guarantees that the induced morphism is an isomorphism
over the smooth locus of $X(\Sigma)$.

On the other hand $\Sigma'$ is a big refinement of $\Sigma$ if and only if
the exceptional locus of $X(\Sigma')\to X(\Sigma)$ has pure codimension
$1$.
\end{rema}
\begin{prop}\label{prop:nu:not:exc}
Let $\nu$ be a primitive element in $(\sigma_{\sing}\cap N)\setminus \Min(\sigma_{\sing}\cap N,\leqs)$.
Then there exists a fan $\Sigma$ which is a big and smooth economical
star refinement of $\sigma$ and such that the cone $\tau$ of $\Sigma$
such that $\nu\in \Relint(\tau)$ has dimension $\geq 2$.

In particular, there exists a divisorial equivariant resolution $f$ of the singularities 
of $X$ such that $\nu$ is not $f$-exceptional; in other words $\nu\notin \TT-\DivEss(X)$.
\end{prop}
\begin{rema}\label{rem:prop:nu:not:exc}
This statement contains the well-known fact that there exists
a big and smooth economical star refinement of any strictly convex
polyhedral cone $\sigma$ (see \cite[Theorem 11.1.9]{CLS11}).
\end{rema}
\begin{coro}[The Nash problem in the toric case]\label{coro:nash:problem:toric}
Let $X$ be an affine toric variety.
Then ${\TT}-\DivEss(X)={\TT}-\Ess(X)=\DivEss(X)=\Ess(X)=\Nash(X)=\MinVal(X)$.
In particular, the Nash problem has a positive answer in the toric case.
\end{coro}
\begin{proof}
Proposition \ref{prop:nu:not:exc} shows that ${\TT}-\DivEss(X)\subset \MinVal(X)$. Since 
the inclusions $\MinVal(X)\subset \Nash(X)\subset \Ess(X)\subset
{\TT}-\Ess(X)\subset {\TT}-\DivEss(X)$ 
and $\Ess(X)\subset \DivEss(X)\subset {\TT}-\DivEss(X)$
always hold, one gets the result.
\end{proof}
Now let us turn to the generalized Nash problem on toric varieties.
The following proposition is a consequence of \cite[Proposition 4.8]{Ish:toric}
and \cite[Example 2.10 \& Lemma 3.11]{Ish:maximal}.
\begin{prop}\label{prop:mds:combin:toric}
On $\DV(X)_{\TT}$ the three poset structures defined by $\leqs$, 
$\leqX$ and $\leqmds$ coincide.
\end{prop}

\subsection{Stable points and Reguera's curve selection lemma}
The notion of stable points of the arc scheme of an algebraic variety
was introduced by Reguera. One of the main feature of these points is that their
formal neighborhood is noetherian, a fact allowing Reguera to obtain
a version the curve selection lemma for arc spaces which turned out to be crucial in
subsequent works on the Nash problem. We only recall here the definitions and properties which are relevant
for our needs; see \cite{Reg:CSL,MR4236197,Reg:towards,MouReg,deF:Doc:diff}
for more information and results on stable points. We also state and prove a simple consequence of Reguera's curve selection lemma
(Corollary \ref{coro:csl}) which we shall use later. Recall that a
wedge on $X$ is an arc on $\arc(X)$, in other words a point of the
scheme $\arc(\arc(X))$. The generic point (resp. the special point) of
a wedge is called its generic arc (resp. its special arc). Note that
the closure in $\arc(X)$ of the generic arc of a wedge on $X$ always
contains the special arc of the wedge.

\begin{definition}
(See \cite{Reg:CSL,MR4236197,Reg:towards} as well as \cite[\S 10]{deF:Doc:diff}.
Let $X$ be an algebraic $k$-variety. A point $\alpha\in \arc(X)$ is stable if it is not contained in $\arc(X^{\sing})$  
and it is the generic point of an irreducible constructible subset of $\arc(X)$.
\end{definition}
\begin{theo}\label{theo:stable}
\begin{enumerate}
\item
Let $\nu\in \DV(X)$ and $\eta_{X,\nu}$ be the generic point of
$\mds_X(\nu)$. Then $\eta_{X,\nu}$ is a stable point of $\arc(X)$.
\label{item:theo:divisorial:stable} 
\item
Let $\alpha\in \arc(X)$ be a stable point. Then the following holds.
\label{item:theo:prop:stable} 
\begin{enumerate}
\item
\label{item:theo:stable:gene}
Every generization of $\alpha$ is again a stable point.
\item \label{item:theo:stable:dim}
The Krull dimension of the local ring $\str{\arc(X),\alpha}$ is finite.
\item  \label{item:theo:stable:csl} 
(the curve selection lemma for stable points) Let $N$ be an irreducible closed subset of $\arc(X)$, such that
$\Adh(\alpha)$ is a proper subset of $N$. Then there
exist an extension $K$ of $k$ and a $K$-wedge $\Spec(K\dbr{t,u})\to
X$ with special arc $\alpha$ and generic arc an element of $N\setminus \Adh(\alpha)$.
\end{enumerate}
\end{enumerate}
\end{theo}
\begin{proof}
Assertion \ref{item:theo:divisorial:stable} comes from 
\cite[\S 2.4]{MouReg}. Assertion  \ref{item:theo:stable:gene} is
  \cite[Proposition 3.7(vi)]{Reg:towards}; it is also a direct
consequence of \cite[Proposition 10.5]{deF:Doc:diff}.
Assertion \ref{item:theo:stable:dim} is
\cite[Proposition 3.7(iv)]{Reg:towards} whereas 
assertion \ref{item:theo:stable:csl}
is a consequence of \cite[Corollary 4.8]{Reg:CSL}.
\end{proof}
\begin{coro}\label{coro:csl}
 Let $\alpha,\alpha'\in \arc(X)$ be two stable points such that
 $\alpha$ is a specialization of $\alpha'$. Then there exist an
 extension $K/k$ and a finite sequence of $K$-wedges $w_1,\dots,w_r$
 on $X$ such that the special arc of $w_1$ is $\alpha$, the generic arc of
 $w_r$ is $\alpha'$ and for any $1\leq i\leq r-1$ the generic arc of
 $w_i$ is the special arc of $w_{i+1}$.
\end{coro}
\begin{proof}
The result is trivial if $\alpha=\alpha'$. Thus one may assume that
$\Adh(\alpha)$ is a proper subset of $N':=\Adh(\alpha')$. By assertion
\ref{item:theo:stable:csl} of Theorem \ref{theo:stable},
there exist an extension $K/k$ and a $K$-wedge 
$w_1$ on $X$ with special arc $\alpha$ and generic arc
$\alpha_1\in N'\setminus \Adh(\alpha)$. If $\alpha_1=\alpha'$, we are done.
Otherwise, note that $\alpha_1$ is a (strict) generization of
$\alpha'$, thus is a stable point of $\arc(X)$, and $\Adh(\alpha_1)$
is a proper subset of $N'$. Thus we may apply the curve selection
lemma again and (extending $K$ if necessary) find a $K$-wedge 
$w_2$ on $X$ with special arc $\alpha_1$ and generic arc
$\alpha_2\in N'\setminus \Adh(\alpha_1)$. Since the local ring
$\str{\arc(X),\alpha}$ has finite Krull dimension, there does not
exist arbitrary long sequences $\alpha_0=\alpha,\alpha_1,\dots,\alpha_{r}$ of $\arc(X)$
such that $\alpha_i$ is a strict generization of $\alpha_{i+1}$. Thus
after a finite number of steps of the above procedure, we end up with
a wedge with generic arc $\alpha'$.
\end{proof}

\section{Algebraic torus actions of complexity one}
\label{sec:algtorac}
Recall that if $\TT$ is an
algebraic torus, the \emph{complexity} of an effective algebraic
$\TT$-action on a variety $X$ is the number $\dim(X)-\dim(\TT)$.
This section introduces preliminaries from Altmann-Hausen's theory
\cite{AH06, AHS08} for the classification of effective algebraic torus
actions on normal varieties, limiting ourselves to the case of
complexity one, though the theory deals with torus actions of arbitrary
complexity. In the case of complexity zero, Altmann-Hausen's theory reduces to the
classical setting of the combinatorial classification of normal toric
varieties. Complexity-one normal $\TT$-varieties are instances of
$\TT$-varieties where, similarly to the toric case, Altmann-Hausen's
description is particularly explicit, and where one may use Timashev's
language of hypercones \cite{Tim08}, which we will also recall.

\subsection{Cones and polyhedrons}
\label{sec:cones-polyedrons}
We start by fixing standard toric notation, trying to respect as much as
possible the notation and terminology of the standard reference \cite{CLS11}.
Namely, $N\simeq \bZ^{d}$ is a lattice, $M =  \Hom(N, \bZ)$ is the
dual lattice and $M_{\bQ}, N_{\bQ}$ are respectively the associated
$\bQ$-vector spaces obtained from $M, N$ by tensoring with $\bQ$. We denote by
\[M_{\bQ}\times N_{\bQ}\rightarrow \bQ,\,\, (m, n)\mapsto \acc{m}{n}
\]
the natural pairing deduced from the duality between $M$ and $N$. The notation $\TT$ stands for the algebraic torus
$\bG_{m}\otimes_{\bZ}N\simeq \bG_{m}^{n}$ whose character and one-parameter subgroup lattices are respectively $M$ and $N$. 
We distinguish two notations: the lattice vector $m\in M$ and the character $\chi^{m}$ corresponding to $m$ seen as 
regular function on the torus $\TT$. 
For a polyhedral cone $\sigma$ we denote by $\Relint(\sigma)$ its relative interior (i.e. the complement of the union of its proper faces) and by 
\[\sigma^{\vee} :=  \{m\in M_{\bQ}\,|\, \acc{m}{n}\geq 0\text{ for any } n\in \sigma\}\] its dual cone. 
Recall that $\sigma$ is said to be \emph{strictly convex} if $\{0\}$
is a face of $\sigma$, and ($N$)-\emph{smooth} if $\sigma$ may be
generated as a cone by a part of a $\bZ$-basis of the lattice $N$.

Given any  polyhedron $\cP\subset N_{\bQ}$ we set 
\[\Tail(\cP):= \{v\in N_{\bQ}\,|\, v + \cP \subset \cP\},\]
which is a polyhedral cone of $N_{\bQ}$.

\subsection{Polyhedral divisors}
Let us fix a polyhedral strictly convex cone $\sigma \subset N_{\bQ}$.
From the datum $(N, \sigma)$ we define the semigroup
\[\Pol_{\bQ}^{+}(N, \sigma):= \{ \cP \subset N_{\bQ}\,|\, \cP
  \text{ polyhedron with } \Tail(\cP) =  \sigma\}\]
whose addition is the Minkowski sum and neutral element is the cone $\sigma$. We also consider the extended semigroup
$\Pol_{\bQ}(N, \sigma) =  \Pol_{\bQ}^{+}(N, \sigma)\cup\{\vide\}$,
where the element $\vide$ is an absorbing element, i.e. $\cP + \vide
=  \vide$ for any $\cP \in \Pol_{\bQ}(N, \sigma)$. 

Let $Y$ be a smooth algebraic curve; hereafter we identify $Y$
with its set of closed points. Write
$\Div(Y)$ (resp. $\Div_{\geq 0}(Y)$) for the group of
Cartier divisors (resp. the semigroup of effective Cartier
divisors) on $Y$.
By a \emph{$\sigma$-tailed polyhedral divisor over $(Y,N)$} we mean an element
\[\PD\in \Pol_{\bQ}(N, \sigma)\otimes_{\bZ_{\geq 0}}\Div_{\geq 0}(Y).\]
In particular, $\PD$ has a decomposition as a formal sum
\[\PD =  \sum_{y\in Y}\PD_{y}\cdot [y],\]
where $\PD_{y}\in \Pol_{\bQ}(N, \sigma)$ is equal to $\sigma$ for all but
finitely many $y\in Y$.
The \emph{tail} of $\PD$ denoted by $\Tail(\PD)$ is the cone $\sigma$
and the \emph{locus} of $\PD$ is the non-empty open set of $Y$ defined by
\[\Loc(\PD)  :=  Y\setminus \bigcup_{y\in Y,\, \PD_{y}  =  \vide} [y].\]
The \emph{evaluation} is the piecewise linear map
\[m\in \sigma^{\vee}\mapsto \PD(m):= \sum_{\PD_{y}\neq \vide} \min \acc{\PD_{y}}{m}\cdot y \in  \Div_{\bQ}(\Loc(\PD)),\]
where $\Div_{\bQ}(\Loc(\PD))$ is the vector space of $\bQ$-Cartier
divisors on $\Loc(\PD)$. The \emph{support}
of $\PD$ is the finite set
\[\Supp(\PD) = \{ y\in Y\, |\, \PD_{y}\not\in \{\vide, \Tail(\PD)\}\}.\]

We define the \emph{degree} of $\PD$ as the Minkowski sum
\[\deg(\PD):= \sum_{y\in Y} \PD_{y}\in \Pol_{\bQ}(N, \sigma)\]
when $\Loc(\PD)$ is complete, and in case  $\Loc(\PD)$ affine, we set $\deg(\PD)  :=  \vide$.
\begin{definition}
Let $Y$ be a smooth algebraic curve. A polyhedral divisor $\PD$ over $(Y, N)$
is \emph{proper} (or a \emph{$p$-divisor}) if one of the following
conditions hold
\begin{enumerate}
\item the locus $\Loc(\PD)$ is affine;
\item the locus $\Loc(\PD)$ is complete, $\deg(\PD)$ is a proper
  subset of $\sigma$ and for every $m\in \sigma^{\vee}$ satisfying
  $\deg(\PD(m)) = 0$,
 the $\bQ$-divisor $\PD(m)$ has a principal multiple.
\end{enumerate}
\end{definition}

For any non-empty open subset $Y_0$ of the curve $Y$, consider the restriction
of $\PD$ to $Y_0$, that is to say the polyhedral divisor
\[
\PD_{|_{Y_0}}:=\sum_{y\in Y_0}\PD_y\cdot [y].
\]
Note that if $\PD$ is a $p$-divisor, then $\PD_{|_{Y_0}}$ is also a $p$-divisor.

Each polyhedral divisor $\PD$ has the property $\PD(m) + \PD(m')\leq \PD(m+ m')$ for all $m, m'\in \Tail(\PD)^{\vee}$.
Hence the multiplication on the field  $k(Y)$ naturally defines an $M$-graded $\cO_{\Loc(\PD)}$-algebra
\[\cA(\PD) := \bigoplus_{m\in \Tail(\PD)^{\vee}\cap M}\cO_{\Loc(\PD)}(\PD(m))\chi^{m}.\]
We may therefore define two affine $k$-schemes 
\[\widetilde{X}(\PD) =  \bSpec_{\Loc{\PD}}\cA(\PD)\text{ and } X(\PD) =  \Spec\Gamma(\Loc(\PD), \cA(\PD)).\]
Moreover the $M$-grading on $\cA(\PD)$ naturally induces algebraic $\TT$-actions on the schemes $\widetilde{X}(\PD)$
and $X(\PD)$. By the very construction, one has:
\begin{lemma}\label{lemm:qaffine}
The structural morphism $q\colon\widetilde{X}(\PD)\rightarrow \Loc(\PD)$ is affine. 
\end{lemma}
The following result is a particular case of Altmann-Hausen's
classification result (see \cite[Theorem 3.1, Theorem 3.4, Theorem 8.8]{AH06}
for more general statements in arbitrary complexity; see also
\cite{Lan15} for the complexity one case)
\begin{theorem}
Let $Y$ be a smooth algebraic curve and $\PD$ be a $p$-divisor over $(Y,N)$. Then $X(\PD)$ is a normal affine variety, and the
natural action of $\TT$ on $X(\PD)$ is effective and of complexity one.

Conversely, for any normal affine variety $X$ equipped with an effective
algebraic action of $\TT$ of complexity one, there exist a smooth
algebraic curve $Y$ and  $p$-divisor $\PD$ over $(Y, N)$ such that $X$
is equivariantly isomorphic to $X(\PD)$.
\end{theorem}

\subsection{Divisorial fans}
According to Sumihiro's  Theorem any normal variety with torus action is a finite union of affine open subsets that are stable
by the torus action.
This leads to describe any normal variety with an effective algebraic
torus action of complexity one by a finite collection of $p$-divisors
satisfying similar conditions to those defining the notion of fan for toric varieties.
\begin{definition}
Let $\PD, \PD'$ be two polyhedral divisors over $(Y, N)$.
The \emph{intersection} of $\PD$ and $\PD'$ is the polyhedral divisor
over $(Y,N)$ defined by the relation
\[\PD\cap \PD'  :=  \sum_{y\in Y} \PD_{y}\cap \PD_{y}'\cdot [y]\]

Assume now that $\PD$ and $\PD'$ are two $p$-divisors.
We say that $\PD$ is a \emph{face} of $\PD'$ if for any $y\in Y$,
$\PD_{y}$ is a face of $\PD_{y}'$ and $\deg(\PD) = \deg(\PD')\cap \Tail(\PD)$.
\end{definition}
If $\PD$ and $\PD'$ are two $p$-divisors such that for any $y\in Y$ the polyhedron
$\PD_{y}$ is a face of $\PD_{y}'$,  then $\PD$ is a face of $\PD'$ if and only
if the natural morphism $X(\PD)\rightarrow X(\PD')$ is an open
immersion (\cite[Lemma 1.4]{IS11}).
\begin{definition}
A \emph{divisorial fan} or an \emph{$f$-divisor} over $(Y, N)$ is a finite set $\E$ of $p$-divisors, stable by intersection 
and such that for all $\PD, \PD'\in \E$ the $p$-divisor $\PD\cap \PD'$ is a mutual face of $\PD$ and $\PD'$. 
\end{definition}
Any divisorial fan $\E$ over $(Y, N)$ defines a $k$-scheme $X(\E)$ with
effective algebraic $\TT$-action of complexity one. The scheme $X(\E)$ is obtained by
gluing the family of varieties $(X(\PD))_{\PD\in \E}$ in a such way
that $X(\PD\cap \PD')$ is identified with $X(\PD)\cap X(\PD')$ for all
$\PD, \PD'\in \E$ (\cite[Theorem 5.3]{AHS08}).
By \cite[Remark 7.4]{AHS08}, the $k$-scheme $X(\E)$ is separated and
is thus a normal $k$-variety.
Conversely, any normal variety with an effective algebraic $\TT$-action of complexity one comes from a divisorial fan
(\cite[Theorem 5.6]{AHS08}).

\subsection{Hypercones and hyperfans}
\label{subsec:hcones:hfans}

We now discuss some combinatorial objects coming from the
classification of Timashev of the algebraic torus actions of
complexity one \cite{Tim08}.
We still consider a lattice $N$ and a smooth algebraic curve $Y$.
The associated \emph{hyperspace} or \emph{book} is the set $\Hyp$
defined as the quotient set
\[Y \times N_{\QQ}\times \QQ_{\geq 0}/\sim,\] where the equivalence
relation $\sim$ is given by
\[(y, a, b)\sim (y', a', b')\text{ if and only if }(y = y', a = a', b
  =b')\text{ or }(a = a', b=b' =0).\]
The image of
$(y, a, b)\in Y \times N_{\QQ}\times \QQ_{\geq 0}$ in $\Hyp$ will be denoted by
$[y, a, b]$. In case $b=0$, it is also denoted by $[\bullet,a,0]$
since it does not depend on $y$.

Note that the natural map
\[N_{\bQ}\rightarrow \Hyp,\,\, a\mapsto [\bullet, a, 0]\]
allows to identify $N_{\bQ}$ with a subset of $\Hyp$ called the
\emph{spine} $\Spi$ of $\Hyp$. The set $\Hypz$ of integral points of
$\Hyp$ is the image in $\Hyp$ of $Y\times N\times \bN$.

Let $y\in Y$. The associated \emph{page} of the book $\Hyp$
is the set
\[\Page{y}:=\{[y, a, b]\,|\, (a, b)\in N_{\bQ}\times \bQ_{\geq 0}\}\]
Note that $\Page{y}$ contains $\Spi$ and may be identified with $N_{\bQ}\times \bQ_{\geq 0}$
in a way compatible with the identification of $\Spi$ with $N_{\bQ}=N_{\bQ}\times \{0\}$.
For any two points $y\neq y'$ of $Y$ one has $\Page{y}\cap \Page{y'}=\Spi$.

Consider now a polyhedral divisor $\PD$ over $(Y, N)$ with tail $\sigma$.
For any $y\in Y$, the associated \emph{Cayley cone} is 
the cone $\cayley_{y}(\PD)\subseteq \Page{y}\subset N_{\QQ}\times \QQ$ generated by
$(\sigma\times \{0\})\cup (\PD_{y}\times \{1\})$.

The \emph{hypercone} associated with $\PD$ is the subset of $\Hyp$ 
defined by
\[\hypercone(\PD) : = \bigcup_{y\in Y}\cayley_y(\PD).\]
In particular, one has $\hypercone(\PD)\cap \Spi=\sigma$ and for any $y\in Y$ one has $\hypercone(\PD)\cap \Page{y}=\cayley_y(\PD)$.

\begin{definition}\label{Definition-Hyperface}
For a $p$-divisor $\PD$ over $(Y, N)$ we say that a subset $\theta\subset \Hyp$
is a \emph{hyperface} of the hypercone $\hypercone(\PD)$ if it satisfies one of the following conditions.
\begin{itemize}
\item[$(i)$] We have $\theta = \hypercone(\PD')$, where $\PD'$ is a $p$-divisor
  over $(Y, N)$, with the property that
  $\theta\cap \deg(\PD)\neq \vide$ and $\cayley_{y}(\PD')$ is a
  non-empty face of $\cayley_y(\PD)$ for any $y\in Y$. In other words,
  $\PD'$ has complete locus and is a face of $\PD$.
\item[$(ii)$] We have $\theta\cap \deg(\PD) = \vide$ and $\theta$
  is a face of $\cayley_y(\PD)$ for some $y\in Y$. In this case $\theta$ is a
  subset of $\Page{y}$.
\end{itemize}
Let us now consider a divisorial fan $\E$ over $(Y, N)$.
We call \emph{hyperfan} of the divisorial fan $\E$ the set 
$\Hfan(\E):= \{ \theta \text{ hyperfaces of } \hypercone(\PD) \text{ for some }\PD\in \E\}$.
For an element $\theta\in \Hfan(\E)$ we define its \emph{relative
  interior} $\Relint(\theta)$ and its dimension $\dim(\theta)$ in a obvious way. 
\end{definition}

The viewpoint of hyperfans has the following geometric interpretation
in terms of valuation theory (Remember our convention about valuations in
Section \ref{sec:essential-valuations}.).

Let $X  =  X(\E)$ be the normal variety with an effective algebraic
$\TT$-action of complexity one described by the divisorial fan
$\E$.

We have a one-to-one correspondence $[y, a, b]\mapsto \valu_{[y, a, b]}$ between $\Hyp$ and the set of 
$\TT$-invariant $\bQ$-valuations on $k(X)$ \cite[\S 2, Lemma 1]{Tim08}. This correspondence can be described as follows.

First, remark that, since $X$ is birationally equivalent to
$Y\times \TT$ (see \cite[\S 1, Corollary 3]{Tim08}), the field $k(X)$
is the fraction field of the semigroup algebra
\[k(Y)[M] = \bigoplus_{m\in M}k(Y)\cdot \chi^{m}.\]
Given any element 
\[f = \sum_{m\in M}f_{m}\cdot \chi^{m}\in k(Y)[M],\]
with $f_{m}\in k(Y)^{\inv}$ and every $f_m$ but a finite number of
them is zero,
we define the corresponding valuation $\val_{[y, a, b]}(f)$ via the formula 
\[\valu_{[y, a, b]} (f) = \Infsubu{m\in M\\ f_m\neq 0} \acc{m}{ a}+ b\cdot \ord_{y}(f_{m}),\]
where $\ord_{y}$ is the vanishing order at the point $y$.

\begin{prop}\label{prop:divisorial:T:valuations}
The above one-to-one correspondence $[y, a, b]\mapsto \valu_{[y, a, b]}$ between $\Hyp$ and the set of 
$\TT$-invariant $\bQ$-valuations on $k(X)$ induces a one-to-one
correspondence between the set $\DV(X)_{\TT}$ of
$\TT$-invariant divisorial valuations on $X$
and the set
$(\cup_{\PD\in \E}\hypercone(\PD))\cap \Hypz$. Modulo this
correspondence, for any $\PD\in \E$
and $\nu\in (\cup_{\PD\in \E}\hypercone(\PD))\cap \Hypz$, one has
$\cent_X(\nu)\in X(\PD)$ if and only if $\nu\in \hypercone(\PD)$.
Moreover, if $[y, a, b]\in \hypercone(\PD)\cap \Hypz$ for $\PD\in \E$, for any 
non-empty open subset $Y_0$ of $Y$ such that $y\in Y_0$, 
$\valu_{y, a, b}$ is centered at $X(\PD_{|_{Y_0}})$.
\end{prop}
\begin{proof}
 The only non-obvious point is that the valuations
 induced by integral elements of the hypercone are divisorial. But this follows from \cite[Proposition 19.8]{Tim11}.
\end{proof}

\subsection{Prime invariant cycles and hyperfaces}
\label{sec:prime-invar-cycl}
Let $Y$ be a smooth algebraic curve, $\E$ be a divisorial fan over $(Y,N)$ and $X=X(\E)$ be the
associated $\TT$-variety.
Call \emph{prime $\TT$-cycle} of the $\TT$-variety $X$ any
$\TT$-stable irreducible closed subset of $X$.
Elements of the hyperfan $\Hfan(\E)$
bijectively correspond to prime $\TT$-cycles of $X$.
More precisely, the prime $\TT$-cycle $\cZ(\theta)\subset X$ (also
denoted by $\cZ(\E,\theta)$ if some confusion on the divisorial fan
under consideration may occur) associated
with $\theta\in \Hfan(\E)$ is the closure of the center in $X$ of any $\bQ$-valuation $\valu_{[y, a, b]}$ where $[y, a,b]$ belongs to 
$\Relint(\theta)$ (see \cite[\S 4, Theorem 6]{Tim08} for more
details; remember our convention on the center of a valuation in Subsection
\ref{sec:essential-valuations}). Furthermore, the codimension of $\cZ(\theta)$ is equal to
$\dim(\theta)$, and the correspondence $\theta\mapsto \cZ(\theta)$
respects the ordering, namely,
$\cZ(\theta_{1})\subset \cZ(\theta_{2})$ if and only if $\theta_{2}$
is a (hyper)face of $\theta_{1}$.
Note also the following: let $\PD\in \E$, $\nu\in \hypercone(\PD)$
and $\theta$ be the face of $\hypercone(\PD)$ such that $\nu\in \Relint(\theta)$.
Assume that $\theta\cap \deg(\PD)\neq \vide$. Then $\Adh(\cent_X(\nu))=\cZ(\hypercone(\PD'))$
where $\PD'$ is the unique face of $\PD$ with tail $\theta\cap \Spi$.

\subsection{Toroidification}\label{subsec:toroidification}
Let $Y$ be a smooth algebraic curve. If $\E$ is a divisorial fan over
$(Y, N)$, then the varieties $\widetilde{X}(\PD)$ for $\PD\in \E$ glue
together into a $\TT$-variety $\widetilde{X}$, which may be described by the
following divisorial fan over $(Y, N)$: let $(U_{i})_{i\in I}$
be any finite set of open sets of $Y$ that cover $Y$; then 
$\widetilde{X}$ is isomorphic to $X(\widetilde{\E})$ where $\widetilde{\E}$
is the divisorial fan generated by $\{\PD_{|U_{i}}\}_{\PD\in \E,\,i\in I}$.
\begin{definition}
Set $X  =  X(\E)$. The morphism $\pi\colon \widetilde{X}\rightarrow X$ obtained by gluing the natural morphisms
$\widetilde{X}(\PD)\rightarrow X(\PD)$  for $\PD\in \E$ is called the \emph{toroidification} (or the \emph{contraction map}) of $X$. 
\end{definition}
Note that the toroidification is always proper and birational \cite[Theorem 3.1 (ii)]{AH06}.
Consider now the \emph{rational quotient} $p\colon X\dashrightarrow Y$ induced by the inclusion $k(Y) =  k(X)^{\TT}\subset k(X)$.
Let $X_{0}\subseteq X$ be a Zariski dense open subset in which $p_{|X_{0}}$ is a morphism. We call \emph{graph}
of the rational map $p$ the Zariski closure of the subset 
\[\{(x, y)\in X_{0}\times Y\,|\, y  =  p(x)\}\subset X\times Y.\]
The next result is an application of Zariski Main Theorem. 
\begin{proposition}\label{VolTor}\cite[\S 3, Lemma 1]{Vol10}
Let $X$ be a normal variety with effective algebraic $\TT$-action of complexity one. Assume that $X$ is described by a divisorial fan $\E$
over $(Y, N)$, where $Y$ is a smooth projective curve. Then the total space of the toroidification $\widetilde{X}$ is equivariantly isomorphic to the normalization of the graph of the rational quotient $p:X\dashrightarrow Y$. Under this identification, the toroidification
$\pi\colon\widetilde{X}\rightarrow X$ is induced by the natural projection on $X$ of the graph of $p$ and the global quotient $q\colon\widetilde{X}\rightarrow Y$ by the projection on $Y$. 
\end{proposition}
\subsection{Exceptional locus of a toroidal refinement}
\label{sec:except-locus-toro}
Let $Y$ be a smooth algebraic curve, $\PD$ be a $p$-divisor over
$(Y,N)$ and $\E$ be a divisorial fan over $(Y,N)$ refining $\PD$.
We assume that $\E$ is toroidal, in other words that for any $\PD'\in E,$
$\Loc(\PD')$ is affine; equivalently, $X(\E)$ is isomorphic to
$X(\widetilde{\E})$, or the natural morphism $f\colon X(\E)\to X(\PD)$ 
factors through the toroidification morphism.
Let
$
\E^{\exc}
$
be the set of elements $\theta$ of $\Hfan(\E)$ such that $\theta\cap
\deg(\PD)\neq \vide$ or $\theta$ is not a hyperface of $\hypercone(\PD)$.
\begin{prop}\label{prop:except-locus-toro}
The exceptional locus of the proper birational morphism $f\colon X(\E)\to X(\PD)$ 
is
\[
\Exc(f)=\cup_{\theta\in \E^{\exc}}\cZ(\E,\theta).
\]
\end{prop} 
As a particular case, we obtain a description of the exceptional locus
of the toroidification morphism.
\subsection{Toric étale charts on the toroidification }
\label{sec:toric-etale-charts}
Let $Y$ be a smooth algebraic curve, $y\in Y$ and $\varpi_{y}$ be a uniformizer of the local ring
$\cO_{Y, y}$. Then the map
\[
\phi \colon U \rightarrow \bA^{1},\,\, u\mapsto \varpi_{y}(u)
\]
is an \'etale morphism for some affine Zariski open subset
$U\subset Y$ containing $y$. If $\phi^{-1}(0) = \{y\}$, then we say
that the pair $(U, \phi)$ is an \emph{\'etale chart} around the point
$y$. The following is a well known fact from the theory of toroidal
embeddings \cite[Chapter 4]{KKMS73}.
Since the explicit construction of the involved étale morphism will be
used later on,
we include a short proof.
\begin{lemma}\label{lemma-toroidal}
Let $\PD$ be a $p$-divisor over
  $(Y, N)$ and let $y\in \Loc(\PD)$. Let $(U, \phi)$ be an \'etale
  chart around $y$ such that $U\subset \Loc(\PD)$ and
  $U\cap \Supp(\PD) \subset \{y\}$.
Consider the following $p$-divisor over $(\bA^{1}, N)$ with locus $U$:
\[\PD_{\phi} :=  \sum_{y\in U}\PD_{y}\cdot [\phi(y)].\]

Then there exists a $\TT$-equivariant
isomorphism $X(\PD_{|U})\simeq U\times_{\bA^{1}}X(\PD_{\varphi})$.
Moreover, this isomorphism induces a $\TT$-equivariant \'etale
morphism between $X(\PD_{|U})$ and the toric 
$\bG_{m}\times \TT$-variety $X_{\cayley_y(\PD)}$.
\end{lemma}
\begin{proof}
For any $m\in \Tail(\PD)^{\vee},$ write
$\PD(m)_{|U} = a_{m}\cdot [y]$ for some $a_{m}\in \bQ$. Consider $t\in k[\bA^{1}]$ so that
$k[\bA^{1}] = k[t]$ and $\phi^{\star}(t) = \varpi_{y}$.
Then, setting $V := \varphi(U)$, observe that
\[k[U]\otimes_{k[t]}\Gamma(V, \PD_{\phi}(m)) =
k[U]\otimes_{k[t]}t^{-\lfloor a_{m}\rfloor}\cdot k[t]\text{ and }
\Gamma(U, \PD(m)) = \varpi_{y}^{-\lfloor a_{m}\rfloor}\cdot
k[U].\]
So the $k$-algebra morphism
\[\delta: k[U]\otimes_{k[\bA^{1}]}\cA(\PD_{\phi})(V)\rightarrow \cA(\PD)(U),\,\, f\otimes \gamma\mapsto f\cdot \phi^{\star}(\gamma)\]
is an isomorphism, proving that
$X(\PD_{|U})\simeq U\times_{\bA^{1}}X(\PD_{\varphi})$. Finally, after
identifying $X(\PD_{|U})$ with $U\times_{\bA^{1}}X(\PD_{\phi})$, we
define an \'etale morphism by composing the natural projection
$X(\PD_{|U})\rightarrow X(\PD_{\phi})$ with the open immersion
$X(\PD_{\phi})\hookrightarrow X_{\cayley_y(\PD)}$. This proves the lemma.
\end{proof}

\subsection{Extension of valuations and étale morphisms}
In this subsection we state and prove some technical lemmas about
extensions of valuation along étale morphisms, which will be useful to
study the Nash order in Section \ref{sec:nash}. The only direct
connection between the present subsection and the rest of Section \ref{sec:algtorac}
is that one lemma involves the étale morphism of Lemma \ref{lemma-toroidal}.

\begin{lemma}\label{lemma-etale-lifting-equiv-val}
Let $X$ and $Z$ be affine algebraic $k$-varieties equipped with the
action of an algebraic torus $\TT$ and $\theta\colon X\to Z$ be
a $\TT$-equivariant étale morphism.
Let $\mu\in \Val(Z)_{\TT}$ and $\nu\in \Val(X)$ such that $\theta^{\ast}(\nu)=\mu$.
Then $\nu$ is $\TT$-equivariant.
\end{lemma}
\begin{proof}
Since $\theta$ and $\mu=\theta^{\ast}(\nu)$ are $\TT$-equivariant, for any $t\in \TT(k)$, one has
$\theta^{\ast}(t\cdot \nu)=\mu$. Since $\theta$ is étale, the set
$\{t\cdot \nu,\,t\in \TT(k)\}$ is thus finite. On the other hand,
$\TT(k)$ acts transitively on it. Since $k$ is algebraically closed,
$\TT(k)$ has no subgroup of finite index. Thus
one has $\{t\cdot \nu,\,t\in \TT(k)\}=\{\nu\}$ and $\nu$ is $\TT$-equivariant.
\end{proof}
\begin{proposition}\label{proposition-toroidal-valuations}
Keep the notation and assumptions of Lemma \ref{lemma-toroidal}.
Denote by $\theta$ the $\TT$-equivariant étale morphism
$X:=X(\PD_{|U})\to Z:=X_{\cayley_{y}(\PD)}$ of Lemma \ref{lemma-toroidal}.

Let $[y,a,b]$ be an element of $\hypercone(\PD)\cap \Hypz$. In particular $(a,b)$ defines an
element of the Cayley cone $\cayley_{z}(\PD)$.
Let $\nu:=\val_{[y, a, b]}$ be the induced $\TT$-equivariant valuation on $X$.
\begin{enumerate}
\item The valuation $\mu:=\theta^{\ast}(\nu)$ is the toric valuation $\val_{a, b}$  on
$Z$ induced by $(a,b)$. Moreover, upon shrinking $U$, $\nu$ is the
only element of $\Val(X)_{\TT}$ such that $\theta^{\ast}\nu=\mu$.
\item Upon shrinking $U$,
one has $\arc(\theta)^{-1}(\arc(Z)^{\ord=\mu})=\arc(X)^{\ord=\nu}$,
the map $\arc(X)^{\ord=\nu}\to \arc(Z)^{\ord=\mu}$ induced by
$\arc(\theta)$ is onto, and maps
$\eta_{X,\nu}$ to $\eta_{Z,\mu}$.
\end{enumerate}
\end{proposition}
\begin{proof}
Identifying $k[Z]$ (resp. $k[X]$) with a subring
of $k[\bZ\times M]$ (resp. $k(U)[M]$),
we may write
\[
k[Z]=\bigoplus_{(r,m)\in {\cayley_z(\PD)}^{\vee}\cap \bZ\times M}
k\cdot t^{r}\chi^{m}
\quad
\text{and}\quad k[X]=\bigoplus_{m\in \sigma^{\vee}\cap M} 
H^0(U,\floor{\PD_{|U}(m)})\cdot \chi^m.
\]
By the very construction, $\theta^{\ast}\colon k[Z]\to k[X]$ maps $t^{r}\chi^{m}$ to $\varpi_z^{r}\chi^{m}$.
Now take
\[
f=\sum_{(r,m)\in {\cayley_y(\PD)}^{\vee}\cap \bZ\times M}\alpha_{r,m}\cdot t^{r}\chi^m\in k[Z]
\]
with $\alpha_{r,m}\in k$ for all $r, m$. 
For $m\in M\cap \sigma^{\vee}$ set
\[
\theta^{\ast}(f)_m:=\sumsubu{r\in \bN\\ (r,m)\in {\cayley_z(\PD)}^{\vee}\cap \bZ\times M} \alpha_{r,m}\varpi_z^{r}.
\]
Note that $\ord_{y}(\theta^{\ast}(f)_m)=\Inf \{r\in \bN,\quad
(r,m)\in {\cayley_y(\PD)}^{\vee}\cap \bZ\times M\quad \text{and}\quad \alpha_{r,m}\neq 0\}$.

Then $\val_{[y, a, b]}(\theta^{\ast}f)$ equals
\[
\Infsubu{m\in M\cap \sigma^{\vee}\\\theta^{\ast}(f)_m\neq 0}
b\cdot \ord_{y}(\theta^{\ast}(f)_m)+\acc{m}{a}
\]
which by the above remark equals
\[
\Infsubu{(\nu,m)\in {\cayley_y(\PD)}^{\vee}\cap \bZ\times M\\\alpha_{r,m}\neq 0}
b\cdot r+\acc{m}{a}
\]
and indeed corresponds to the value at $f$ of the toric valuation
$\mu$ associated with $(a,b)$ (see Subsection \ref{subsec:toric}). This shows the first part of the first
assertion.

Now, since there is only a finite number of closed points $y'$ in $U$
such that $\varpi_y(y')=0$, upon shrinking $U$, one may assume that
for any closed point $y'\neq y$ in $U$, one has
$\ord_{y'}(\varpi_y)=0$. Let $\nu'$ be an element of $\Val(X)_{\TT}$
such that $\theta^{\ast}(\nu')=\mu$. Write $\nu'=[y',a',b']$ with $y'\in
U$ and $(a',b')\in \cayley_{z'}(\PD)\cap \Hypz$.

First assume $y'\neq y$. 
Since $\theta^{\ast}(\nu')=\mu$ and $\ord_{y'}(\varpi_y)=0$, for any
element $(m,r)\in\cayley_y(\PD)^{\vee}\cap (M\times \bZ)$
one has
\[
b\cdot r+\acc{m}{a}=\acc{m}{a'}
\]
In particular, taking $r=0$ and $m\in \sigma^{\vee}_{M}$, one infers
that $a=a'$. Taking $(m,r)$ with $r>0$, one then obtains $b=0$.
Thus $\nu'=\val_{[y',a,0]}=\val_{[y,a,0]}=\nu$.

In case $y'=y$, one obtains for any element
$(m,r)\in\cayley_y(\PD)^{\vee}_{M\times \bZ}$ the relation
\[
b'\cdot r+\acc{m}{a}=b\cdot r+\acc{m}{a'}
\]
which easily gives $a=a'$ and $b=b'$.

Let $\beta\in \arc(Z)$ such that $\ord_{\beta}=\mu$.
Let $\mfq\subset k[Z]$ (resp. $\mfp\subset k[X]$)
be the prime ideal defining the center of $\mu$ in $Z$ (resp. of $\nu$
in $X$). Since $\mu=\theta^{\ast}(\nu)$, one has $\mfp\cap k[Z]=\mfq$,
thus  there exists an extension $K$ of $\kappa(\mfp)$ such that the
$K$-point of $Z$ defined by $\mfq$ and the extension $K/\kappa(\mfq)$
lifts to $X(K)$. Since $X\to Z$ is étale, one infers that 
there exists $\alpha\in \arc(X)$ such that $\arc(\theta)(\alpha)=\beta$.
Since $\theta^{\ast}(\ord_{\alpha})=\mu$
and $\theta$ is $\TT$-equivariant, by Lemma \ref{lemma-etale-lifting-equiv-val}, $\ord_{\alpha}$ is $\TT$-equivariant.
By the first assertion, upon shrinking $U$, one may conclude that $\ord_{\alpha}=\nu$.

Let us show that $\arc(\theta)(\eta_{X,\nu})=\eta_{Z,\mu}$.
Let $\alpha\in \arc(X)^{\ord=\nu}$ such that $\arc(\theta)(\alpha)=\eta_{Z,\mu}$.
Since $\alpha$ is a specialization of $\eta_{X,\nu}$, $\eta_{Z,\mu}$
is a specialization of $\arc(\theta)(\eta_{X,\nu})$. Since $\ord_{\arc(\theta)(\eta_{X,\nu})}=\mu$,
one infers that $\arc(\theta)(\eta_{X,\nu})=\eta_{Z,\mu}$.

\end{proof}

\begin{lemma}\label{lemma-etale-lifting-fat-arcs}
Let $X$ and $Z$ be affine algebraic $k$-varieties and $\theta\colon X\to Z$ be
an étale morphism. Let $\nu\in \Val(X)$ and
$\mu:=\theta^{\ast}(\nu)\in \Val(Z)$. Let $\beta\in \arc(Z)$ such
that $\ord_{\beta}=\mu$. Then there exists $\alpha\in \arc(X)$ such
that $\arc(\theta)(\alpha)=\beta$ and $\ord_{\alpha}$ has the same
center on $X$ as $\nu$.
\end{lemma}
\begin{proof}
Let $\mfp$ be the prime ideal of $k[X]$ given
by $\mfp:=\cent_X(\nu)$ and let $\mfq:=k[Z]\cap \mfp$.
Then $\mfq:=\cent_X(\mu)$.
Let $K$ be the residue field of $\beta$
and $\beta^{\ast}\colon k[Z]\to K\dbr{t}$ be the induced morphism.
Upon enlarging $K$, one may assume that the natural extension
$\kappa(\mfq)\to K$ factors through $\kappa(\mfq)\to \kappa(\mfp)$
Then $\mfp$ and the extension $\kappa(\mfp)\to K$ define a $K$-point
of $X$ whose image by $f$ is the $K$-point of $Z$ defined by
$(t\mapsto 0)\circ \beta^{\ast}$.

Since $\arc(X)=\arc(Z)\times_Z X$, the above data define a $K$-arc
$\alpha^{\ast}\colon k[X]\to K\dbr{t}$ such that $\alpha^{\ast}\theta^{\ast}=\beta^{\ast}$ 
and the kernel of $(t\mapsto 0)\circ \alpha^{\ast}$ is $\mfp$. This
defines $\alpha\in \arc(X)$ such that $\arc(\theta)(\alpha)=\beta$
and the center of $\ord_{\alpha}$ on $X$ is $\mfp$.
\end{proof}

\begin{lemma}\label{lemma-etale-lifting-wedges}
Let $\theta\colon X\to Z$ be an étale morphism of affine algebraic
$k$-varieties. Let $K$ be an extension of $k$, and $w$ be a
$K$-wedge on $Z$, whose special arc lifts to an arc $\alpha\in \arc(X)$.
Then there exist an extension $L$ of $K$ and a $L$-wedge $\wt{w}$
on $X$ lifting $w$ and whose special arc is $\alpha$.
\end{lemma}
\begin{proof}
Since $\theta\colon X\to Z$ is étale, by \cite[Proposition 3.7.1]{MR2559866}, for any extension $L$ of $k$, one has a bijection
  $\arc(X)(L\dbr{t})=\arc(Y)(L\dbr{t})\times_{\arc(Y)(L)}\arc(X)(L)$
such that the natural map $\arc(X)(L\dbr{t})\to \arc(X)(L)$ sending a
$L$-wedge on $X$ to its $L$-special arc is induced by the second
projection. The result follows.
\end{proof}

\section{Equivariant resolutions of $\TT$-varieties of complexity one}\label{sec:eqres}
In this section, by studying the equivariant desingularizations of a
non-rational variety with complexity-one torus action, we obtain some information on the location
of essential valuations on the hypercone.

\subsection{Luna's Slice Theorem}\label{subsection-Luna}
We recall a consequence of Luna's Slice Theorem. 
\begin{lemma}\cite[Page 98, III.1, Corollaire 2]{Lun73} \label{lemma-slice}
Let $V$ be an affine variety with an algebraic action of a reductive group $G$. Assume that the ring of invariants $k[V]^{G}$
is $k$. Then $V$ has a unique closed orbit $G\cdot x$ and
one has a homogeneous fiber space decomposition
$V=G\times^{G_{x}} W$, where $G_{x}$
is the isotropy group at $x$ and $W$ is an affine $G_{x}$-variety
having a unique closed orbit which is a fixed point $y$. If further
$W$ is smooth at $y$, then $W$ is $G_{x}$-equivariantly isomorphic to $T_{y}W$.
\end{lemma}
We now fix a smooth projective curve $Y$ and denote by  $\genus(Y)$ its genus.
\begin{proposition}\label{proposition-smoothness}
Let $\PD$ be a $p$-divisor over $(Y, N)$ with complete locus. If $X(\PD)$ is a $d$-dimensional smooth variety, then
$X(\PD)$ is isomorphic to $\bG_{m}^{r}\times \bA^{d-r}$ for $0\leq r\leq d$. In particular, the curve $Y$ is isomorphic 
to the projective line. 
\end{proposition}
\begin{proof}
In order to have $X(\PD) \simeq \bG_{m}^{r}\times \bA^{d-r}$ for $0\leq r\leq d$, we need  to prove that $X(\PD)$ is a toric variety. Indeed, by Lemma \ref{lemma-slice} we have a homogeneous fiber space decomposition
$X(\PD) =  \TT\times^{\TT_{x}} W$, where $\TT_{x}$ acts with a unique
closed orbit which is a fixed point $y$.
Note 
that the natural projection $\TT\times^{\TT_{x}}W\rightarrow \TT/\TT_{x}$ is a locally trivial fibration for the \'etale topology. So $W$ is smooth and therefore by \emph{loc. cit.} we may identify $W$ with the $\TT_{x}$-variety $T_{y}W$.
Let $\mathbb{D}$ be a maximal torus containing the image of $\TT_{x}$ in  ${\rm GL}(T_{y}W)$ by the linear action. Set $\bG:= \TT\times \mathbb{D}$.
Then the $\bG$-action on $\TT\times W$ given by $(g,h)\cdot (t, w) =  (g\cdot t, h\cdot w)$ for all $g, t\in \TT$, $h\in \mathbb{D}$ and $w\in W$ descends to a $\bG$-action on $X(\PD)$ with an open orbit. This implies that $X(\PD)$ is toric. Finally, 
since $ X(\PD)$ is a rational variety and the  rational quotient $q: X(\PD)\dashrightarrow Y$ is dominant,
the curve $Y$ is isomorphic to $\mathbb{P}^{1}$  by  L\"uroth's Theorem. This proves the proposition. 
\end{proof}
\begin{corollary}\label{corollary-positive-genus}
Let $\E$ be a divisorial fan over $(Y, N)$. Assume that $X =  X(\E)$ is smooth and that $\genus(Y)>0$. Then for any $\PD\in \E$
the curve $\Loc(\PD)$ is affine. 
\end{corollary}
\begin{proof}
If there is $\PD\in \E$ with complete locus, then we would have $Y\simeq \mathbb{P}^{1}$ from Proposition \ref{proposition-smoothness}, contradicting the assumption $\genus(Y)>0$. 
\end{proof}
\subsection{Singular locus}
We now fix a smooth projective curve
$Y$ of genus $\geq 1$. Our goal is to give a combinatorial
description of the singular locus of $X(\E)$ for $\E$ a divisorial fan over $Y$. 
\begin{definition}\label{DefinitionSingularLocus}
  Let $\PD$ be a $p$-divisor with locus a Zariski open dense subset of
  $Y$.  We say that a hyperface $\theta$ of $\hypercone(\PD)$ is of
  \emph{orbit type} if $\theta$ satisfies Condition $(i)$ of
  Definition \ref{Definition-Hyperface} or $\theta$ satisfies
  Condition $(ii)$ of \emph{loc. cit.} with the extra property that
  $\theta \not\subset N_{\mathbb{Q}}$.  Note that the associated
  $\TT$-stable closed subset $\cZ(\theta)$ is an orbit closure if and
  only if $\theta$ is of orbit type.

We denote by $\hypercone(\PD)_{\sing}^{\star}$ the set of hyperfaces $\theta\subset \hypercone(\PD)$ such that
$\theta$ is a non-smooth cone whenever $\theta$ satisfies Condition $(ii)$ of \emph{loc. cit.}. Finally for any divisorial fan $\E$ over $Y$ we set
\[\hypercone(\E)_{\sing}^{\star} =  \bigcup_{\PD \in \E} \hypercone(\PD)_{\sing}^{\star}.\]
\end{definition}
\begin{rema}\label{rem:singloc:toroidal}
Assume that for any $\PD\in\E$, $\Loc(\PD)$ is affine (Recall that in
this case only Condition $(ii)$ of Definition
\ref{Definition-Hyperface} can be satisfied.)
Then $X=X(\E)$ is toroidal and using Lemma \ref{lemma-toroidal} and the classical
description of the singular locus of a toric variety, one sees that
\[X^{\sing} =  \bigcup_{\theta\in \hypercone(\E)_{\sing}^{\star}} \cZ(\theta).\]
Note that in this case $\theta\in \hypercone(\E)_{\sing}^{\star}$ if and
only if $\theta$ is a non-smooth face of a Cayley cone of an element
of $\E$. Note also that the above holds even if $Y$ is a rational curve.
\end{rema}
\begin{proposition}\label{prop:singloc}
Let $\E$ be a divisorial fan over a smooth projective curve $Y$
of genus $\geq 1$. Then the singular locus of $X = X(\E)$ is given
by the formula
\[X^{\sing} =  \bigcup_{\theta\in \hypercone(\E)_{\sing}^{\star}} \cZ(\theta).\]
\end{proposition}
\begin{proof}
Let $\pi\colon\widetilde{X}\rightarrow X$ be the toroidification and
let $\theta\in \Hfan(\E)$ be of orbit type. Consider the open orbit
$O(\theta)\subset \cZ(\theta)$. If $\theta$ satisfies Condition
$(ii)$ of Definition \ref{Definition-Hyperface}, then $O(\theta)$
identifies with an orbit of $\widetilde{X}\setminus E$, where $E$ is
the exceptional locus of $\pi$.
By Remark \ref{rem:singloc:toroidal},
we deduce that
$O(\theta)\subset X^{\sing}$ if and only if $\theta$ is a non-smooth
cone. Now assume that $\theta$ satisfies Condition $(i)$ of
\emph{loc. cit.}  Then $\theta = \hypercone(\PD')$ for some
$p$-divisor $\PD'$ with locus $Y$. It follows, for instance, from
\cite[Lemma 1.4]{IS11}, that $X(\PD')$ is an open subset of
$X$. Moreover, the hyperface-orbit relations imply that $O(\theta)$
is the unique closed orbit of $X(\PD')$. But Luna's Slice Theorem
(see Proposition \ref{proposition-smoothness}), Lur\"oth's Theorem
and the fact that the genus of $Y$ is $\geq 1$ yield that
$O(\theta)\subset X^{\sing}$, which establishes our
statement.
\end{proof}

\subsection{Equivariant resolutions of singularities and toroidification}
\label{subsec:eqres}
We keep the same notation as in \ref{subsection-Luna}.
Strictly speaking, the next proposition is not needed in order to
establish the main result of this section. It has however its own interest,
and explains in some sense why the set of essential valuations of a
non-rational $\TT$-variety of complexity one is rather close to the
set of essential valuations of its toroidification. The situation is
dramatically different in the rational case (see Section \ref{sec:casePone}).

\begin{proposition}\label{prop:equivres}
Let $\E$ be a divisorial fan over $(Y, N)$  and let $X =  X(\E)$ be
the associated $\TT$-variety. Assume that $\genus(Y)> 0$.
Let $\pi\colon\widetilde{X}\rightarrow X$ be the toroidification. 
Consider a smooth variety $X'$  with an algebraic $\TT$-action. 
Then for any $\TT$-equivariant proper birational morphism
$\psi\colon X'\rightarrow X$, there is a $\TT$-equivariant proper birational
morphism $\varsigma\colon X'\rightarrow \widetilde{X}$ such that $\psi  =  \pi \circ \varsigma$, 
\end{proposition}
\begin{proof}
Assume that there is a  $\TT$-equivariant proper birational morphism  $\psi: X'\rightarrow X$. 
Note that this implies that the $\TT$-action on $X'$ is of complexity one. 
Let $\E'$ be a divisorial fan over $(Y', N)$ such that $X' =  X(\E')$, where $Y'$ is a smooth projective curve. Since $\psi^{\star}$ maps isomorphically $k(X)^{\TT} =  k(Y)$  to $k(X')^{\TT} =  k(Y')$ we see that $\psi$ induces an isomorphism $\gamma: Y'\rightarrow Y$. 
In particular, $\genus(Y')> 0$ and by Corollary \ref{corollary-positive-genus} any element of $\E'$ has affine locus. So
the natural invariant maps $X(\PD)\rightarrow Y'$ for $\PD\in \E'$ glue together into a global quotient $q': X'\rightarrow Y'$. 
Set $q_{0} =  \gamma\circ q'$. Let $\mathcal{G}\subset X\times Y$ be the graph of the rational quotient $q: X\dashrightarrow Y$. 
Since $q\circ \psi =  q_{0}$, the morphism
$$s: X'\rightarrow \mathcal{G},\,\, x\mapsto (\psi(x), q_{0}(x))$$
is well defined, $\TT$-equivariant and birational (as $\psi$ is birational). The composition of $s$ with the projection from $\mathcal{G}$ to $X$
is the proper morphism $\psi$. So $s$ is proper.
The universal property of the normalization applied to $s$ and
proposition \ref{VolTor} then give the existence of a $\TT$-equivariant proper birational morphism $\varsigma : X'\rightarrow \widetilde{X}$ satisfying $\psi =  \pi\circ \varsigma$. 
\end{proof}
\subsection{Smooth refinements with respect to a polyhedron}
\label{sec:equiv-valu-which}
This section contains the technical tools needed to establish the
analog of Proposition \ref{prop:nu:not:exc} (which locates the essential
valuations in the toric case) in the case of $X(\PD)$ where $\PD$ is a polyhedral divisor over
a smooth projective curve of positive genus.
Proposition \ref{prop:nu:not:exc} itself is a crucial ingredient, but more work is
needed, in particular due to the fact that there may exist non-Nash
valuations on $X$ which become Nash on the toroidification and to
which Proposition \ref{prop:nu:not:exc}, in some sense, no longer
applies (see Remark \ref{rema:nash:tildex:not:x}).

In the whole subsection, unless otherwise specified, we consider the following setting:
let \(N\) be a lattice, \(\cC\subset N_{\bQ}\) be a strictly convex polyhedral cone in \(N_{\bQ}\),
\(\sigma\) be a face of \(\cC\) and \(\cP\subset \sigma\) be a
\(\sigma\)-tailed polyhedron.
Being given a fan $\Sigma$ refining $\cC$,
the set of cones of \(\Sigma\) which are contained in 
\(\sigma\) is a fan refining
$\sigma$, called the fan induced by $\Sigma$ on $\sigma$ and denoted
by $\Sigma_{|_{\sigma}}$.
\begin{nota}\label{nota:sigmanu}
Being given a fan $\Sigma$ and $\nu$ an element of the support of
$\Sigma$, we denote by $\Sigma(\nu)$ the unique cone of $\Sigma$ whose
relative interior contains $\nu$.  
\end{nota}

One may compare the following definition with definition \ref{defi:strong:economic}.
\begin{defi}\label{defi:P:strong:economic}
Let $\Sigma$ be a smooth fan refining $\cC$.
\begin{enumerate}
\item We say that $\Sigma$ is a \emph{$\cP$-economical refinement} of $\cC$
if any smooth face of $\cC$ which does not meet $\cP$ is a cone of
$\Sigma$.
\item
We say that $\Sigma$ is a \emph{$\cP$-big refinement} of $\cC$ if:
\begin{itemize}
\item any cone of $\Sigma$ which is not a face of $\cC$ and does not
 meet $\cP$ has a ray which is not a ray of $\cC$;
\item any cone of $\Sigma$ which meets $\cP$ has a ray which meets
 $\cP$ or is not a ray of $\cC$.
\end{itemize}
\end{enumerate}
\end{defi}
\begin{lemm}\label{lemm:Pstrong:refin:bis}
Let $\Sigma$ be a smooth refinement of $\cC$. Assume that the fan
$\Sigma_{|_{\sigma}}$
induced by $\Sigma$ on $\sigma$ is a big
(resp. $\cP$-big) refinement of $\sigma$. Then there exists a
refinement $\Sigma'$ of $\Sigma$ which is a smooth big (resp. $\cP$-big) refinement of $\Sigma$. 
If moreover $\Sigma$ is a smooth economical (resp. $\cP$-economical)
refinement of $\cC$, then $\Sigma'$ may be chosen as a smooth economical
(resp. $\cP$-economical) refinement of $\cC$.
\end{lemm}
\begin{proof}
Let us call special cone of $\Sigma$ any cone $\tau$ of
$\Sigma$ which is not a face of $\cC$ and such that every ray of $\tau$ is a ray of $\cC$.
If $\Sigma_{|_{\sigma}}$
is a big refinement of $\sigma$, any special
cone $\tau$ of $\Sigma$ is not contained in $\sigma$. Thus the sum
$n_{\tau}$ of the primitive generators of the ray of $\tau$ is not a
ray of $\cC$ and does not belong to $\sigma$. Therefore, the star
subdivision $\St(\Sigma,n_{\tau})$ of $\Sigma$ with respect to $n_{\tau}$ is
smooth, has strictly less special cones than $\Sigma$, and satisfies
$\St(\Sigma,n_{\tau})_{|_{\sigma}}=\Sigma_{|_{\sigma}}$. Moreover any smooth face of $\cC$ which is
a cone of $\Sigma$ cannot have $\tau$ as a face, and is thus a cone of $\St(\Sigma,n_{\tau})$.
Thus we obtain a fan $\Sigma'$ as in the statement after a finite number of such star subdivisions.

Let us now call $\cP$-special cone of $\Sigma$
any cone $\tau$ of $\Sigma$ which is either a special cone of $\Sigma$
which does not meet $\cP$, or meets $\cP$ and any ray of $\tau$ is
either a ray of $\cC$ or meets $\cP$ ; note that in the latter case
$\tau\cap \sigma$ has the same properties.
Thus if $\Sigma_{|_{\sigma}}$ is a $\cP$-big refinement of $\sigma$, any $\cP$-special
cone $\tau$ of $\Sigma$ is a special cone of $\Sigma$ which is not
contained in $\sigma$ and does not meet $\cP$. In particular, with the same notation
as before, $n_{\tau}$ is not a ray of $\cC$ and does not meet $\cP$,
and $\St(\Sigma,n_{\tau})$ is a smooth refinement of
$\Sigma$ that has strictly less $\cP$-special cones than $\Sigma$, and
satisfies
$\St(\Sigma,n_{\tau})_{|_{\sigma}}=\Sigma_{|_{\sigma}}$. Moreover any smooth face of $\cC$ which
does not meet $\cP$ and is
a cone of $\Sigma$ cannot have $\tau$ as a face, and is
thus a cone of $\St(\Sigma,n_{\tau})$.
We conclude as before.
\end{proof}

\begin{lemm}\label{lemm:Pstrong:refin}
Let $\Sigma$ be a smooth $\cP$-economical refinement of $\cC$.
Let $\nu\in \cC\cap N$ and $\tau$ be the unique face of $\cC$ whose
relative interior contains $\nu$.
\begin{enumerate}
\item \label{lemm:item:i:Pstrong:refin}
Assume that any cone of $\Sigma$ which is not a face of $\cC$ and does not
meet $\cP$ has a ray which is not a ray of $\cC$.
\begin{enumerate}
\item \label{lemm:item:ia:Pstrong:refin}
Then there exists a smooth star refinement $\Sigma'$ of $\Sigma$ which
is a $\cP$-economical and $\cP$-big refinement of $\cC$,
and such that every smooth cone of $\Sigma$ which does not meet $\cP$
is a cone of $\Sigma'$.
\item \label{lemm:item:ib:Pstrong:refin}
Assume moreover that $\dim(\tau)\geq 2$, and $\tau$ has a ray which
meets $\cP$
or is not a ray of $\cC$. Then one has the same conclusion as in
\eqref{lemm:item:ia:Pstrong:refin}, with the additional condition
that the cone $\Sigma'(\nu)$ (see notation \ref{nota:sigmanu}) has dimension $\geq 2$ and meets $\cP$.
\end{enumerate}
\item \label{lemm:item:ii:Pstrong:refin}
Assume that $\dim(\tau)\geq 2$, $\tau$ has a ray which meets $\cP$,
and every face of $\tau$ is either a face of $\cC$ or meets $\cP$.
Then one has the same conclusion as in \eqref{lemm:item:ia:Pstrong:refin},
with the additional condition
that the cone $\Sigma'(\nu)$ has dimension $\geq 2$ and meets $\cP$.
\end{enumerate}
\end{lemm}
\begin{proof}
Let us call $1$-special cone of $\Sigma$ any cone $\tau'$ of
$\Sigma$ which is not a face of $\cC$, does not
meet $\cP$ and such that every ray of $\tau'$ is a ray of $\cC$,
and $2$-special cone of $\Sigma$
any cone $\tau'$ of $\Sigma$ which meets $\cP$ and such that no ray of $\tau'$ meets $\cP$.
In particular, if $\Sigma$ has no $1$-special nor $2$-special cone,
$\Sigma$ is a big $\cP$-refinement of $\cC$.

Let us show \eqref{lemm:item:ia:Pstrong:refin}.
By assumption, $\Sigma$ is a $\cP$-economical refinement of $\cC$
and has no $1$-special cone.
Assume that $\Sigma$ has a $2$-special cone, and
let $\tau'$ be a minimal $2$-special cone of $\Sigma$.
Consider the fan $\St(\Sigma,n_{\tau'})$ obtained from $\Sigma$ by the star subdivision with
respect to the sum $n_{\tau'}$ of the primitive generators of the rays of $\tau'$.

By the minimality of $\tau'$, one has $\bQ_{\geq 0}n_{\tau'}\cap \cP\neq \vide$.
Since $\bQ_{\geq 0}n_{\tau'}$ is a ray of any cone of $\St(\Sigma,n_{\tau'})\setminus\Sigma$
and $\tau'\notin \St(\Sigma,n_{\tau'})$, $\St(\Sigma,n_{\tau'})$ has strictly less $2$-special cones than $\Sigma$, and
has no $1$-special cone.
Moreover any cone of $\Sigma$ which does not meet $\cP$ cannot
have $\tau'$ as a face, and is therefore an element of $\St(\Sigma,n_{\tau'})$. 

Therefore, after a finite number of applications of such star
subdivisions, we end up with a fan $\Sigma'$ which is 
a smooth $\cP$-economical and $\cP$-big refinement of $\cC$, and
such that any cone $\Sigma$ which does not meet $\cP$ is a cone of $\Sigma'$.

Let us show \eqref{lemm:item:ib:Pstrong:refin}.

First assume that $\tau$ has a ray which meets $\cP$. 
Let $\tau''$ be the cone of the above described fan $\St(\Sigma,n_{\tau'})$ such that
$\nu\in\Relint(\tau'')$. It suffices to show that $\tau''$ has
dimension $\geq 2$ and has a ray which meets $\cP$. This is clear if $\tau'$ is not a
face of $\tau$, since in this case $\tau\in \St(\Sigma,n_{\tau'})$.
Assume now that $\tau'$ is a face of $\tau$.
Since no ray of $\tau'$ meets $\cP$, $\tau'$ is a proper face of $\tau$.
Therefore, since $\nu\in\Relint(\tau)$, $n_{\tau'}$ is a ray of
$\tau''$ and $\dim(\tau'')\geq 2$. 

Assume now that $\tau$ has a ray $\rho$ which is not a ray of $\cC$. 
By the procedure described in the proof of
\eqref{lemm:item:ia:Pstrong:refin}, one may assume that any minimal
$2$-special cone of $\Sigma$ is a face of $\tau$. If each of these
cones has a ray which is not a ray of $\cC$, then any $2$-special cone
of $\Sigma$ has a ray which is not a ray of $\cC$, thus $\Sigma$ is a
big $\cP$-refinement of $\cC$ and we are done. Otherwise, there
exists a face $\tau'$ of $\tau$ which is a minimal $2$-special cone 
of $\Sigma$ and $\rho$ is not a ray of $\tau'$. Thus, with the same
notation as before, one has $\dim(\tau'')\geq 2$ and $\rho$ is a ray
of $\tau''$ which is not a ray of $\cC$.

Let us now show \eqref{lemm:item:ii:Pstrong:refin}. By
\eqref{lemm:item:ib:Pstrong:refin}, it suffices to construct, using
star subdivisions, a smooth refinement $\Sigma_1$ of $\Sigma$ which
has no $1$-special $1$-cone, is a $\cP$-economical refinement of
$\cC$, and contains $\tau$.

Let $\tau'$ be a $1$-special cone of $\Sigma$
and consider the fan $\St(\Sigma,n_{\tau'})$ obtained from $\Sigma$ by the star subdivision with
respect to the sum $n_{\tau'}$ of the primitive generators of the rays of $\tau'$.
Since $\dim(\tau')\geq 2$, $\bQ_{\geq 0}n_{\tau'}$ is not a ray of $\cC$.
Since $\tau'\notin \Sigma_1$ and $\bQ_{\geq 0}n_{\tau'}$ is a ray of any element of $\Sigma_1\setminus
\Sigma$, $\Sigma_1$ has strictly less $1$-special cones than $\Sigma$.
On the other hand, by the assumptions on $\tau$ and $\tau'$, $\tau'$
can not be a face of $\tau$, thus $\tau\in \Sigma_1$.
Similarly, any smooth face of $\cC$ which does not intersect $\cP$ is a cone of $\St(\Sigma,n_{\tau'})$.
Therefore, after a finite number of applications of such star
subdivisions, we end up with a fan $\Sigma_1$ with the desired properties.
\end{proof}

\begin{lemm}\label{lemm:twodim}
Let \(N\) be a two-dimensional lattice, \(\gamma\) be a polyhedral
strictly convex full-dimensional cone of $N_{\bQ}$, \(\nu_0\) and
\(\nu_1\) be the primitive generators of the rays of $\gamma$.
Assume that \(\gamma\) is \emph{not} smooth. Let \(\nu\in \Min(\Relint(\gamma),\leq_{\gamma})\).
Then \(\nu\notin \nu_0+\gamma\).
\end{lemm}
\begin{proof}
One may find a a \(\bZ\)-basis \((e_1,e_2)\) of $N$ such that \(\nu_0=e_2\) and \(\nu_1=de_1-ke_2\),
\(d\), \(k\) are coprime positive integers with \(k<d\) and \(\gcd(d,k)=1\).
Thus \(\gamma^{\vee}\) is generated by \(e_1^{\vee}\) and \(ke_1^{\vee}+de_2^{\vee}\)
Hereafter, we use the description of the minimal resolution of
singularities of a toric surface in terms of Hirzebruch-Jung continued
fractions and the notation of \cite[Chapter 10.2]{CLS11}.
Any element of \(\Min(\Relint(\gamma),\leq_{\gamma})\) may be written as
\(u_i=P_{i}e_1-Q_{i}e_2\), \(i\geq 1\). In order to show that \(u_i\notin
\nu_0+\gamma\), it suffices to show that \(kP_{i}-(d+1)Q_{i}<0\).
For any \(i\), one checks that \(k_iQ_{i+1}-k_{i+1}Q_i=k\).
Since  \(\frac{P_i}{Q_i}-\frac{P_{i+1}}{Q_{i+1}}=\frac{1}{Q_iQ_{i+1}}\), one
infers that \(\frac{P_i}{Q_i}=\frac{d}{k}+\frac{k_i}{kQ_i}\)
thus
\[
\frac{P_i}{Q_i}=\frac{d}{k}+\frac{d}{kQ_i}+\frac{k_i-d}{kQ_i}
\]
Since \(k_i<k<d\), this shows that \(kP_{i}-(d+1)Q_{i}<0\).
\end{proof}

\begin{prop}\label{prop:gentor}
Let \(N\) be a lattice, \(\cC\subset N_{\bQ}\) be a strictly convex polyhedral cone in \(N_{\bQ}\),
\(\sigma\) be a face of \(\cC\) and \(\cP\subset \sigma\) be a \(\sigma\)-tailed polyhedron.
Recall that we denote by \(\cC_{\sing}\) the union of the relative interiors
of the non-smooth faces of \(\cC\). 

Let \(\cC^{\star}\) be the reunion of \(\cC_{\sing}\) with the union
of the relative interiors of the faces of \(\cC\) which intersect \(\cP\).

Let \(\nu\) be a primitive element of \((\cC^{\star}\cap N)\setminus \Min(\cC^{\star}\cap N,\leq_{\cC})\)

Then there exists a smooth fan \(\Sigma\) 
which is a smooth $\cP$-economical and $\cP$-big star
refinement of $\cC$ and such that $\Sigma(\nu)$ has dimension $\geq 2$ and either meets $\cP$ or is not a cone
of \(\cC\).
\end{prop}
\begin{rema}\label{rema:nash:tildex:not:x}
In the proof, we shall consider separately the three following cases:
\begin{enumerate}
\item
 \(\nu\in (\cC_{\sing}\cap N)\setminus \Min(\cC_{\sing}\cap N,\leq_{\cC})\);
\item \(\nu\in (\cC^{\star}\cap N)\setminus \cC_{\sing}\);
\item \(\nu\in \Min(\cC_{\sing}\cap N,\leq_{\cC})\).
\end{enumerate}
In the first case, one essentially reduces easily to the toric case,
\ie to Proposition \ref{prop:nu:not:exc} with the aid of the
previous lemmas. The third case is the most challenging.
It corresponds geometrically to the case of a $p$-divisor $\PD$ with
locus a smooth projective curve of positive genus such that there are
Nash valuations on the toroidification $\widetilde{X}$ of $X(\PD)$
which are not Nash valuations on $X(\PD)$ (see Section
\ref{sec:nash}). See Section \ref{subsec:nash:tildex:not:x} below for an
explicit example.
\end{rema}
\begin{proof}
First assume that \(\nu\in (\cC_{\sing}\cap N)\setminus \Min(\cC_{\sing}\cap N,\leq_{\cC})\).
By Proposition \ref{prop:nu:not:exc}, there exists a star refinement $\Sigma$ of $\cC$ 
such that:
\begin{itemize}
\item $\Sigma$ is a smooth economical and big refinement of $\cC$;
\item the cone $\tau$ of $\Sigma$ such that $\nu\in \Relint(\tau)$ is not a face of $\cC$ and has dimension $\geq 2$.
\end{itemize}
Since $\Sigma$ is a smooth economical and big refinement of $\cC$,
$\tau$ has a ray which is not a ray of $\cC$, and we may apply lemma
\ref{lemm:Pstrong:refin}\eqref{lemm:item:ia:Pstrong:refin} in case
$\tau\cap \cP=\vide$,
and lemma \ref{lemm:Pstrong:refin}\eqref{lemm:item:ib:Pstrong:refin}.
in case $\tau\cap \cP\neq \vide$.

Assume now that \(\nu\in (\cC^{\star}\cap N)\setminus \cC_{\sing}\).
Let $\tau$ be the unique face of $\cC$ whose relative interior
contains $\cC$.
Thus $\tau$ is a smooth face of $\cC$ which meets $\cP$.
Since $\nu\notin\Min(\cC^{\star}\cap N,\leq_{\cC})$, one has
$\dim(\tau)\geq 2$. Let $\tau'$ be a minimal face of $\tau$ meeting $\cP$
and $n_{\tau'}$ be the sum of the primitive elements generating the rays of $\tau'$.
In particular $\bQ_{\geq 0}n_{\tau'}\cap \cP\neq \vide$ and since 
$\nu\notin\Min(\cC^{\star}\cap N,\leq_{\cC})$, one has $n_{\tau'}\neq \nu$.
Let $\Sigma_1$ be the fan obtained from $\cC$ by the star subdivision
with respect to $n_{\tau'}$. Then any smooth cone of $\cC$ which does
not meet $\cP$ is a cone of $\Sigma_1$. Moreover, $\Sigma_1(\nu)$
is a smooth cone of  $\Sigma_1$ containing $\nu$ in its relative interior, has 
$\bQ_{\geq 0}n_{\tau'}$ as a ray, and such that any face of $\Sigma_1(\nu)$
is either a face of $\cC$ or contains $\bQ_{\geq 0}n_{\tau'}$, thus meets $\cP$.
Let $\Sigma$ be a smooth economical star refinement of
$\Sigma_1$. In particular, $\Sigma$ is a smooth $\cP$-economical
refinement of $\cC$. Now one may apply lemma \ref{lemm:Pstrong:refin}\eqref{lemm:item:ii:Pstrong:refin}.

We now assume that $\nu\in \Min(\cC_{\sing}\cap N,\leq_{\cC})$.  We
are going to show that there exists a star subdivision of $\cC$ such
that the resulting fan $\Sigma_0$ is such that the cone $\Sigma_0(\nu)$ is
smooth, has dimension $\geq 2$ and each of its face is either a face
of $\cC$ or meets $\cP$.
Take this for granted for the moment. Then, by Remark
\ref{rem:prop:nu:not:exc}, there exists a fan \(\Sigma\) which is a
smooth economical star refinement of $\Sigma_0$. In particular, the cone
$\Sigma_0(\nu)$ and any smooth face of $\cC$ which does not meet $\cP$
are elements of $\Sigma$ and one may apply lemma \ref{lemm:Pstrong:refin}\eqref{lemm:item:ii:Pstrong:refin} again.

Let us now explain the construction of $\Sigma_0$.
Let \(\nu_0\in \Min(\cC^{\star}\cap N,\leq_{\cC})\)
such that \(\nu_0\leq_{\cC}\nu\) and \(\tau_0\) be the face of
\(\cC\) containing \(\nu_0\) in its relative interior. By
assumption, \(\nu_0\) does not belong to \(\cC_{\sing}\cap N\). Thus \(\tau_0\)
is a smooth face of \(\sigma\) which intersects \(\cD\). Since
\(\nu_0\in \Min(\cC^{\star}\cap N,\leq_{\cC})\) and \(\tau_0\) is smooth,
\(\nu_0\) is the sum of the primitive elements generating the rays
of \(\tau_0\).

Let \(P\) be the vector plane generated by \(\nu_0\) and \(\nu\),
and \(\gamma\) be the two-dimensional cone \(\cC\cap P\). Since
\(\tau_0\) is a face of \(\cC\) not containing \(\nu\),
\(\tau_0\cap P\) is a proper face of \(\gamma\), thus \(\nu_0\)
generates one of the rays of \(\gamma\). Moreover, since
\(\nu_0\leq_{\cC}\nu\), one has \(\nu_0\leq_{\gamma} \nu\). In
particular, \(\nu\) lies in the relative interior of \(\gamma\);
otherwise \(\nu\) and \(\nu_0\) would be collinear, but this would
contradict \(\nu\notin \tau_0\)).

Note that any element \(\nu'\) of \(\Relint(\gamma)\cap N\) lies in
\(\cC_{\sing}\cap N\).  Otherwise, \(\nu'\) would belong to a smooth face \(\tau'\)
of \(\cC\) such that \(\tau'\cap P=\gamma\), but since
\(\nu\in \gamma\cap \cC_{\sing}\cap N\), one has a contradiction.

Since \(\nu\in \Min(\cC_{\sing}\cap N,\leq_{\cC})\) we infer that
\(\nu\in \Min(\Relint(\gamma),\leq_{\gamma})\). Since
\(\nu_0\leq_{\gamma}\nu\), lemma \ref{lemm:twodim} shows that \(\gamma\) is
a \(P\cap N\)-smooth cone, thus (since \(N\cap P\) is a saturated
submodule of \(N\)), also a \(N\)-smooth cone. Let \(\nu_1\) be the
primitive generator of the other ray of \(\gamma\). Note that since
\(\tau_0\cap P=\bQ_{\geq 0}\nu_0\), \(\nu_1\) does not belong to the
vector space generated by \(\tau_0\).

Note also that since \(\nu\in \Min(\Relint(\gamma),\leq_{\gamma})\),
and $\gamma$ is smooth, one has \(\nu=\nu_0+\nu_1\). Let \(\tau_1\) be the face of \(\cC\)
whose relative interior contains \(\nu_1\). Since
\(\nu_1\leq_{\cC}\nu\) and \(\nu\in \Min(\cC_{\sing}\cap N,\leq_{\cC})\)
\(\tau_1\) is smooth. Note that \(\nu_0\) does not belong to the
vector space \(\Span(tau_1)\) generated by $\tau_1$.
Otherwise, one would have \(\nu_0\in \Span(\tau_1)\cap \cC=\tau_1\)
thus \(\nu\in \tau_1\), which would contradict \(\nu\in \cC_{\sing}\cap N\).

Let us show that the cone \(\tau\) generated by \(\tau_1\) and
\(\nu_0\) is smooth. Note that this is a simplicial cone,
containing \(\nu\) in its relative interior. Let \(\eta\) be the
face of \(\cC\) containing \(\nu\) in its relative interior. In
particular \(\eta\) is not a smooth cone. Note that \(\tau_1\) and
\(\tau_0\) are faces of the cone \(\eta\), which therefore contains
\(\tau\). Since \(\Relint(\tau)\cap \Relint(\eta)\neq \vide\), one
has \(\Relint(\tau)\subset \Relint(\eta)\).

Denote by \(e_1,\dots,e_r\) the primitive generators of the rays
of $\tau_1$ and assume that \(\tau\) is not smooth. Then by
\cite[Proposition 11.1.8]{CLS11} there exists
\((\lambda_i)_{0\leq i\leq r}\in ]0,1]\in {\bQ^{r+1}}\) such that
\[
\nu_3:=\lambda_0\nu_0+\sum_{i=1}^r \lambda_ie_i\in \tau \cap N
\]
and at least one \(\lambda_i\) is \(<1\).
Since \(\nu=\nu_0+\nu_1=\nu_0+\sum_{i=1}^r\lambda'_ie_i\) with
$\lambda'_i\geq 1$ for every $i$, one deduces that
\(\nu_3\leq_{\cC}\nu\) and \(\nu_3\neq \nu\). But
\(\nu_3\in \Relint(\tau)\), thus \(\nu_3\in \Relint(\gamma)\).
This contradicts the fact that \(\nu\in \Min(\Relint(\gamma),\leq_{\gamma})\).
Thus $\tau$ is smooth.

Now let \(\Sigma_0\) be the star subdivision of \(\cC\) with respect to
  \(\nu_0\). Note that \(\tau\) is a smooth cone of \(\Sigma_0\),
  containing \(\nu\) in its relative interior, and \(\bQ_{\leq 0}\nu_0\)
  is a ray of \(\tau\) which intersects $\cP$.  Moreover, since \(\bQ_{\leq}\nu_0\cap \cP\neq \vide\),
  any face of \(\cC\) which does not intersect $\cP$ does not contain
  $\nu_0$, thus is a cone of $\Sigma_0$.

On the other hand, by the construction of $\tau$, each face of $\tau$ is either a face of
$\tau_0$, thus a face of $\cC$, or contains $\nu_0$, thus intersects $\cP$.

\end{proof}
\subsection{Location of essential valuations}
\label{subsec:prop:T:ess}
In the whole subsection, unless otherwise specified, we consider the
following setting and notation. Let \(\PD\) be a \(\sigma\)-tailed \(p\)-divisor over \((Y,N)\) where
$Y$ is smooth projective curve. Let $\{y_1,\dots,y_r\}$ be a finite
set of points of $Y$ such that $\Supp(\PD)\subset \{y_1,\dots,y_r\}$. For $1\leq i\leq r$,
set \(U_i:=\Loc(\PD)\setminus \{y_1,\dots,y_{i-1},y_{i+1},\dots,y_r\}\).
Let \(\widetilde{\PD}\) be the toroidal divisorial fan over \((Y,N)\)
generated by the \(p\)-divisors $\{\PD_{|_{U_i}}\}_{1\leq i\leq r}$. In particular
$X(\widetilde{\PD})$ is the toroidification $\widetilde{X}$ of $X:=X(\PD)$.
Note that for any $1\leq i\leq r$, if $\Sigma$ is a smooth (resp. smooth economical, resp. big)
refinement of $\cayley_{y_i}(\PD)$, then $\Sigma_{|_{\sigma}}$
is a smooth (resp. smooth economical, resp. big) refinement of $\sigma$.
\begin{lemma}\label{lemm:divisorial:refinement}
Let \(\Sigma_1,\dots,\Sigma_r\) be fans refining respectively
\(\cayley_{y_{1}}(\PD), \ldots, \cayley_{y_{r}}(\PD)\)
and inducing the same fan $\Sigma(\sigma)$ on the tail \(\sigma\). 
\begin{enumerate}
\item
There exists a toroidal divisorial fan
  \(\DF=\DF(\Sigma_1,\dots,\Sigma_r)\) over \((Y,N)\) which is a
  refinement of \(\widetilde{\PD}\) and induces the fan \(\Sigma_i\) on
  \(\cayley_{y_{i}}(\PD)\) for every \(1\leq i\leq r\).  We denote by
  \(f\colon X(\DF)\to \widetilde{X}\) the induced proper birational
  equivariant morphism.
\item 
Consider the following sets of cones of $\Hfan(\E)$, ordered by
inclusion:
\[
  \Sigma_{i,\exc}(\widetilde{X})=\{\tau\in \Sigma_i,\quad \tau\nprec
  \cayley_{y_{i}}(\PD)\}, \quad 1\leq i\leq r,
\]
\[
  \DF_{\exc}(\widetilde{X}):=\cup_{i=1}^r
  \Sigma_{i,\exc}(\widetilde{X})
\]
\[
  \DF_{\exc}(X):=\DF_{\exc}(\widetilde{X}) \cup \{\tau\in \Sigma(\sigma)
  ,\,\tau\cap \deg(\PD)\neq \vide\}
\]

Let \(\nu\) be a primitive element of \(\hcone(\PD)\cap \Hypz\)
and \(\theta(\DF,\nu)\) be the unique cone of $\Hfan(\E)$ whose relative
interior contains $\nu$. Then \(\nu\) is \(f\)-exceptional
(resp. \(\pi\circ f\) exceptional) if and only if
\(\theta(\DF,\nu)\) is a minimal element of
\(\DF_{\exc}(\widetilde{X})\) (resp. of \(\DF_{\exc}(X))\))
\item
In particular, if each \(\Sigma_i\) is a smooth economical
refinement of \(\cayley_{y_i}(\PD)\), then \(f\) is an equivariant
resolution of singularities of \(\wt{X}\). If in addition
each \(\Sigma_i\) is a big refinement of \(\cayley_{y_i}(\PD)\),
then \(f\) is a divisorial equivariant resolution of singularities of \(\wt{X}\).
\item
Assume that $\Loc(\PD)=Y$ and $Y$ is a smooth projective curve of
positive genus.
If each \(\Sigma_i\) is a smooth $\deg(\PD)$-economical
refinement of \(\cayley_{y_i}(\PD)\), then \(\pi\circ f\) is an equivariant
resolution of singularities of \(X\). If in addition
each \(\Sigma_i\) is a $\deg(\PD)$-big refinement of \(\cayley_{y_i}(\PD)\),
then \(\pi\circ f\) is a divisorial equivariant resolution of singularities of \(X\).
\end{enumerate}
\end{lemma}
\begin{proof}
Let \((\tau_{j})_{j\in J}\) be the maximal cones of the fan with support \(\sigma\) induced by the
\(\Sigma_i\)'s. For $1\leq i\leq r$ and $j\in J$, let $\Sigma^{(j)}_i$ be the set of 
cones $\gamma\in \Sigma_i$ such that $\gamma \not\subset\sigma$ and $\tau_j$ is a face of $\gamma$.
For any such cone $\gamma$, set
\[
\gamma_{y_i}:=\gamma \cap \{[y_i,a,1]\}_{a\in N_{\bQ}}.
\]
Then one can take for \(\DF\) the divisorial fan generated by the following family of \(p\)-divisors:
\[
\gamma_{y_i}\cdot [y_i]+\sumsubu{z\in U_i\\z\neq y_i} \tau_j\cdot [z],\quad 1\leq i\leq r, \,j\in J, \,\gamma\in \Sigma_i^{(j)}.
\]
The remainder of the proposition is a consequence of \S \ref{sec:prime-invar-cycl}
and propositions \ref{prop:except-locus-toro} and \ref{prop:singloc}.
\end{proof}

\begin{lemm}\label{lemm:fan:extensions}
Let \(\Sigma_1\) be a smooth fan which is a star refinement of \(\cayley_{y_{1}}(\PD)\).
\begin{enumerate}
\item There exist smooth fans
  \(\Sigma_2,\dots,\Sigma_r\) refining
  \(\cayley_{y_{2}}(\PD),\dots,\cayley_{y_r}(\PD)\) respectively 
  and such that for $2\leq i\leq r$, one has
  $(\Sigma_i)_{|_{\sigma}}=(\Sigma_1)_{|_{\sigma}}$.
\item Assume that \(\Sigma_1\) is a big (resp. economical,
  resp. big and economical) refinement
  of \(\cayley_{y_{1}}(\PD)\). Then for any \(2\leq i\leq r\),
  \(\Sigma_i\) may be chosen as a big (resp. economical, resp. big and economical) refinement
  of \(\cayley_{y_{i}}(\PD)\).
\item Same statement as before with ``economical'' and ``big''
  replaced respectively with  ``$\deg(\PD)$-economical'' and ``$\deg(\PD)$-big''.
\end{enumerate}
\end{lemm}
\begin{proof}
Set \(\Sigma(\sigma):=(\Sigma_1)_{|_{\sigma}}\).

Note that if \(\Sigma_1\) is a big (resp. smooth economical) refinement
 of \(\cayley_{y_{1}}(\PD)\), then \(\Sigma(\sigma)\) is a
a big (resp. smooth economical) refinement of \(\sigma\).

Let $n_1,\dots,n_s$ be a finite sequence of elements of \(\cayley_{y_{1}}(\PD)\)
such that \(\Sigma_{1}\) is obtained by applying successive star
subdivisions at $n_1,\dots,n_s$.

For \(r\geq i\geq 2\), consider the fan $\Sigma'_i$ refining \(\cayley_{y_{i}}(\PD)\)
obtained from \(\cayley_{y_{i}}(\PD)\) by applying the following successive operations for $j\in
\{1,\dots,s\}$: if $n_j\in  \sigma$, apply the star subdivision at $n_j$; if $n_j\notin \sigma$, do nothing.

By construction, the fan induced by $\Sigma'_i$ on $\sigma$ is $\Sigma(\sigma)$.
If $\Sigma'_i$ is not smooth, by Remark \ref{rem:prop:nu:not:exc}, there exist a smooth economical
refinement $\Sigma_i$ of $\Sigma_i'$. Since every cone of
$\Sigma(\sigma)$ is smooth, the fan induced by  
$\Sigma_i$ on $\sigma$ is $\Sigma(\sigma)$. 

Assume that $\Sigma_1$ is a smooth economical refinement of \(\cayley_{y_{1}}(\PD)\).
Let $\tau$ be a smooth face of \(\cayley_{y_{i}}(\PD)\)
Then, for any $j\in \{1,\dots,s\}$ such that $n_j\in \sigma$, one has
$n_j\notin \tau$; otherwise, $n_j$ would lie in the relative interior of a common
smooth face $\tau'$ of $\sigma$ and $\tau$, and $\tau'$ would not be a
cone of \(\Sigma_{1}\), contradicting the fact that $\Sigma_1$ is a
smooth economical refinement of \(\cayley_{y_{1}}(\PD)\). 
Thus, by the construction of $\Sigma'_i$, one has \(\tau\in \Sigma'_{i}\), thus \(\tau\in \Sigma_{i}\)
and $\Sigma_i$ is a smooth economical refinement of \(\cayley_{y_{i}}(\PD)\).
Likewise, if $\Sigma_1$ is a smooth $\deg(\cP)$-economical refinement of \(\cayley_{y_{1}}(\PD)\),
we may conclude that $\Sigma_i$ is a smooth $\deg(\cP)$-economical refinement of \(\cayley_{y_{i}}(\PD)\).

The remaining assertions are now consequences of lemma \ref{lemm:Pstrong:refin:bis}.

\end{proof}
\begin{theo}\label{theo:T:ess}
Let $Y$ be a smooth projective curve.
Let $\PD$ be a $p$-divisor over $(Y,N)$, $X:=X(\PD)$ and
$\widetilde{X}$ be the toroidification of $X$.
\begin{enumerate}
\item  
Let $\nu\in \DV(\widetilde{X})_{\TT}^{\sing}$ such that  \(\nu\notin \Min(\DV(\widetilde{X})_{\TT}^{\sing},\leqD)\).
Then there exists an equivariant divisorial resolution of
singularities of $\widetilde{X}$ such that $\nu$ is not exceptional
with respect to this resolution.
In particular, the set $\TT-\DivEss(\widetilde{X})$ of divisorially $\TT$-essential valuations on $\widetilde{X}$ is
contained in $\Min(\DV(\widetilde{X})_{\TT}^{\sing},\leqD)$
\item
Assume that $\Loc(\PD)=Y$ and $\genus(Y)>0$.
Let $\nu\in \DV(X)_{\TT}^{\sing}$ be a primitive element such that
\(\nu\notin \Min(\DV(X)_{\TT}^{\sing},\leqD)\).
Then there exists an equivariant divisorial resolution of
singularities of $\widetilde{X}$ such that $\nu$ is not exceptional
with respect to this resolution.
In particular, the set $\TT-\DivEss(X)$ of divisorially $\TT$-essential valuations on $X$ is
contained in $\Min(\DV(X)_{\TT}^{\sing},\leqD)$.
\end{enumerate}
\end{theo}
\begin{proof}
Let \(\nu\in \DV(\widetilde{X})_{\TT}^{\sing}\) be a primitive element 
such that \[\nu\notin \Min(\DV(\widetilde{X})_{\TT}^{\sing},\leqD).\]
Let $z\in Y$ such that \(\nu\in \cayley_{z}(\PD)=:\cC\).
Since \(\nu\notin \Min(\DV(\wt{X})_{\TT}^{\sing},\leqD)\), by
Remark \ref{rem:singloc:toroidal} and the very definition of
\(\leqD\), one has \(\nu\in \cC_{\sing}\cap \Hypz\setminus
\Min(\cC_{\sing}\cap \Hypz,\leq_{\cC})\).
By Proposition \ref{prop:nu:not:exc}, there exists a fan \(\Sigma\) which
is a big smooth economical star refinement of \(\cC\) such that $\dim(\Sigma(\nu))\geq 2$.
Let $\{y_1,\dots,y_r\}$ be a finite set of points of $\Loc(\PD)$
containing $z$ and $\Supp(\PD)$.
Using lemmas \ref{lemm:fan:extensions} and \ref{lemm:divisorial:refinement}, one may then construct an
equivariant divisorial resolution \(f\colon X(\DF)\to \wt{X}\) of the
singularities of $\wt{X}$ such that \(\nu\) is not \(f\)-exceptional.

Assume now that $\Loc(\PD)=Y$ and $\genus(Y)>0$, and
consider a primitive element \(\nu\in \DV(X)_{\TT}^{\sing}\)
such that \(\nu\notin \Min(\DV(\widetilde{X})_{\TT}^{\sing},\leqD)\).
Let $z\in Y$ such that \(\nu\in \cayley_{z}(\PD)=:\cC\).
Since \(\nu\notin \Min(\DV(\wt{X})_{\TT}^{\sing},\leqD)\), by Proposition
\ref{prop:singloc} and the very definition of
\(\leqD\), and using the notation of Proposition \ref{prop:gentor}
with $\cP:=\deg(\PD)$ and $N:=N\times \bZ$, one has
\((\cC^{\star}\cap \Hypz) \setminus \Min(\cC^{\star}\cap \Hypz,\leq_{\cC})\).
We now may conclude by applying proposition \ref{prop:gentor} and
lemmas \ref{lemm:fan:extensions} and \ref{lemm:divisorial:refinement}.
\end{proof}

\section{The Nash order for torus actions of complexity one}\label{sec:nash}
\subsection{The hypercombinatorial order on the equivariant valuations  of a $\TT$-variety of complexity one}
\label{subsec:hypercombin:order}
Let $X$ be a normal complexity one $\TT$-variety.
By Subsections \ref{subsec:arc-spaces-fat} and \ref{subsec:nash:order}, the set $\DV(X)$
carries two natural poset structures $\leqmds$ and $\leqX$ such that $\leqmds\imply\leqX$. In
addition, one can define on the set $\DV(X)_{\TT}$ of
$\TT$-equivariant divisorial valuations a third natural poset structure of combinatorial nature.
We use the notation and terminology introduced in Section \ref{sec:algtorac}.

\begin{definition}\label{defi:hyperorder}
Let $Y$ be a smooth algebraic curve and $\PD$ be a $p$-divisor over
  $(Y, N)$.  Define a poset structure $\leqD$ on $\hypercone(\PD)$ as
  follows: let $\nu_1$ and $\nu_2$
  be elements of $\hypercone(\PD)\cap \Hypz$. Then one has
  $\nu_1\leqD\nu_2$ if and only if there exists a page $\Page{y}$
  containing $\nu_1$ and $\nu_2$ and one has
  $\nu_2\in \nu_1+\cayley_{y}(\PD)$.

Now let $\E$ be a divisorial fan over $(Y, N)$ 
and $X:=X(\E)$ be the associated normal $\TT$-variety of complexity one. 
Define an order $\leqE$ on $\DV(X)_{\TT}$ as follows: let $\nu_1,\nu_2\in \DV(X)_{\TT}$.
Then $\nu_1\leqE \nu_2$ if and only if there exists a $p$-divisor
$\PD$ of $\E$ such that the centers of $\nu_1$ and $\nu_2$ lie on $X(\PD)$,
and, identifying $\DV(X(\PD))_{\TT}$ with $\hypercone(\PD)\cap \Hypz$, one has $\nu_1\leqD\nu_2$.
\end{definition}
\begin{rema}\label{rema:hyper:toric}
The restriction of $\leqD$ to any Cayley cone $\cayley_{y}(\PD)$ of
$\PD$ is the order $\leq_{\cayley_y(\PD)}$ (see Definition \ref{defi:leqs}).
\end{rema}
\begin{rema}
Identifying $\DV(X(\E))_{\TT}$ with $\cup_{\PD\in\E}\hypercone(\PD)\cap \Hypz$, the order $\leqE$ may thus also be
described as follows: let $\nu_1,\nu_2\in \DV(X)_{\TT}$; then
$\nu_1\leqE\nu_2$ if and only if  there exists $\PD\in \E$ such that
$\nu_1,\nu_2\in \hypercone(\PD)$ and $\nu_1\leqD\nu_2$
(see Proposition \ref{prop:divisorial:T:valuations}).

If $\nu_1,\nu_2$ are elements of $\DV(X)_{\TT}$ such that
$\nu_1\leqE\nu_2$, then for any $p$-divisor $\PD$ of $\E$ such that $X(\PD)$
contains the centers of $\nu_1$ and $\nu_2$, one has $\nu_1\leqD \nu_2$. 

In particular, in case $\E$ is the divisorial fan generated by a
single $p$-divisor $\PD$, on $\DV(X(\PD))_{\TT}$ one has $\leqD=\leqE$. 
\end{rema}

\begin{prop}\label{prop:comp:leqD:leqX}
Let $Y$ be a smooth algebraic curve, $\E$ be a divisorial fan over $(Y, N)$
and $X:=X(\E)$ be the associated $\TT$-variety of complexity one. 
\begin{enumerate}
\item 
Let $\nu_1,\nu_2\in \DV(X)_{\TT}$.
Then one has: $\nu_1\leqE \nu_2$ if and and only if $\nu_1\leqX\nu_2$
and $(\nu_1)_{|_Y}\leqY(\nu_2)_{|_Y}$.
In particular, one always has $\leqE \imply \leqX$. 
\item
Assume that the locus $\Loc(\PD)$ is affine for any $\PD\in \E$.
Then the three poset structures $\leqE$, $\leqmds$ and  $\leqX$ coincide on $\DV(X)_{\TT}$.
In particular, one has:
\[
\MinVal(X)=\Nash(X)=\Min(\DV(X)_{\TT}^{\sing},\leqE).
\]
\item 
In general, on $\DV(X)_{\TT}$, the Nash order is finer than the hypercombinatorial order, in other words one has
\[
\leqE \imply \leqmds.
\]
In particular, one always has the inclusion
\[
\Nash(X)\subset \Min(\DV(X)_{\TT}^{\sing},\leqE).
\]
\end{enumerate}
\end{prop}
\begin{rema}\label{rema:mini:sing}
Assume that $\E$ is toroidal or $Y$ is a smooth projective curve of
positive genus. Then using the description $\DV(X)_{\TT}^{\sing}$
deduced from Proposition \ref{prop:singloc} one sees
that any element $\Min(\DV(X)_{\TT}^{\sing},\leqE)$ lies either on the
spine or on some non-trivial Cayley cone associated with a polyhedral
divisor in $\E$. In particular $\Min(\DV(X)_{\TT}^{\sing},\leqE)$ is
contained in the reunion of a finite number of Cayley cones.
\end{rema}

\begin{proof}
First let us show the ``only if'' part of the first assertion.
By the very definitions of the involved poset structures, one may
assume that $X=X(\PD)$ is affine.
Let $\nu_1,\nu_2\in \DV(X)_{\TT}$ such that $\nu_1\leqD \nu_2$.
Recall that if $f=f_m\cdot \chi^m$ is a semi-invariant function and
$\nu=[y,n,\ell]\in \DV(X)_{\TT}$ then $\nu(f)=\ell \ord_y(f_m)+\acc{n}{m}$.
By the very definition of $\leqD$, there exists $y\in Y$ such that
$\nu_1=[y,n_1,\ell_1]$ and $\nu_2=[y,n_2,\ell_2]$ and
$(n_2,\ell_2)-(n_1,\ell_1)\in \cayley_y(\PD)$. In particular one has $\ell_1\leq \ell_2$.
Thus one has $\nu_1(f)\leqX\nu_2(f)$ for every semi-invariant element $f\in k[X]$, hence $\nu_1\leqX \nu_2$. Moreover, since
$(\nu_i)_{|_Y}=\ell_i\ord_y$ and $\ell_1\leq \ell_2$,
for every non-empty affine open subset $Y_0$ of the $Y$ containing
$y$,  $(\nu_1)_{|_Y}$ and $(\nu_2)_{|_Y}$ are centered at $Y_0$ and $(\nu_1)_{|_Y}\leq_{Y_0}(\nu_2)_{|_Y}$.
Thus $(\nu_1)_{|_Y}\leqY(\nu_2)_{|_Y}$.

Now assume $\nu_1,\nu_2\in \DV(X)_{\TT}$
satisfy $\nu_1\leqX \nu_2$ and $(\nu_1)_{|_Y}\leqY(\nu_2)_{|_Y}$.
Let $\PD\in \E$ such that $\cent_X(\nu_1)\in X(\PD)$.
Write $\nu_1=[y_1,\ell_1,n_1]$ and $\nu_2=[y_2,\ell_2,n_2]$
with $y_1,y_2\in \Loc(\PD)$, $[\ell_i,n_i]\in \cayley_{y_i}(\PD)\cap \Hypz$.

Since $(\nu_1)_{|_Y}\leqY(\nu_2)_{|_Y}$, there exists a non-empty affine open subset $Y_0$ of the $Y$
such that $(\nu_1)_{|_Y}$ and $(\nu_2)_{|_Y}$ are centered at $Y_0$
and $(\nu_1)_{|_Y}\leq_{Y_0}(\nu_2)_{|_Y}$. Since
$(\nu_i)_{|_Y}=\ell_i\ord_{y_i}$, the latter condition implies that
$\ell_1=0$ or $y_1=y_2$ and $\ell_1\leq \ell_2$. In particular one may
assume that $y_1=y_2=:y$.

Now write $k[X(\PD)]=k[f_i\cdot \chi^{m_i}]_{i\in I}$ where $f_i\cdot
\chi^{m_i}$ are semi-invariant regular functions and $I$ is finite. Then for every
$i\in I$, one has $\nu_1(f_i)\leq \nu_2(f_i)$, thus
\[
(\ell_2-\ell_1)\ord_{y}(f_i)+\acc{m_i}{n_2-n_1}\geq 0.
\]
This and the above condition $\ell_1\leq \ell_2$ exactly says that $[y,\ell_2-\ell_1,n_2-n_1]\in \cayley_{y}(\PD)$.
Thus $\nu_1\leqD\nu_2$. This completes the proof of the first assertion.

Assume that the locus $\Loc(\PD)$ is affine for any $\PD\in \E$ and
let us show the second assertion. In this case, for any $\PD\in \E$, $k[\Loc(\PD)]$
is a subring of $k[X(\PD)]$. Thus for any $\nu_1,\nu_2\in
\DV(X)_{\TT}$ such that $\nu_1\leqX\nu_2$, one has $(\nu_1)_{|_Y}\leqY(\nu_2)_{|_Y}$.
By the first assertion, one has $\leqE=\leqX$.

Thus by Proposition \ref{prop:leqmds:implies:leqX}, it remains to show
that that $\leqE \imply \leqmds$. Let $\nu_1,\nu_2\in \DV(X)_{\TT}$ such that
$\nu_1\leqE \nu_2$ and let us show that $\nu_1\leqmds\nu_2$. By the
definition of $\leqE$ and Remark \ref{rema:nash:order:local}, one may
assume that $X=X(\PD)$ is affine and $\nu_1\leqD\nu_2$. There exist a closed point $y$ of
$\Loc(\PD)$ and elements $(a_1,b_1),(a_2,b_2)\in \cayley_y(\PD)\cap \Hypz$
such that $\nu_i=\val_{[y,a_i,b_i]}$ and $(a_2,b_2)\in (a_1,b_1)+\cayley_y(\PD)$.
Since $\Loc(\PD)$ is affine, for any open affine subset
$Y_0$ of $\Loc(D)$ containing $y$, $X(\PD_{|Y_0})$ is an open affine subset of $X$ containing the centers
of $\nu_1$ and $\nu_2$ (by Proposition \ref{prop:divisorial:T:valuations}).
By Remark \ref{rema:nash:order:local} it suffices to show that there exists
an open affine subset $Y_0$ of $\Loc(D)$ containing $y$
such that $\mds_{X(\PD_{|Y_0})}(\nu_1)\subset \mds_{X(\PD_{|Y_0})}(\nu_2)$.

By Lemma \ref{lemma-toroidal} and Proposition \ref{proposition-toroidal-valuations},
one thus may assume that
$\Supp(\PD)\subset \{y\}$ and there is a $\TT$-equivariant étale morphism
$\theta\colon X\to Z$ where $Z$ is the toric $\bG_m\times \TT$-variety associated
with $(\cayley_y(\PD),N)$ and $\arc(\theta)$ maps $\eta_{X,\nu_i}$ to
$\eta_{Z,\mu_i}$ where $\mu_i$ is the toric valuation associated with $(a_i,b_i)$.
Since $(a_2,b_2)\in (a_1,b_1)+\cayley_y(\PD)$, one has
$\nu_1\leq_{\cayley_y(\PD)}\nu_2$, thus by Proposition
\ref{prop:mds:combin:toric}, $\eta_{Z,\mu_2}$ is a specialization of $\eta_{Z,\mu_1}$.

By Corollary \ref{coro:csl} and Theorem \ref{theo:stable}, there
exist an extension $K$ of $k$ and a finite sequence
$w_1,\dots,w_r$ of $K$-wedges on $Z$ such that, denoting by $\alpha_i$
(resp. $\beta_i$) the generic arc (resp. the special arc) of $w_i$,
one has $\beta_{i+1}=\alpha_{i}$ for any $1\leq i\leq r-1$,
$\alpha_1=\eta_{Z,\mu_1}$ and $\beta_r=\eta_{Z,\mu_2}$.

By Lemma \ref{lemma-etale-lifting-wedges},
and Proposition \ref{proposition-toroidal-valuations},
upon extending $K$, there exists a $K$-wedge $\wt{w_r}$ on $X$ lifting $w_r$
with special arc $\wt{\beta_r}=\eta_{X,\nu_2}$ and generic arc $\wt{\alpha_r}$. In particular $\arc(\theta)(\wt{\alpha_r})=\alpha_r$.
Applying Lemma \ref{lemma-etale-lifting-wedges} again,
upon extending $K$, there exists a $K$-wedge $\wt{w_{r-1}}$ on $X$
lifting $w_{r-1}$ with special arc $\wt{\beta_{r-1}}=\wt{\alpha_r}$ and generic arc $\wt{\alpha_{r-1}}$.
Continuing in this way, one ends up with the following: upon
extending $K$, there exist a finite sequence
$\wt{w_1},\dots,\wt{w_r}$ of $K$-wedges on $X$ such that, denoting by $\wt{\alpha_i}$
(resp. $\wt{\beta_i}$) the generic arc (resp. the special arc) of $\wt{w_i}$,
one has $\wt{\beta_{i+1}}=\wt{\alpha_{i}}$ for any $1\leq i\leq r-1$,
$\wt{\beta_r}=\eta_{X,\nu_2}$ and $\arc(\theta)(\wt{\alpha_1})=\eta_{Z,\mu_1}$.
Thus $\eta_{X,\nu_2}$ is a specialization of $\wt{\alpha_1}$.
Since $\arc(\theta)(\wt{\alpha_1})=\eta_{Z,\mu_1}$, still by
Proposition \ref{proposition-toroidal-valuations}, one has
$\wt{\alpha_1}\in \arc(X)^{\ord=\nu_1}$. Thus $\eta_{X,\nu_2}$ is a
specialization of an element of $\arc(X)^{\ord=\nu_1}$. Therefore
one has $\mds_{X}(\nu_2)\subset \mds_{X}(\nu_1)$, as was to be shown.

It remains to show that $\leqE\,\imply\, \leqmds$ holds in general.
But on the toroidification, the hypercombinatorial order coincide with
the mds order. Now we may apply Proposition
\ref{prop:proper:bir:ord:mds} on order to conclude.
\end{proof}
\begin{rema}
We will show later that when the locus $Y$ is a smooth projective
curve of positive genus, then one also has $\leqmds=\leqD$.

However, in case $Y$ is the projective line,
on the set $\DV(X)_{\TT}$, $\leqX$ is in general strictly finer than $\leqmds$, and $\leqmds$
is in general strictly finer than $\leqD$, see Section \ref{sec:casePone}.
\end{rema}
\subsection{Lifting wedges to the toroidification}\label{subsec:lifting:wedges}
Let $X$ be a $\TT$-variety of complexity one and locus a projective
curve of positive genus. We show that any wedge on $X$ not contained in the
singular locus lifts to the toroidification. This is a consequence of
the following more general result.
\begin{proposition}\label{prop:lifting:wedges}
Let $X$ and $X'$ be algebraic $k$-varieties. Assume that there exist a proper
birational morphism $\pi\colon X'\to X$ and an affine morphism $q\colon X'\to Y$ 
from $X'$ to a smooth projective algebraic $k$-curve $Y$ with positive
genus. Let $U$ be a non-empty open subset of $X$ such that $\pi$
induces an isomorphism over $U$.

Let $K$ be an extension of $k$ and $w\colon
\Spec(K\dbr{t,u})\to X$ be a $K$-wedge on $X$, whose image is not
contained in $X\setminus U$. Then, upon replacing $w$ by the induced
$L$-wedge $\Spec(L\dbr{t,u})\to X$, where $L/K$ is an algebraic extension,
the wedge $w$ lifts to
$X'$, that is, there exists a morphism $w'\colon \Spec(K\dbr{t,u})\to X'$ such that $\pi\circ w'=w$.
\end{proposition}
\begin{rema}
The existence of a lifting upon replacing $K$ by an extension
is sufficient for our needs. That being said, this restriction in the
conclusion was put for the sake of convenience, since basically the same
argument as below shows that a lifting exists even without extending $K$.
\end{rema}
\begin{proof}
Set $S:=\Spec(K\dbr{t,u})$. Upon replacing $K$ by an algebraic
extension, one may assume that $K$ is algebraically closed.
Extending the scalars from $k$ to $K$, one obtains the following
commutative diagram
\[
  \xymatrix{
&
&Y_K
\\
&
X'_K
\ar[ur]_{q_K}
\ar[d]_{\pi_K}
\ar[r]
&
X'
\ar[d]_{\pi}
\\
S
\ar[r]_{w_K}
\ar@/_2pc/[rr]_{w}
&
X_K\ar[r]
&
X
}
\]
Thus it suffices to show that $w_K$ lifts to $X'_K$.
Note that $\pi_K\colon X'_K\to X_K$ is a proper birational morphism
inducing an isomorphism over $U_K$, 
$q_K$ is an affine morphism $X'_K\to Y_K$, $Y_K$ is a smooth
projective $K$-curve such that $\genus(Y_K)\geq 1$ and the image of $S$ by
$w_K$ meets $U_K$. Thus one may assume that $K=k$.

Since $S$ is a noetherian two-dimensional
regular scheme, $\pi$ is an isomorphism over $U$, the image of $S$ by $w$ meets $U$, and
$\pi\colon X'\to X$ is proper, by \cite[Theorem, p 45]{Sha68}, there exist a scheme
$S'$, a morphism $w'\colon S'\to X'$, and a morphism
$\phi\colon S'\to S$ such that $\phi$ is a finite composition
of blow-ups at a maximal ideal and the following diagram is commutative:
\[
  \xymatrix{
S'
\ar[d]_{\phi}
\ar[r]^{w'}
&
X'
\ar[d]_{\pi}
\\
S
\ar[r]_{w}
&
X
}
\]  
The exceptional locus $E$ of $\phi$ is connected, and since $k$ is
algebraically closed, each of its irreducible
component is isomorphic to $\bP^1_k$.
Since $Y$ has positive genus, the morphism $\bP^1_k\to Y$ induced by the composition of
$w'$ with $q\colon X'\to Y$ is constant.
Since $q$ is an affine morphism, there exists an open affine 
subset $V$ of $X'$ such that $w'^{-1}(V)$ contains $E$. In particular
the image by $\phi$ of the closed subset $S'\setminus w'^{-1}(U)$ does not contain the closed point of $S$. Thus
$w'^{-1}(V)=S'$.

Now since $\phi$ is proper birational and $S$ is normal, one has
\(\phi_*\str{S'} = \str{S}\).
In particular one has a factorization $\phi=\phi_1\circ\phi_2$ where $\phi_2\colon S'\to \Spec(H^0(S',\str{S'}))$
is the natural morphism and $\phi_1\colon \Spec(H^0(S',\str{S'}))\to S$ is an isomorphism.
But since $V$ is affine, the morphism $w'\colon S'\to V\subset
X'$ factors as $w'=\psi\circ \phi_2$ with $\psi\colon
\Spec(H^0(S',\str{S'}))\to V$. Thus $\psi\circ \phi_1^{-1}$
gives the sought-for lifting of $w$ to $X'$.
\end{proof}

\subsection{The Nash order in case the locus is a projective curve of
  positive genus}\label{subsec:nashorder:posgen}
As a consequence of the previous sections, we obtain a solution of the
generalized Nash problem for normal complexity one $\TT$-variety with
locus a projective curve of positive genus, in the sense that the Nash
order on the set of equivariant valuations is explicitly described by
the hypercombinatorial poset structure of Definition \ref{defi:hyperorder}
\begin{theo}\label{theo:nash:order:Yggeq1}
Let $Y$ be a smooth projective algebraic curve, $\E$ be a divisorial fan over $(Y, N)$
and $X:=X(\E)$ be the associated $\TT$-variety of complexity one.   
Assume that $\genus(Y)\geq 1$. Then on $\DV(X)_{\TT}$ one has
$\leqmds=\leqE$.
\end{theo}
\begin{proof}
By Proposition \ref{prop:comp:leqD:leqX} it suffices to show that
on $\DV(X)_{\TT}$ one has $\leqmds\imply \leqE$.
Let $\nu_1,\nu_2$ be elements of $\DV(X)_{\TT}$ such that
$\mds_X(\nu_2)\subset \mds_X(\nu_1)$. Let $\PD$ 
be a $p$-divisor of $\E$ such that $X(\PD)$
contains the centers of $\nu_1$ and $\nu_2$ (see Remark \ref{rema:mds:local}).
One has to show that $\nu_1\leqD\nu_2$.
By Remark \ref{rema:nash:order:local}, one has
$\mds_{X(\PD)}(\nu_2)\subset \mds_{X(\PD)}(\nu_1)$.
One may assume that $X=X(\PD)$. In case $\Loc(\PD)$ is affine, the
result is given by Proposition \ref{prop:comp:leqD:leqX}.
From now on we assume $\Loc(\PD)=Y$. We consider the toroidification
$\pi\colon \wt{X}\to X$ (see Subsection \ref{subsec:toroidification}).
It is a $\TT$-equivariant proper birational morphism. In particular we may identify
$\DV(X)_{\TT}$ and $\DV(\wt{X})_{\TT}$, the poset structures defined
by $\leqD$ and $\leqtD$ coincide, and $\arc(\pi)$ induces a continuous
bijection $\arc(\wt{X})\to \arc(X)$.

Since $\mds_{X(\PD)}(\nu_2)\subset \mds_{X(\PD)}(\nu_1)$, by 
Theorem \ref{theo:stable} and Corollary \ref{coro:csl},
there exists an extension $K/k$ and a finite sequence of $K$-wedges $w_1,\dots,w_r$
on $X$ such that the special arc of $w_1$ is $\eta_{X,\nu_1}$, the generic arc of
 $w_r$ is $\eta_{X,\nu_2}$ and for any $1\leq i\leq r-1$ the generic arc of
 $w_i$ is the special arc of $w_{i+1}$. Since $\eta_{X,\nu_1}$ and
 $\eta_{X,\nu_2}$ are fat, the generic arc of any of the $w_i$'s is
 fat. In particular, the image in $X$ of any of the $w_i$'s is not
 contained in any proper closed subset of $X$. On the other hand,
 $\wt{X}$ is equipped with an affine morphism $q\colon \wt{X}\to Y$
 (Lemma \ref{lemm:qaffine}).

Thus one may apply Proposition \ref{prop:lifting:wedges}; upon
extending $K$, one obtain $K$-wedges $\wt{w_1},\dots,\wt{w_r}$ on
$\wt{X}$ that lift $w_1,\dots,w_r$.
Since  $\arc(\pi)\colon \arc(\wt{X})\to \arc(X)$ is a bijection mapping
$\eta_{\wt{X},\nu_i}$ to $\eta_{X,\nu_i}$ ($i=1,2$), 
the special arc of $\wt{w_1}$ is $\eta_{\wt{X},\nu_1}$, the generic arc of
 $\wt{w_r}$ is $\eta_{\wt{X},\nu_2}$ and for any $1\leq i\leq r-1$ the generic arc of
 $\wt{w_i}$ is the special arc of $\wt{w_{i+1}}$. In  particular 
$\eta_{\wt{X},\nu_2}$ is a specialization of $\eta_{\wt{X},\nu_2}$,
which shows that $\mds_{\wt{X},\nu_2}\subset  \mds_{\wt{X},\nu_1}$. By
Proposition \ref{prop:comp:leqD:leqX}, one has $\nu_1\leqtD\nu_2$,
thus $\nu_1\leqD\nu_2$, as was to be shown.
\end{proof}

\subsection{The classical Nash problem in case the locus is a projective curve of
  positive genus}
\begin{theo}\label{theo:nash:problem:Yggeq1}
Let $Y$ be a smooth projective algebraic curve, $\E$ be a divisorial fan over $(Y, N)$
and $X:=X(\E)$ be the associated $\TT$-variety of complexity one.   
Assume that either $\genus(Y)\geq 1$ or for every $\PD\in \E$ the
locus $\Loc(\PD)$ is affine. Then the Nash problem has a positive
answer, \ie the inclusion $\Nash(X)\subset \Ess(X)$ is an equality.
\end{theo}
\begin{proof}
By Remark \ref{rema:mds:local}, one may assume that $X$ is affine and defined by a $p$-divisor
$\PD$ whose locus is either affine or a smooth projective algebraic
curve of positive genus. By Remark \ref{rema:nash:equiv}, one has $\Nash(X)=\Min(\DV(X)_{\TT}^{\sing},\leqmds)$.
By Proposition \ref{prop:comp:leqD:leqX} (in case $\Loc(\PD)$ is
affine) and Theorem \ref{theo:nash:order:Yggeq1} (in case $\Loc(\PD)$
is projective with positive genus),
on $\DV(X)_{\TT}$ one has $\leqmds=\leqD$. Thus 
$\Nash(X)=\Min(\DV(X)_{\TT}^{\sing},\leqD)$. On the other hand, by \ref{theo:T:ess},
one has $\TT-\DivEss(X)\subset \Min(\DV(X)_{\TT}^{\sing},\leqD)$.
Since the inclusions $\Ess(X)\subset \TT-\DivEss(X)$ and $\Nash(X)\subset
\Ess(X)$ always hold (see Section \ref{sec:essential-valuations} and
Proposition \ref{prop:Nash:Ess}), one ends up with the conclusion that $\Nash(X)=\Ess(X)$.
\end{proof}
\begin{rema}\label{rema:nash:toroidal}
A natural general question is whether the bijectivity of the Nash map is
invariant by surjective étale morphisms.
We thank Shihoko Ishii for pointing out the following. Let $f\colon
X\to Y$ be an étale morphism between algebraic variety; let $\nu$ be a
divisorial valuation on $X$ and $\mu$ be the valuation induced by $f$ on $Y$.
Then one can show that if $\nu$ is essential, $\mu$ is also essential, and if $\mu$ is
Nash, $\nu$ is also Nash. In particular, if the Nash map is bijective
for $Y$, it is also bijective for $X$ (Assuming that $f$ is surjective, it is not clear whether the
converse is true). Thus in case $\Loc(\PD)$ is affine,
the result of Theorem \ref{theo:nash:problem:Yggeq1} is also a
consequence of Ishii-Kollar's result.
\end{rema}
\begin{rema}
In case $\Loc(\PD)$ is affine, Proposition \ref{prop:comp:leqD:leqX}
also gives that the inclusion $\MinVal(X)\subset \Nash(X)$ is an equality.
In case $\Loc(\PD)$ is projective of positive genus, this is no longer
true in general (see Subsection \ref{subsec:nash:nonterm:nonmin} below
for an example).
\end{rema}
\begin{rema}
Similarly to the toric case, we obtain as a direct consequence of the
above argument and Section \ref{sec:essential-valuations} that for the
varieties under consideration any divisorially essential valuation is
an essential valuation.
\end{rema}
\newpage
\section{Terminal valuations and torus actions of complexity one}
\label{sec:term}
The aim of this section is to give a combinatorial description of the
terminal valuations of a $\TT$-variety of complexity one with locus a
smooth projective curve of positive genus. One
ingredient is the description of terminal toric valuations in
\cite[\S 6]{MR3437873}, and our description (see Theorem
\ref{theo:terval:cplxone}) has analogies with Proposition 6.2 of \opcit.

The main result of \opcit\ is that on any algebraic variety, any terminal valuation is a Nash
valuation. An example is given of a toric Nash valuation which is
not terminal, and it is pointed out that any toric Nash valuation is
minimal, and that no example of a Nash valuation which is neither
terminal nor minimal was known.
In the next section, we will use Theorem \ref{theo:terval:cplxone}
and the results of the previous sections to exhibit examples of non-toroidal $\TT$-varieties of
complexity one possessing a Nash valuation which is neither terminal
nor minimal. 

To conclude the introduction of this section, let us point out that
the assumption on the genus
in Theorem \ref{theo:terval:cplxone}
is crucial, and that the result fails in the
genus zero case (see Section \ref{sec:casePone}). As for the description of the Nash order and of the
essential valuations, the description of terminal valuations seems much
more challenging in this case.

\subsection{Preliminary results}\label{SectionPrelimanaryTerminalVal}
Let $\E$ be a divisorial fan over $(Y, N)$, where $Y$ is a smooth
projective curve.
As before, write $\TT$ for the algebraic torus with
one-parameter lattice $N$.
\begin{definition}
The prime $\TT$-divisors on the $\TT$-variety $X  =  X(\E)$ are divided into two sorts.
\begin{itemize}
\item[$(i)$] The ones whose restriction of the vanishing order to $\bC(X)^{\TT}\simeq \bC(Y)$ is non-trivial;
\item[$(ii)$] and the others. 
\end{itemize}
The divisors of type $(i)$ are called the \emph{vertical divisors}
while the ones of type $(ii)$ are called the \emph{horizontal divisors}.
\end{definition}
In the sequel, we set 
\[\deg(\E) :=  \bigcup_{\PD\in \E}\deg(\D)\]
and let $\Sigma  :=  \Sigma(\E)$ be the fan generated by the tail cones of the coefficients of the elements of $\E$. Note 
that $X  =  X(\E)$ is toroidal (i.e isomorphic to its toroidification) if and only if $\deg(\E) =  \vide$. We also set 
\[\Ray(\E) :=  \{ \rho \text{ rays of } \Sigma\,|\, \rho \cap \deg(\E) =  \vide\}\text{ and }\]
\[\Ver(\E) :=  \{ (y, v)\in Y\times N_{\bQ}\,|\, v\text{ vertex of }  \PD_{y}\text{ for some } \PD\in\E\}.\]

For $v\in N_{\bQ}$, denote by $\mu(v)$  the number $\inf\{d\in\bZ_{>0}\,|\, dv\in N\}$.
Note that associating an element $(y,v)\in
\Ver(\E)$ with the ray $\rho_{y,v}$ generated in $\Page{y}$ by \([y,v,1]\) defines
a bijection between $\Ver(\E)$ and the rays of $\Hfan(\E)$ not
contained in the spine, and that \([y,\mu(v).v,\mu(v)]\) is a
primitive generator of $\rho_{y,v}$.

For any ray $\rho$ of $\Sigma$, denote by $v_{\rho}$ its primitive generator.
From \S \ref{sec:prime-invar-cycl}, one deduces:
\begin{proposition}
  \begin{itemize}
\item[$(i)$]  With an element $\rho$ of $\Ray(\E)$ one associates
    $\cent_{X(\E)}(\val_{[\bullet,v_{\rho},0]})$. This defines a bijection between the set $\Ray(\E)$ and the set of horizontal prime $\TT$-divisors on $X(\E)$.
\item[$(ii)$]
    With an element $(y,v)$ of $\Ver(\E)$ one associates
    $\cent_{X(\E)}(\val_{[y,\mu(v).v,\mu(v)]})$. This defines a bijection between the set $\Ver(\E)$ and the set of vertical prime $\TT$-divisors on $X(\E)$.
  \end{itemize}

\end{proposition}
\begin{remark}\label{RemarkPrincipalDivisorHomogeneous}
For $\rho\in \Ray(\E)$ (respectively $(y, v)\in \Ver(\E)$) we denote by $D_{\rho}$ (respectively $D_{(y,v)}$) the corresponding 
divisors. Their vanishing orders can be explicitly described as
follows.
Let $f\in \bC(Y)\setminus \{0\}$ and let $m\in M$. Then $\xi =  f\otimes \chi^{m}$ is a
homogeneous element of the function field $\bC(X(\E))$ and
by \S \ref{subsec:hcones:hfans}
we have the formulae
\[\ord_{D_{\rho}}(\xi) =  \acc{m}{v_{\rho}}\text{ and } \ord_{D_{(y,v)}}(\xi) =  \mu(v)(\ord_{y}(f) + \acc{m}{v}),\]
Therefore the principal divisor associated with $\xi$ is given by the relation
\[\div(\xi) =  \sum_{\rho\in \Ray(\E)}\acc{m}{v_{\rho}}\cdot D_{\rho} + \sum_{(y, v)\in \Ver(\E)}\mu(v)(\ord_{y}(f) + \acc{m}{v})\cdot D_{y,v}.\]
\end{remark}
For the next proposition, see \cite[Theorem 3.21]{MR2783981}.
\begin{proposition}\label{PropositionCanonicalClassTvar}
Let $K_{Y} :=  \sum_{y\in Y} K_{Y, y}\cdot [y]$ be a representative of the canonical class of $Y$. 
Then the canonical class of the $\TT$-variety $X  =  X(\E)$ is represented by the Weil divisor
\[K_{X} =  \sum_{(y, v)\in \Ver(\E)} (\mu(v) K_{Y, y} + \mu(v)-1)\cdot D_{(y, v)} - \sum_{\rho\in \Ray(\E)} D_{\rho}.\]
\end{proposition}
We say that a line bundle on an algebraic variety is {\em semiample}
if a positive power of it is basepoint-free.
\begin{lemma}\label{LemmaPullbackSemiAmple}
Let $\varphi\colon S\rightarrow B$ be a morphism between algebraic varieties. If $\cL$ is a semiample line bundle on $B$,
then the line bundle $\varphi^{\star}\cL$ is semiample.
\end{lemma}
\begin{proof}
Note that the semiample condition on $\cL$ is equivalent to the
existence of a morphism $\psi\colon B\rightarrow B_{0}$ and an ample
line bundle $\cL_{0}$ on $B_{0}$  such that $\cL =
\varphi^{\star}(\cL_{0})$.
So $\varphi^{\star}(\cL) =  \varphi^{\star}(\psi^{\star}(\cL_{0})) =  (\psi \circ \varphi)^{\star}(\cL_{0})$
is  semiample. 
\end{proof}
\begin{remark}
Let $\varphi\colon S\rightarrow B$ be a dominant morphism between
algebraic varieties and $D$ a
Cartier divisor on $B$. Let $(U_{i}, f_{i})_{i\in I}$, where
$U_{i}\subset B$ is a dense open subset and $f_{i}\in k(B)^{\star}$, be 
local data representing $D$. 
We recall that the pullback $\varphi^{\star}(D)$
is defined as the Cartier divisor represented by the local data $(\varphi^{-1}(U_{i}), f_{i}\circ \varphi)$.
\end{remark}
\begin{notation}\label{nota:kxplus:kxminus}
  With the same notation as Proposition \ref{PropositionCanonicalClassTvar} we set 
  \[\Gamma(\E) :=  \supp(\E)\cup \supp(K_{Y}),\,\, \Ver^{+}(\E) :=  \{(y,v)\in \Ver(\E)\,|\, y\in \Gamma(\E)\},\]
  \[K_{X}^{+} :=  \sum_{(y, v)\in \Ver^{+}(\E)}\mu(v)( K_{Y, y} + 1)\cdot D_{(y, v)},\]
  \[\text{ and }\quad K_{X}^{-} :=  -\sum_{(y, v)\in \Ver^{+}(\E) }D_{(y, v)}-\sum_{\rho\in\Ray(\E)}D_{\rho}.\]
  Note that we have $K_{X} =  K_{X}^{+} + K_{X}^{-}$. 
\end{notation}
\begin{lemma}\label{CanonicalPlusSemiAmple}
Assume that $X  =  X(\E)$ is toroidal and that the genus of $Y$ is
positive. Then the divisor $K^{+}_{X}$ is  Cartier and semiample. 
\end{lemma}
\begin{proof}
Let $q\colon X\rightarrow Y$ be the $\TT$-invariant dominant morphism induced by the global quotient of $X$.
Let $E=\sum_{y\in Y}a_{y}\cdot [y]$ be any divisor on $Y$.
Given a local equation $\alpha$ of $E$ on some open
  subset $U$ of $Y$, $q^{\star}\alpha$ is a local equation of $q^{\star}(E)$
  and by Remark \ref{RemarkPrincipalDivisorHomogeneous} one has
\[\div(q^{\star}(\alpha)) = \sumsubu{(y,v)\in \Ver(\E)\\y\in \Supp(\E)\cup \Supp(E)}\mu(v)\ord_{y}(\alpha)\cdot D_{(y, v)}.\]
Thus
\[
q^{\star}E=\sumsubu{(y,v)\in \Ver(\E)\\y\in \Supp(\E)\cup \Supp(E)}\mu(v)a_y\cdot D_{y,v}.
\]
Now consider the divisor $E := K_{Y} +  \sum_{y\in \Gamma(\E)}[y]$.
The above shows that $q^{\star}(E) =  K_{X}^{+}$.
On the other hand the assumption on the genus of $Y$ implies that $E$ is principal or $\deg(E)>0$. So $E$ is a Cartier semiample divisor on $Y$. 
By Lemma \ref{LemmaPullbackSemiAmple}, $K^{+}_{X}$ is semiample. 
\end{proof}
\begin{definition}\label{DefinitionNefLineBundle}
Let $V$ be a variety and let $\cL$ be a line bundle over $V$. We say
that $\cL$ is \emph{nef} if for any (irreducible) complete curve $C$
inside $V$ the intersection $(\cL, C)$ is non-negative. Recall that
$(\cL, C)$ is defined as the degree of any associated divisor of
$\kappa^{\star}\cL$, where $\kappa\colon\widetilde{C}\rightarrow V$ is
the composition of the normalization $\widetilde{C}\rightarrow C$ and
the inclusion $C \rightarrow V$. We define similarly the intersection
number and the nef condition for Cartier $\bQ$-divisors.
\end{definition}

\begin{lemma}\label{LemmaSemiAmpleImpliesNef}
Let $V$ be a variety and let $\cL$ be a  line bundle on $V$. If $\cL$ is semiample, then $\cL$ is nef. 
\end{lemma}
\begin{proof}
Taking the notation of Definition \ref{DefinitionNefLineBundle}, for any complete curve $C$ on $V$ the pullback $\kappa^{\star}(\cL)$ is semiample by Lemma \ref{LemmaPullbackSemiAmple}; so $(\cL, C)\geq 0$. 
\end{proof}

\begin{lemma}\label{ToricLemmaSemiAmple}
Let $\E$ be toroidal divisorial fan with tail fan $\Sigma$
  and let $D =  \sum_{\rho\in \Ray(\E)}a_{\rho} D_{\rho}$ be a horizontal invariant $\bQ$-divisor
on $X(\E)$.

Then $D$ is Cartier if and only if there exists a map
$\theta\colon|\Sigma|\rightarrow \bQ$ linear on each cone of $\Sigma$
such that $\theta(v_{\rho}) = a_{\rho}$ for any $\rho\in
\Ray(\E)$. Moreover, $D$ is semiample if and only if the map $\theta$ is convex.
\end{lemma}
\begin{proof}
The Cartier characterization comes from \cite[Theorem 5]{Tim00} and the formula for the principal divisor associated with a homogeneous element (see Remark \ref{RemarkPrincipalDivisorHomogeneous}), while the semiample characterization follows from \cite[Theorem 6]{Tim00}. 
\end{proof}

\subsection{Minimal models and terminal valuations}
In this section, we let $\PD$ be a proper $\sigma$-tailed polyhedral
divisor over a smooth projective curve $Y$ of genus $\geq 1$. We write
$X  =  X(\PD)$ for the associated $\TT$-variety.

Our goal is to describe the \emph{terminal valuations} of $X$.
Let us first recall from \cite{MR3437873} the general definition of this notion.
For any algebraic variety $X$, a \emph{relative minimal model over $X$} is a projective birational
morphism $f\colon Z\to X$ such that $Z$ has terminal singularities
(in particular $Z$ is normal and its canonical divisor is
\(\bQ\)-Cartier) and the canonical divisor of $Z$ is relatively nef
over $X$, that is to say, has positive intersection with any complete
curve on $Z$ (see Definition \ref{DefinitionNefLineBundle}) which is
contracted by the morphism $f$.
The terminal valuations of $X$ are the valuations defined by
the codimension one irreducible components of
any relative minimal model $f\colon Z\to X$.

\begin{definition}\label{TerToricDefinition}
Let $\tau\subset N_{\bQ}$ be a strictly convex polyhedral
cone. Recall that $\tau$ is \emph{smooth} if it is generated by a
subset of a basis of $N$ and that $\tau_{\sing}$ designates the union of
the relative interiors of the non-smooth faces of $\tau$. 
Let $\Gamma(\tau)$ be the convex hull of $\tau\cap N\setminus\{0\}$ in
$N_{\bQ}$ and let $\partial_{c}\Gamma(\tau)$ be the union of the
bounded faces of $\Gamma(\tau)$. Following \cite[Section 6]{MR3437873}, we define the set of \emph{terminal
  points} of $\tau$ as the subset
\[\Ter(\tau) =  \tau_{\sing} \cap \partial_{c}\Gamma(\tau) \cap N.\]
\end{definition}
\begin{remark}
For any face $\tau'$ of $\tau$, one has
$\partial_{c}\Gamma(\tau)\cap \tau'=\partial_{c}\Gamma(\tau')$.
Moreover, if $\tau$ is smooth, $\partial_{c}\Gamma(\tau) \cap N$ is the set of
primitive generators of the rays of $\tau$.
\end{remark}
\begin{remark}\label{RemarkTerminalVal}
Using star subdivisions, one can build a triangulation of the
polyhedral complex $\partial_{c}\Gamma(\tau)$
with set of vertices equal to $\partial_{c}\Gamma(\tau) \cap N$.
In this way we define a fan subdivision of $\tau$ whose set of primitive
generators of rays is exactly $\partial_{c}\Gamma(\tau) \cap N$, and
whose set of primitive generators of rays which are not rays of $\tau$
is exactly $\Ter(\tau)$.
This construction can be generalized for the polyhedral divisor $\PD$,
using triangulations on the polyhedral complexes $\partial_{c}\cayley_y(\PD)$
that are compatible on $\partial_{c}\sigma$, where $\sigma$ is the tail.
Again, this may be achieved by suitable star subdivisions.

The toroidal divisorial fan induced by these subdivisions (see Lemma \ref{lemm:divisorial:refinement}) will be denoted by $\E$.
\end{remark}
\begin{theorem}\label{TheoremTvarMinimalModel}
Let $\PD$ be a proper polyhedral divisor over a smooth complete curve of positive genus.
Let $\E$ be the divisorial fan of Remark \ref{RemarkTerminalVal} and
$X'=X(\E)$. Then the induced proper birational morphism $p\colon X'\to
X$ is a relative minimal model over $X$.
\end{theorem}
\begin{definition}
  Keep the notation in the statement of Theorem \ref{TheoremTvarMinimalModel}.
Let $\beta\colon X'\to Y$ be the quotient map.
A complete curve $C$ on $X'$ is called:
\begin{itemize}
\item  \emph{vertical} if $C$ is
  contained in a fiber of $\beta$;
\item \emph{horizontal} otherwise, that is, if the restriction map $\beta_{|C}\colon C\rightarrow Y$ is dominant.
\end{itemize}
\end{definition}
\begin{rema}\label{rema:inter:vert:horiz}
By the very definitions, the intersection number of a $\bQ$-Cartier vertical
prime divisor with an horizontal curve is positive.
\end{rema}
\begin{proof}(of Theorem \ref{TheoremTvarMinimalModel})
Since the Cayley cones of $\E$ are terminal and $X'= X(\E)$ is toroidal,
$X'$ has terminal singularities. In order to prove that $X'$ is a minimal
model over $X$, we have to show that
$(K_{X'}, C)\geq 0$ for any complete curve $C$ on $X'$ contracted by
$p$ (note that the latter condition holds in fact for any complete
curve on $X'$ since $X$ is affine).

First we deal with the case of horizontal curves.  
Consider the decomposition $K_{X'}^{-} =  V + H$, where the divisors $V$ and $H$ are the vertical and horizontal
parts of $K_{X}^{-}$ respectively, \ie
\[
V:=-\sum_{(y, v)\in \Ver^{+}(\E) }D_{(y, v)},\quad H:=-\sum_{\rho\in\Ray(\E)}D_{\rho}.
\]
The canonical class of the toric variety associated with the tail fan
$\Sigma$ of $\E$ is semiample according to \cite[Section 6]{MR3437873}. So it
corresponds to a convex piecewise linear map $\theta\colon
|\Sigma|\rightarrow \bQ$. Using Lemma \ref{ToricLemmaSemiAmple} we
see that $H$ is semiample. Let $C$ be an horizontal curve on $X'$.
By Lemma \ref{LemmaSemiAmpleImpliesNef},
in order to show that $(K_{X'},C)\geq 0$, it 
suffices to show that $(K_{X'}^++V, C)\geq 0$.

As the genus of $Y$ is positive, the canonical class of $Y$ is
semiample. Thus for some $r\in \bZ_{>0}$, $r\cdot K_Y$ is linearly
equivalent to an effective divisor $\sum_{y\in Y} a_y\cdot [y]$.
By remark \ref{RemarkPrincipalDivisorHomogeneous}, $rK_{X'}^{+} + rV$ is linearly equivalent to 
\[
\sum_{(y, v)\in \Ver(\E)}(\mu(v)a_y + r\mu(v) - r)\cdot D_{(y,v)}
\]
Thus, using remark \ref{rema:inter:vert:horiz}, 
one has
\[(rK_{X'}^{+} + rV, C) = \sum_{(y, v)\in \Ver(\E)}(\mu(v)a_y + r\mu(v) - r)\cdot (D_{(y,v)}, C)\geq 0\]
That concludes the proof for horizontal curves.

Now consider a vertical curve $C$ on $X'$, contained in a fiber $\beta^{-1}(\{y\})$, and let us show that
$(K_{X'},C)\geq 0$. By Lemmas \ref{LemmaSemiAmpleImpliesNef} and
\ref{CanonicalPlusSemiAmple}, one has $(K_{X'}^+,C)\geq 0$, thus
it suffices to show that $(K_{X'}^-,C)\geq 0$.

Let $(U, \varphi)$ be an \'etale
  chart of $Y$ around the point $y$, \ie the map
  $\varphi\colon U \rightarrow \bA^{1}_{k}$ is an \'etale morphism
  such that $\varphi^{-1}(0) = \{y\}$. Shrinking $U$ if necessary we
  may assume that $U\cap \Supp(\PD')\subset \{y\}$ for any
  $\PD'\in \E$, and that $U\cap \Supp(K_Y)\subset \{y\}$.
Set $\E_{|U} =  \{ \PD_{|U}'\,|\, \PD'\in \E\}$. Note that the curve $C$ is
contained in $W:=X(\E_{|U})$, and that it suffices to show that
$(K_W^{-},C)\geq 0$.
  
Denote by $\E_{\varphi}$ the divisorial fan over $(\bA^{1}_{k}, N)$
generated by the $p$-divisors
\[\{\PD'_y\cdot [0]+  \sum_{z\in \varphi(U)\setminus \{0\}} \Tail(\PD')\cdot [z] \,|\, \PD'\in \E\}.
\]
From Lemma \ref{lemma-toroidal} we have a $\TT$-equivariant isomorphism $X(\E_{|U})\simeq U\times_{\bA^{1}}X(\E_{\varphi})$. Let
\[\gamma\colon W= X(\E_{|U})\rightarrow X(\E_{\varphi})\]
be the corresponding \'etale morphism.
Denote by $\bar{\E}_{\varphi}$ the divisorial fan over $(\bA^{1}_{k}, N)$
  generated by the $p$-divisors
\[\PD'_y\cdot [0]+  \sum_{z\in \bA^1_k\setminus \{0\}} \Tail(\PD')\cdot [z],\,|\, \PD'\in \E\}.
\]
Observe that there is a natural open immersion of $X(\E_{\varphi})$
into $V:=X(\bar{\E}_{\varphi})$ and that 
$V$ is the toric $\bG_{m}\times \TT$-variety associated with the fan
generated by by $\{\cayley_{y}(\PD')\,|\, \PD'\in \E\}$.
Note that using Notation \ref{nota:kxplus:kxminus} for $V:=X(\bar{\E}_{\varphi})$,
$K_{V}^{-}$ is the canonical class of $V$.
Thus by \cite[Section 6]{MR3437873} the divisor $K_{V}^{-}$ is nef.
Let $r\in \bZ_{>0}$ such that $rK_{V}^{-}$ is Cartier.
Since $\varphi^{-1}(0) =  \{y\}$
and using the explicit description of $\gamma$ (see lemma \ref{lemma-toroidal})
we see that $\gamma^{\star}(rK_{V}^{-}) = rK_{W}^{-}$.
Furthermore, denoting by $\alpha\colon V\rightarrow \bA^{1}_{k}$ the quotient map for the $\TT$-action,
the restriction map
\[\beta^{-1}(\{y\})\rightarrow \alpha^{-1}(\{0\}),\,\, x\mapsto \gamma(x)\] is an isomorphism.
Denoting by $C'$ the image of $C$ by this isomorphism,
  we thus have $(C',K_V^-)=(C,K_W^-)$ hence $(C,K_W^-)\geq 0$ since
  $K_V^{-}$ is nef.
\end{proof}
\begin{definition}
We generalize Definition \ref{TerToricDefinition} in the setting of
hypercones. Let $\PD$ be a $p$-divisor with locus a smooth projective
curve $Y$ of genus $\geq 1$. We set 
\[\hypercone(\PD)_{\sing} =  \bigcup_{ \theta\in \hypercone(\PD)_{\sing}^{\star}}\theta\]
(see Definition \ref{DefinitionSingularLocus} for the notation $\hypercone(\PD)_{\sing}^{\star}$) and $\partial_{c}\Gamma(\PD)$ stands for 
\[\partial_{c}\Gamma(\PD) := \left(\,\,\bigcup_{y\in Y}\{y\}\times
    \partial_{c}\Gamma(\cayley_y(\PD))\,\,\right)/\sim.\]
\end{definition}
\begin{theorem}\label{theo:terval:cplxone}
Let $\PD$ be a $p$-divisor over a smooth projective curve $Y$ of genus $\geq 1$. Then the set of terminal valuations 
of $X  =  X(\PD)$ is given by the formula
\[\Ter(\PD)  =  \hypercone(\PD)_{\sing}  \cap \partial_{c}\Gamma(\PD)\cap \Hypz\subset \Hyp.\]
\end{theorem}
\begin{proof}
This directly follows from Theorem \ref{TheoremTvarMinimalModel}. 
and \S \ref{sec:except-locus-toro}.
\end{proof}
  \begin{remark}
In  other words, assuming that $\Supp(\PD)$ is non-empty (for example
assuming that $\Loc(\PD)=Y$),
the set $\Ter(\PD)$ is the reunion of the sets
    $\Ter(\cayley_y(\PD))$ where $y$ runs over $\Supp(\PD)$ and the set
    of primitive generators of the rays of the tail $\sigma$ which meet $\deg(\PD)$.
In case $\Supp(\PD)$ is empty, $\Ter(\PD)=\Ter(\sigma)$ 
  \end{remark}

\section{Some examples}\label{sec:exam}
In this section we describe some examples of non-toroidal non-rational
$\TT$-varieties of complexity one constructed from polyhedral divisors
and illustrating our results.
\subsection{} In \cite[Section 6]{MR3437873}, an example of a Nash
toric valuation which is not terminal is given.
Here we construct a simple family of $\TT$-varieties of complexity one of arbitrarily large
dimension with a similar property. Note that the construction does
not rely on the aforementioned toric example.
Let $d$ be an integer such that $d\geq 2$,
$N=\bZ^d$ and \(\sigma\) be the $N$-smooth cone generated by the
canonical $\bZ$-basis of $N$.
Let \(Y\) be a smooth projective curve with positive genus and $y_0,\dots,y_r\in Y$.
be a finite set of points of $Y$. For any $1\leq i\leq r$, choose
$n_i\in \Relint(\sigma)$, and set $\PD_{y_i}:=n_i+\sigma$.
For $y\in Y\setminus \{y_0,\dots,y_r\}$, set $\PD_{y}:=\sigma$.
Thus $\deg(\PD)=\sum_{i=0}^rn_i+\sigma$ is contained $\Relint(\sigma)$
Therefore $\PD$ is a proper polyhedral divisor over $(Y,N)$ with locus $Y$.
By Theorem \ref{theo:nash:order:Yggeq1}, $\Nash(X(\PD))=\{[\bullet,(1,\dots,1),0]\}$. Note that
$\wt{X}(\PD)$ is smooth, and that no ray of $\hcone(\PD)$ meets $\deg(\PD)$.
Thus by Theorem \ref{theo:terval:cplxone}, $X(\PD)$ has no terminal
valuation. Hence $[\bullet,(1,\dots,1),0]$ is a Nash valuation which
is not terminal. Similar examples with more Nash valuations may be
obtained; take \eg a subset $\Delta$ of the set of faces of $\sigma$
containing no ray and such that no element of $\Delta$ is a face of
another element of $\Delta$, and replace each $n_i$ by the convex hull
of arbitrary elements of the relative interiors of each face of $\Delta$.
The resulting $X(\PD)$ has a number of Nash valuations equal to the
cardinality of $\Delta$, and no one of them is terminal.

\subsection{}\label{subsec:nash:tildex:not:x}
We now give example where the toroidification $\widetilde{X}(\PD)$ has a Nash valuation which is
not a Nash valuation on $X(\PD)$ (see Remark \ref{rema:nash:tildex:not:x}).

Take $N=\bZ^2$ and let \(\sigma\) be the $N$-smooth cone in \(\bQ^2\) spanned by \(e_1=(1, 0)\) \(e_2=(0, 1)\).
Let \(Y\) be a smooth projective curve with positive genus, and
\(y_1,y_0\in Y\) such that the divisors \(2\cdot y_1\) and \(2\cdot y_0\).
are linearly equivalent. For example, take $Y$ with \(\genus(Y)=1\),
\(y_0\in Y\) and \(y_1\) a 2-torsion point of the elliptic curve \((Y,y_0)\).

Set \(\PD_{y_0}:=[(0,0),(1,-\frac{1}{2})]+\sigma\)
and \(\PD_{y_1}=\{(\frac{1}{2},\frac{1}{2})\}+\sigma\)
For \(y\in Y\setminus \{y_0,y_1\}\), set \(\PD_y=\sigma\).
Thus 
\[\deg(\PD{})=[(\frac{1}{2},\frac{1}{2}),(\frac{3}{2},0)]+\sigma\]

For \(m=a\cdot e_0^{\vee}+b\cdot e_1^{\vee}\in \sigma^{\vee}\), 
\[\acc{m}{\PD_{y_0}}=-\frac{b}{2}\quad\text{and}\quad\acc{m}{\PD_{y_1}}=\frac{a}{2}+\frac{b}{2}\]
thus \(\deg(\PD(m))=\frac{a}{2}\)
and therefore 
\[\deg(\PD(m))=0\iff a=0 \iff \PD(m)=\frac{b}{2}(y_1-y_0).\]
Since the divisor \(\frac{b}{2}(y_1-y_0)\) is \(\bQ\)-principal, we
infer that \(\PD\) is a proper polyhedral divisor. 

The non-trivial Cayley cones are 
\(\cayley_{y_1}(\PD)=\ide{[\bullet,(1,0),0],[\bullet,(0,1),0],[y_1,(1,1),2]}\) which is smooth
and \(\cayley_{y_0}(\PD)=\ide{[\bullet,(1,0),0],[\bullet,(0,1),0],[y_0,(0,0),1],[y_0,(2,-1),1]}\)
which is not smooth but whose any proper face is smooth.

Thus
\(\DV(\wt{X})_{\TT}^{\sing}=\Relint(\cayley_{y_0}(\PD))\cap \Hypz\)
and one may check that
\[
\Nash(\wt{X})=\Min(\DV(\wt{X})_{\TT}^{\sing},\leqD)=[y_0,(1,0),1]=:\nu_1.
\]
On the other hand, since the only faces of \(\sigma\) meeting the degree
are \(\sigma\) and the ray generated by \((1,0)\),
\[\DV(X(\PD))_{\TT}^{\sing}
=(\Relint(\cayley_{y_0}(\PD))\cup \Relint(\bQ_{\geq 0}[\bullet,(1,0),0])\cup \Relint(\sigma))
\cap \Hypz
\]
and one may check that
\[
\Nash(X)=\Min(\DV(X)_{\TT}^{\sing},\leqD)=[\bullet,(1,0),0]=:\nu_2.
\]
In particular, \(\nu_1\) is a Nash valuation of \(\wt{X}\) which is not
a Nash valuation of \(X\). 

One may construct as follows a smooth $\deg(\PD)$-big and
$\deg(\PD)$-economical star refinement of $\cayley_{y_0}(\PD)$. For
the sake of simplicity, one now identifies $\cN_{y_0}$ with $\bQ^3$, thus 
\[\cayley_{y_0}(\PD)=\ide{(1,0,0),(0,1,0),(0,0,1),(2,-1,1)}.\]
Then the star subdivision $\Sigma$ of $\cayley_{y_0}(\PD)$ with respect to $(0,0,1)$
is a smooth fan and its maximal cones are  \(\ide{(1,0,0),(0,0,1),(2,-1,1)}\)
and \(\ide{(1,0,0),(0,1,0),(0,0,1)}\), meeting along the
face \(\ide{(1,0,0),(0,0,1)}\), which contains \(\nu_1\) in its relative
interior. One easily sees that $\Sigma$ is a $\deg(\PD)$-big and
$\deg(\PD)$-economical star refinement of $\cayley_{y_0}(\PD)$, and as
in section $\Sigma$ allows to define a divisorial equivariant resolution of
$f\colon Z\to X$ such that $\nu_1$ is not $f$-exceptional. Note that
$f$ factors through $g\colon Z\to \wt{X}$, but $\nu_1$ is
$g$-exceptional, and the corresponding exceptional component is a
curve, contained in $\cent_Z(\nu_2)$ which is a component
of codimension $1$ of the exceptional locus of $f$ and but is not
contained in the exceptional locus of $g$.

\subsection{}\label{subsec:nash:nonterm:nonmin}
We now describe a family of examples where $X(\PD)$ has a Nash
valuation which is neither terminal nor minimal.
To the best of our knowledge, until now, no such example was known for
any algebraic variety (see the discussion in the introduction).

Let $d$ be an integer such that $d\geq 2$,
$N=\bZ^d$ and \(\sigma\) be the $N$-smooth cone generated by the
canonical $\bZ$-basis $(\rho_i)_{1\leq i\leq d}$ of $N$.
Let \(Y\) be a smooth projective curve with genus \(\geq 1\) and \(y_0\in Y\).
Set \(\PD_{y_0}:=\{(\frac{1}{2},\dots,\frac{1}{2})\}+\sigma\)
and \(\PD_{y}=\sigma\) for \(y\neq y_0\).
Thus \(\deg(\PD{})=\PD_{y_0}\subset \Relint(\sigma)\)
and \(\PD\) is a proper polyhedral divisor. 

The cone \(\cayley_{y_0}(\PD)\) is the unique
non-trivial Cayley cone. For the sake of simplicity, identify
$\Page{y_0}$ with $N_{\bQ}\times \bQ$ and set $\rho_{d+1}:=(0_N,1)$.
Then \(\cayley_{y_0}(\PD)\) is the simplicial cone generated by
$\{\rho_i\}_{1\leq i\leq d}\cup \{\sum_{i=1}^d \rho_i+2\rho_{d+1}\}$.
It is not smooth but any of its proper faces is. Its dual cone is the cone
generated by $\{\rho_{d+1}^{\vee}\}\cup \{2\rho_i^{\vee}-\rho^{\vee}_{d+1}\}_{1\leq i\leq d}$.
From this, one easily checks that the set of minimal
elements of $\cayley_{y_0}(\PD)_{\sing}\cap (N\times \bZ)$ for the
combinatorial order is \(\{\nu_1:=\sum_{i=1}^{d+1}\rho_i\}\).
On the other hand, the only face of \(\sigma\) meeting the degree is
\(\sigma\) itself. Set $\nu_0:=\sum_{i=1}^{d+1}\rho_i$. One checks that
$\nu_1-\nu_0\notin \cayley_{y_0}(\PD)$. Thus, using our combinatorial
interpretation of the Nash order, one sees that
\[
\Nash(X)=\{\nu_0,\nu_1\}.
\]
Since \(\cayley_{y_0}(\PD)\) is simplicial and has no ray which meets
the degree, by Theorem \ref{theo:terval:cplxone} \(X\) has no terminal valuation.

Now take \(f=\sum_{m\in M\cap \sigma^{\vee}} f_m\cdot \chi^m\)
a global regular function on $X$.
For any \(m=\sum_{i=1}^d m_i\cdot \rho_i^{\vee}\in M\cap \sigma^{\vee}\), one has
\[
\nu_0(f_m\chi^m)=\sum_{i=1}^dm_i\quad\text{and}\quad \nu_1(f_m\chi^m)=\ord_{y_0}f_m+\sum_{i=1}^dm_i.
\]
But since \(f_m\in H^0(Y,\floor{\frac{1}{2}(\sum_{i=1}^dm_i)}\cdot [y_0])\), one
must have \(\ord_{y_0}f_m\leq 0\).
Thus \(\nu_1(f)\leq \nu_0(f)\). We infer that the set of minimal
valuations of \(X\) is reduced to \(\nu_1\), and that \(\nu_0\) is a Nash
valuation on $X$ which is neither terminal nor minimal.

\section{The case where the locus is the projective line}\label{sec:casePone}
Let $Y$ be a smooth algebraic curve, $\E$ be a divisorial fan over $(Y, N)$
and $X:=X(\E)$ be the associated $\TT$-variety of complexity
one. Assume that $\Loc(\PD)$ is affine for any $\PD\in \E$ or $Y$ is projective
with positive genus. Then any equivariant resolution of $X$ factors through the
toroidification (Proposition \ref{prop:equivres}). Moreover the poset
structures on $\DV(X)_{\TT}$ defined by the hypercombinatorial order $\leqD$
and the Nash order $\leqmds$ are the same (Proposition \ref{prop:comp:leqD:leqX}
and Theorem \ref{theo:nash:order:Yggeq1}), and one has $\Nash(X)=\Ess(X)$
(Proposition \ref{prop:comp:leqD:leqX} and Theorem \ref{theo:nash:problem:Yggeq1}).
On the other hand, if $X:=X(\PD)$ is defined by a polyhedral
divisorial with locus $Y\simeq \bP^1$, one can see that no one
of the above properties hold in general. Even in case $X$ is a surface, 
though $\Nash(X)=\Ess(X)$ always hold, the first two properties may
fail, as the following example shows.
\begin{exam}\label{exam:twodim:genuszero}
Consider the affine surfaces $X:=\{x_0^3+x_1^4+x_2^5=0\}$.
Let $\sigma:=\bQ_{\geq 0}$ and consider the following $\sigma$-tailed
polyhedrons: $\PD_0:=[5/3,+\infty[$,
$\PD_1:=[-5/4,+\infty[$ and $\PD_{\infty}:=[-2/5,+\infty[$.
For $z\in \bP^1\setminus \{0,1,\infty\}$ set $\PD_z=\sigma$.
Then $\PD:=\sum_{z\in \bP^1}\PD_z\cdot \{z\}$
is a $p$-divisors with locus $\bP^1$, and $X$
endowed the $\bG_m$-action $\alpha\cdot (x_0,x_1,x_2)=(\alpha^{20}x_0,\alpha^{15}x_1,\alpha^{12}x_2)$ is isomorphic to
$X(\PD)$ (see \cite{MR4057986} for more general computations of polyhedral
divisors defining affine trinomial hypersurfaces)
One easily checks that
\[
\Min(\DV(X)_{\TT}^{\sing},\leqD)=\{[\bullet,1,0],[0,2,1],[1,-1,1],[\infty,-1,3],[\infty,0,1]\}.
\]
But on the toroidification, the divisor corresponding to $[\bullet,1,0]$ is a $(-1)$-curve. Contracting
this curve, one obtains the minimal resolution of $X$, which is an equivariant resolution of $X$ 
which does not factor through the toroidification.
Using the description of the minimal
resolution in terms of a divisorial fan and the fact that
the Nash problem holds for surfaces, one deduces that
\[
\Nash(X)=\Min(\DV(X)_{\TT}^{\sing},\leqD)\setminus \{[\bullet,1,0]\}
\]
which shows that $\leqmds$ is strictly finer than $\leqD$.
On the other hand, using \eg Proposition \ref{prop:pointwise}, one has 
$\Min(\DV(X)_{\TT}^{\sing},\leqX)=\{[\infty,-1,3]\}$, showing that
$\leqX$ is strictly finer than $\leqmds$.  
\end{exam}
\begin{rema}
One may also easily construct examples for which $\leqmds$ is strictly
finer than $\leqD$ in any dimension, using toric downgradings and the
fact that the Nash problem holds for toric varieties.
\end{rema}
\begin{exam}[Johnson-Kollar's threefold]\label{exam:johnson:kollar}
Let $\sigma$ be the cone of $\bQ^2$ generated by
$(1,0)$ and $(1,10)$. Let $\PD_0=[(1,0),(1,1)]+\sigma$,
$\PD_1=\left\{\left(-\frac{2}{5},0\right)\right\}+\sigma$ and
$\PD_{\infty}=\left\{\left(-\frac{1}{2},0\right)\right\}+\sigma$.
For any $z\in \bP^1\setminus \{0,1,\infty\}$, set $\PD_z=\sigma$.
The polyhedral divisor $\sum_{z\in \bP^1}\cD_z\otimes \{z\}$ 
has locus $\bP^1$ and degree $[(\frac{1}{10},0),(\frac{1}{10},1)]+\sigma$. 
It is thus a $p$-divisor. One can check that the associated
$\bG_m^2$-variety $X(\PD)$ is isomorphic to the affine threefold
$X=\{x_0x_1=x_2^2+x_3^5\}$, endowed with the restriction of the action of
$\bG_m^2$ on $\bA^4$ given for any $(\alpha,\beta)\in \bG_m^2$ by
\[
(\alpha,\beta)\cdot x_0=\beta\,x_0,
(\alpha,\beta)\cdot x_1=\alpha^{10}\beta^{-1}\,x_1,
(\alpha,\beta)\cdot x_2=\alpha^2\,x_2,
(\alpha,\beta)\cdot x_3=\alpha^5\,x_3,
\]
The threefold $X$ was considered in \cite{MR3094648} (without the
$\bG_m^2$-structure) where it was shown that it was a counter-example
to the Nash problem. More precisely, translating the results into
the hypercombinatorial description of the structure of $\bG_m^2$-variety on
$X$, the authors of \opcit showed that $X$ has exactly one Nash valuation, namely
$[1,(-1,1),3]$, and exactly one essential valuation which is not Nash,
namely $[1,(0,2),1]$.
\end{exam}
The previous example is a dramatic illustration of the fact that when
the locus is the projective line, the Nash order is finer than the
hypercombinatorial order. In this example, one can check that $\Min(\DV(X)^{\sing}_{\TT},\leqD)$
has 17 elements whereas by Johnson-Kollar's
result $\Min(\DV(X)^{\sing}_{\TT},\leqmds)$ is reduced to a singleton.
It seems an interesting yet difficult challenge to determine whether the Nash order
on the equivariant valuations of a $\TT$-variety of complexity one with locus $\bP^1$ has a sensible
interpretation in terms of the hypercombinatorial description of the
variety. Another question, perhaps more tractable, is to obtain a
sensible hypercombinatorial description of the sets $\TT-\DivEss(X)$
or $\TT-\Ess(X)$.

As a first step towards a combinatorial interpretation of the
Nash order, let us explain how one can use a construction of Ilten and
Manon in \cite{MR3978437} in order to get an interpretation of the pointwise order. Since one has
$\leqX\imply \leqmds\imply \leqD$, this gives us two approximations,
one by excess and the other by default, of the sought-for
interpretation of the Nash order.

First recall the construction of \opcit.
Let $\PD$ be a $p$-divisor over $(\bP^1,N)$ with locus $\bP^1$, tail $\sigma$ and
support \(\{y_0,\dots,y_s\}\subset \bP^1\)
Recall that for \(m\in M=N^{\vee}\) and a polyhedron $\cP$ one denotes \(\Min_{n\in \cP}\acc{m}{n}\) by \(\cP(m)\).
Ilten and Manon define an embedding of \(X=X(\PD)\) in a toric variety \(Z\) of
dimension \(\rk(N)+s\) as follows:
let \((e_i)\) be the canonical basis of \(\bZ^s\) and
\(\cC_X\) be the polyhedral cone in \((\bZ^s\oplus N)_{\bQ}\) defined as the convex hull of 
\[
\cT:=(\bQ_{\geq 0}(-\sum_{i=1}^se_i)\times \cayley_z(\PD_{y_0}))\cup \cup_{i=1}^{s}(\bQ_{\geq 0}e_i\times \cayley_z(\PD_{y_i}))
\]
Note that \(((v_i),m)\in (\bZ^s\oplus M)_{\bQ}\) lies in
\(\cC_X^{\vee}\) if and only if
\[
\forall m\in \sigma^{\vee},\,1\leq i\leq s,\, v_i+\PD_{y_i}(m)\geq 0, \quad  \quad \sum_{i=1}^s v_i\leq \PD_{y_0}(m), \quad 
\]
Note also that \(\cT\cap (\bZ^s\times N)\) is naturally identified with a subset \(\DV(X)^{\star}_{\TT}\)
of the set \(\DV(X)_{\TT}\) of equivariant divisorial valuation of \(X\). More precisely 
\(\DV(X)^{\star}_{\TT}\) consists of those  \(\nu\in \DV(X)_{\TT}\) which may be represented as
 \([y_i,\ell,n]\) with \(0\leq i\leq s\) and \((\ell,n)\in \cayley_{y_i}(\PD)\). 
Let \(\iota\colon \DV(X)^{\star}_{\TT}\to \cC_X\) be the natural embedding.
Let \(Z\) be the toric variety associated with \(\cC_X\).

The following gives a combinatorial interpretation of the pointwise
order on the set \(\DV(X)^{\star}_{\TT}\). In case \(s=1\), \(X=Z\) is toric, and the
result is already known by \ref{prop:mds:combin:toric}.
Note that by \ref{prop:comp:leqD:leqX} and \ref{rema:mini:sing}, 
$\DV(X)^{\star}_{\TT}$ contains $\Nash(X)$.

\begin{prop}\label{prop:pointwise}
Let \(\nu,\nu'\in \DV(X)^{\star}_{\TT}\). Then \(\nu\leq_X\nu'\) if
and only if \(\iota(\nu')\in \iota(\nu)+\cC_X\). In other words,
identifying $\DV(X)^{\star}_{\TT}$  with its image by $\iota$, 
the pointwise order on $\DV(X)^{\star}_{\TT}$ is the restriction of
the order $\leq_{_{\cC_X}}$ (see Definition \ref{defi:leqs}).
\end{prop}
\begin{proof}
Set \(k[\bZ^s]=k[z_i^{\pm 1}]_{1\leq i\leq s}\) and
\(k(\bP^1)=k(z)\). Without loss of generality, one may assume that \(\infty\notin \{y_i\}_{0\leq i\leq s}\).
The embedding of \(X\) into \(Z\) is given at the level of regular functions
by the surjective morphism of \(k\)-algebras \(\phi\colon k[Z]=k[\cC_X^{\vee}\cap (\bZ^s\oplus M)]\to k[X]\) which maps \(z^{v}\chi^{m}\)
to \(\prod_{i=1}^s (\frac{z-y_i}{z-y_0})^{v_i}\chi^m\).

Let \(z^v\chi^m\) be a monomial in \(k[Z]\) and \(\nu\in \DV(X)^{\star}_{\TT}\).
First note that
\[
\nu(\phi(z^v\chi^m))=\acc{(v,m)}{\iota(\nu)}
\]
Indeed, assume that \(\nu=[y_i,\ell,n]\) with \(1\leq i\leq s\).
Then 
\[
\nu(\phi(z^v\chi^m))=\ell\ord_{y_i}\left(\prod_{j=1}^s (\frac{z-y_j}{z-y_0})^{v_j}\right)+\acc{m}{n}
=\ell v_i+\acc{m}{n}=\acc{(v,m)}{\iota(\nu)}
\]
the last equality being a consequence of the definition of \(\iota\).

Now if \(\nu=[y_0,\ell,n]\) one has
\[
\nu(\phi(z^v\chi^m))=\ell\ord_{y_0}\left(\prod_{i=1}^s (\frac{z-y_i}{z-y_0})^{v_i}\right)+\acc{m}{n}
\]
\[
=-\ell(\sum_{i=1}^n v_i)+\acc{m}{n}=\acc{(v,m)}{\iota(\nu)},
\]
the last equality being again a consequence of the definition of \(\iota\).

Now consider  \(\nu,\nu'\in \DV(X)^{\star}_{\TT}\)
such that \(\nu \leq_X \nu'\). In particular for every \(z^{v}\chi^m\) in \(k[Z]\) one has
\(\nu(\phi(z^{v}\chi^m))\leq \nu'(\phi(z^{v}\chi^m))\), hence \(\acc{(v,m)}{\iota(\nu)}\leq \acc{(v,m)}{\iota(\nu')}\). 
This shows \(\iota(\nu')\in \iota(\nu)+\cC_X\)

Assume \(\iota(\nu')\in \iota(\nu)+\cC_X\). In particular for every \(z^{v}\chi^m\) in \(k[Z]\) one has
\(\nu(\phi(z^{v}\chi^m))\leq \nu'(\phi(z^{v}\chi^m))\). In order to shows \(\nu \leq_X \nu'\), 
it suffices to show that for every semi-homogeneous element \(g=f\cdot
\chi^{m}\in k[X]\), one has \(\nu(g)\leq nu'(g)\) thus it is enough to show that there exists \(z^{v}\chi^m\in k[Z]\) such that 
\(\nu(\phi(z^{v}\chi^m))=\nu(g)\) and \(\nu'(\phi(z^{v}\chi^m))=\nu'(g)\). 

The following is inspired by the proof of Theorem 5.4 in \cite{MR3978437}.
Let \(i,j\in \{0,\dots,s\}\) such that \(\nu=[y_i,\ell,n]\) and \(\nu'=[y_j,\ell',n']\)
First assume \(i,j\geq 1\). For \(1\leq k\leq n\) let \(v_k=\ord_{y_k}f\).
In particular \(\nu(g)=\ell v_i+\acc{m}{n}\) and \(\nu'(g)=\ell' v_j+\acc{m}{n'}\).
Since \(f\cdot \chi^{m}\in k[X]\), one has \(m\in \sigma^{\vee}\) and
\[
v_i+\Delta_i(m)\geq 0,\quad 0\leq i\leq s,\quad \ord_{y}(f)\geq 0,\quad y\notin\{y_0,\dots,y_n\}
\]
And since \(\sum_{y\in \bP^1}\ord_y(f)=\sum_{i=0}^sv_i+\sum_{y\notin\{y_0,\dots,y_n\}}\ord_{y}(f)=0\), one has \(\sum_{i=1}^n v_i\leq \PD_{y_0}(m)\).
In particular \(z^v\chi^m\in k[Z]\) and clearly \(\nu(\phi(z^{v}\chi^m))=\nu(g)\) and \(\nu'(\phi(z^{v}\chi^m))=\nu'(g)\). 

In case \(i=0\), first notice that \(-w=v_0+\sum_{i=1}^sv_i\) is nonnegative. Pick \(k\notin \{i,j\}\) (one may assume \(s\geq 2\))
and let \(v'\) be defined by \(v'_{r}=v_r\), \(r\neq k\) and \(v'_k=v_k+w\). Then one still has
\[
v'_i+\Delta_i(m)\geq 0,\quad 0\leq i\leq s,\quad \sum_{i=1}^n v'_i=-v_0\leq \PD_{y_0}(m)
\]
thus \(z^{v'}\chi^m\in k[Z]\) and the same argument works.
\end{proof}

\end{document}